\documentclass[a4paper,11pt]{amsart} %
\usepackage[utf8x]{inputenc}

\usepackage{amsmath}
\usepackage{amsfonts}
\usepackage{amssymb}
\usepackage{mathrsfs}
\usepackage[all]{xy}

\numberwithin{equation}{section}

\allowdisplaybreaks[3]

\theoremstyle{plain}
\newtheorem{theorem}{Theorem}[section]
\newtheorem{lemma}[theorem]{Lemma}
\newtheorem{corollary}[theorem]{Corollary}
\newtheorem{proposition}[theorem]{Proposition}

\newtheorem{definition}[theorem]{Definition}
\newtheorem{remark}[theorem]{Remark}

\newtheorem{conjecture}[theorem]{Conjecture}

\newtheorem{fact}[theorem]{Fact}

\def\GL{\operatorname{GL}}

\def\OO{\operatorname{O}}

\def\SL{\operatorname{SL}}
\def\SO{\operatorname{SO}}

\def\SU{\operatorname{SU}}
\def\U{\operatorname{U}}

\def\ad{\operatorname{ad}}
\def\Ad{\operatorname{Ad}}

\def\d{\operatorname{d\!}}

\def\det{\operatorname{det}}

\def\diag{\operatorname{diag}}

\def\End{\operatorname{End}}
\def\exp{\operatorname{exp}}

\def\Hom{\operatorname{Hom}}

\def\Im{\operatorname{Im}}
\def\Ind{\operatorname{Ind}}

\def\Ker{\operatorname{Ker}}

\def\Lie{\operatorname{Lie}}
\def\max{\operatorname{max}}

\def\Pf{\operatorname{Pf}}

\def\pr{\operatorname{pr}}

\def\Re{\operatorname{Re}}

\def\sgn{\operatorname{sgn}}

\def\sm{\operatorname{sm}}

\def\Spin{\operatorname{Spin}}

\def\Stab{\operatorname{Stab}}

\def\tr{\operatorname{tr}}

\newcommand{\bbC}{\mathbb{C}}

\newcommand{\bbR}{\mathbb{R}}
\newcommand{\bbZ}{\mathbb{Z}}

\begin{document}

\title[Restriction of representations to parabolic subgroups]{Restriction of irreducible unitary representations of Spin(N,1) to parabolic subgroups}
\author{Gang Liu}
\address{Gang Liu, Institut Elie Cartan de Lorraine, CNRS-UMR 7502, Universit\'e de Lorraine,  3 rue Augustin Fresnel, 57045 Metz, France}
\email{gang.liu@univ-lorraine.fr}
\author{Yoshiki Oshima}
\address{Yoshiki Oshima, Department of Pure and Applied Mathematics, Graduate School of Information Science and
Technology, Osaka University, 1-5 Yamadaoka, Suita, Osaka 565-0871, Japan.}
\email{oshima@ist.osaka-u.ac.jp}
\author{Jun Yu}
\address{Jun Yu, Beijing International Center for Mathematical Research, Peking University, No. 5 Yiheyuan Road,
Beijing 100871, China}
\email{junyu@bicmr.pku.edu.cn}
\begin{abstract}In this paper, we obtain explicit branching laws for all irreducible unitary representations of
$\Spin(N,1)$ restricted to a parabolic subgroup $P$. The restriction turns out to be a finite direct sum of
irreducible unitary representations of $P$. We also verify Duflo's conjecture for the branching law of tempered
representations of $\Spin(N,1)$ with respect to a parabolic subgroup $P$. That is to show: in the framework of
the orbit method, the branching law of a tempered representation is determined by the behavior of the moment map
from the corresponding coadjoint orbit. A few key tools used in this work include: Fourier transform, Knapp-Stein
intertwining operator, Casselman-Wallach globalization, Zuckerman translation principle, du Cloux's results for
smooth representations of semi-algebraic groups.
\end{abstract}

\maketitle

\setcounter{tocdepth}{1}

\tableofcontents

\section{Introduction}\label{S:introduction}

The unitary dual problem concerning the classification of irreducible unitary representations of a Lie group and
the branching law problem concerning the decomposition of the restriction of irreducible unitary representations
to a closed Lie subgroup are two of the most important problems in the representation theory of real Lie groups.
The orbit method of Kirillov (\cite{Kirillov2}, \cite{Kirillov}) and Kostant (\cite{Auslander-Kostant}, \cite{Kostant})
relates both problems to the geometry of coadjoint orbits.

In a series of seminal papers \cite{Kobayashi}, \cite{Kobayashi2}, \cite{Kobayashi3} Kobayashi initiated the study of
discrete decomposability and admissibility for representations when restricted to non-compact subgroups. Let $G$ be a
Lie group and let $H$ be a closed Lie subgroup. For an irreducible unitary representation $\pi$ of $G$, the restriction of
$\pi$ to $H$, denoted by $\pi|_H$, is said to be \emph{discretely decomposable} if it is a direct sum of irreducible
unitary representations of $H$. If moreover, all irreducible unitary representations of $H$ have only finite
multiplicities in $\pi$, then $\pi|_H$ is said to be \emph{admissible}. Kobayashi established criteria for the
admissibility for a large class of unitary representations with respect to reductive subgroups. Based on his work,
branching laws for admissible restriction have been studied in many papers including \cite{Duflo-Vargas2},
\cite{Gross-Wallach}, \cite{Kobayashi}, \cite{Kobayashi2}, \cite{Kobayashi3}, \cite{Kobayashi4},
\cite{Moellers-Oshima2}, \cite{Oshima}, \cite{Sekiguchi}. In this paper we set $G=\Spin(N,1)$ for $N>2$ and set $H$ to
be a minimal parabolic subgroup of $G$, which we denote by $P(\subset G)$. In the first half of the paper we obtain
explicit branching laws for all irreducible unitary representations of $G$. The formulas are given in
\S\ref{SS:branchinglaw} and \S\ref{SS:branchinglaw2}. We find that the restriction is always a finite direct sum of
irreducible unitary representations of $P$. Hence in particular, the restriction is admissible.

In the second half of the paper, we study moment maps of coadjoint orbits which is related to branching laws via the
so-called orbit method. Let $\pi$ be an irreducible unitary representation of $G$ associated to a $G$-coadjoint orbit
$\mathcal{O}$ in $\mathfrak{g}^{\ast}$. It is well known that equipped with the Kirillov-Kostant-Souriau symplectic
form, $\mathcal{O}$ becomes a $G$-Hamiltonian space and hence an $H$-Hamiltonian space. The corresponding moment map
is the natural projection $q\colon \mathcal{O} \rightarrow\mathfrak{h}^{\ast}$. The orbit method predicts that the
branching law of $\pi|_H$ is given in terms of the geometry of the moment map $q$ (see \cite{Kirillov}). In fact,
when $G$ and $H$ are unipotent groups, Corwin-Greenleaf~\cite{Corwin-Greenleaf} proved that the multiplicity of an
$H$-representation associated with an $H$-coadjoint orbit $\mathcal{O}'$ equals almost everywhere with the cardinality
of $q^{-1}(\mathcal{O}')/H$. Concerning more general Lie groups, recently Duflo formulated a precise conjecture which
describes a connection between the branching law of the restriction to a closed subgroup of discrete series on the
representation theory side and the moment map from strongly regular coadjoint orbits on the geometry side. The
conjecture is inspired by Heckman's thesis \cite{Heckman} and the ``quantization commuting with reduction" program
\cite{Guillemin-Sternberg}.

\begin{conjecture}\label{C:Duflo}
Let $\pi$ be a discrete series of a real almost algebraic group $G$, which is attached to a coadjoint orbit
$\mathcal{O}_{\pi}$. Let $H$ be a closed almost algebraic subgroup, and
 let $q\colon \mathcal{O}_{\pi}\rightarrow \mathfrak{h}^{\ast}$
 be the moment map from $\mathcal{O}_{\pi}$. Then,
\begin{itemize}
\item[(i)] $\pi\vert_{H}$ is $H$-admissible (in the sense of Kobayashi) if and only if the moment map $q\colon
\mathcal{O}_{\pi}\rightarrow\mathfrak{h}^{\ast}$ is \textit{weakly proper}.
\item[(ii)] If $\pi\vert_{H}$ is $H$-admissible, then each irreducible $H$-representation $\sigma$ which appears
in $\pi\vert_{H}$ is attached to a \textit{strongly regular} $H$-coadjoint orbit $\mathcal{O}'_{\sigma}$ (in the sense
of Duflo) contained in $q(\mathcal{O}_{\pi})$.
\item[(iii)] If $\pi\vert_{H}$ is $H$-admissible, then the multiplicity of each such $\sigma$ can be expressed
geometrically in terms of the \textit{reduced space} $q^{-1}(\mathcal{O}'_{\sigma})/H$.
 \end{itemize}
\end{conjecture}

Let us give some explanations for the conjecture. The notion of ``almost algebraic group" is defined in \cite{Duflo2}.
Recall that an element $f\in\mathfrak{g}^{\ast}$ is called strongly regular
if $f$ is regular (i.e., the coadjoint orbit containing $f$ is of maximal dimension) and its ``reductive factor"
$\mathfrak{s}(f):= \{X\in \mathfrak{g}(f):\ad(X)\text{ is semisimple}  \}$ is of maximal dimension among reductive factors
of all regular elements in $\mathfrak{g}^*$ ($f$ is regular implies that $\mathfrak{g}(f)$ is commutative). Let
$\Upsilon_{sr}$ denote the set of strongly regular elements in  $\mathfrak{h}^*$. A coadjoint orbit $\mathcal{O}$ is called
strongly regular if there exists an element $f\in \mathcal{O}$ (then every element in $\mathcal{O}$) which is strongly regular.
``Weakly proper" in (i) means that the preimage (for $q$) of each compact subset which is contained in $q(\mathcal{O}_{\pi})
\cap\Upsilon_{sr}$ is compact in $\mathcal{O}_{\pi}$. Note that it is known that the classic properness condition is not
sufficient to characterize the $H$-admissibility when $H$ is not reductive (see \cite{Liu2}, \cite{Liu3}).
If $G$ is compact, then Duflo's conjecture is a special case of the $\text{Spin}^c$ version of  \emph{quantization commutes
with reduction} principle (see \cite{Paradan2}). More generally, if $G$ and $H$ are both reductive, then the assertions (ii)
and (iii) of the conjecture are consequences of a recent work of Paradan \cite{Paradan3}. Note that in this case, Dulfo-Vargas
and Paradan \cite{Duflo-Vargas1}, \cite{Duflo-Vargas2}, \cite{Paradan3} proved that $\pi\vert_{H}$ is $H$-admissible if and
only if the moment map $q\colon \mathcal{O}_{\pi}\rightarrow\mathfrak{h}^{\ast}$ is \textit{proper}. In order to prove the
assertion (i) of the conjecture in this case, one needs to prove the equivalence between properness and weak properness of
the moment map.

There is a fact that if $\pi$ is a tempered $G$-representation, then each irreducible $H$-representation appearing in the
spectrum of $\pi\vert_{H}$ is tempered. Thus when $\pi$ is a tempered representation (with regular infinitesimal character)
of a reductive group $G$,  Conjecture~\ref{C:Duflo} still makes sense. In this paper, based on our explicit branching laws
and an explicit description of the moment map, we verify Conjecture~\ref{C:Duflo} for the restriction to a minimal parabolic
subgroup of all tempered representations of $\Spin(N,1)$. In our setting the restriction is admissible for any irreducible
unitary representation $\pi$ while the moment map $q\colon \mathcal{O}\to \mathfrak{h}^{\ast}$ is weakly proper for any
$\mathcal{O}$. The restriction $\pi\vert_{H}$ is always multiplicity-free and the reduced space $q^{-1}(\mathcal{O}'_{\sigma})/H$
is always a singleton. In addition, we extend the conjecture to Vogan-Zuckerman's derived functor modules
$A_{\mathfrak{q}}(\lambda)$ and verify it. The representations $A_{\mathfrak{q}}(\lambda)$ are possibly non-tempered and are considered
to be associated with possibly singular elliptic orbits. For the conjecture in this case we need a certain modification
in the correspondence of orbits and representations for the parabolic subgroup (see \S\ref{SS:nontempered} for details).

For the proof of our branching laws, a key idea is
 to consider the non-compact picture ($N$-picture)
 of principal series representations of $G$
 and to take the classical Fourier transform.
Such an idea appeared in Kobayashi-Mano~\cite{Kobayashi-Mano}
 for the construction of
 an $L^2$-model (called the Schr\"{o}dinger model) of a minimal representation of $\OO(p,q)$.
This was extended to other groups and representations in
 \cite{HKMO}, \cite{Moellers}, \cite{Moellers-Oshima}.
They embed an irreducible representation into a degenerate principal series
 and then take the Fourier transform of the non-compact picture.
We will apply this method in our case and obtain $L^2$-models for all unitary principal series (see Appendix~\ref{S:principalSeries}) and for all irreducible unitary representations of $G$
with infinitesimal character $\rho$ and some complementary series (see Appendix \ref{S:trivial}).
The treatment for the latter representations is more involved. More precisely, we realize any of such a
representation as the image of a normalized Knapp-Stein intertwining operator between two non-unitary principal
series. Then, applying the Fourier transform to the non-compact picture of the non-unitary principal series, we
obtain the $L^2$-model. In this process, we find an explicit formula for the Fourier transformed counter-part of
the normalized Knapp-Stein intertwining operator and analyze the growth property at infinity and the singularity
at zero of the Fourier transformed functions (or distributions) carefully. Since the $P$-action on the $L^2$-model
is of a simple form, we have explicit branching laws for these representations. In order to obtain branching laws
for all irreducible unitary representations of $G$, we employ du Cloux's results \cite{duCloux} on moderate growth
smooth Fr\'echet representations of semi-algebraic groups and Zuckerman translation principle (\cite{Zuckerman},
\cite{Knapp}).

On the geometry side, let $\mathcal{O}_{f}=G/G^{f}$ be a coadjoint $G$-orbit. By parametrizing the double coset
space $P\backslash G/G^{f}$ we find explicit representatives of $P$-orbits in $\mathcal{O}_{f}$. By calculating
the Pfaffian and the characteristic polynomial of related skew-symmetric matrix, we are able to identify the
$P$-class of the moment map image of each representative. With this, we calculate the image and show geometric
properties of the moment map. The moment map is always weakly proper, but it is not proper unless
$\mathcal{O}_{f}$ is a zero orbit. This supports Duflo's belief that weak properness is the correct counterpart
of $H$-admissibility. We prove moreover that the reduced space for each regular $P$-coadjoint orbit in the image
of the moment map is a singleton. By comparing the branching law of discrete series (or unitary principal series)
and the behavior of the moment map from the corresponding coadjoint orbit, we verify Conjecture~\ref{C:Duflo}
when $G=\Spin(N,1)$ and $H$ is a parabolic subgroup.

One might compare our branching laws with
 Kirillov's conjecture which says that the restriction to
 a mirabolic subgroup of any irreducible unitary representation of
 $\GL_{n}(k)$ (for $k$ an archimedean or non-archimedean local field) is irreducible.
Kirillov's conjecture was proved by Bernstein~\cite{Bernstein} for $p$-adic groups.
It awaited nearly ten years for a breakthrough by Sahi~\cite{Sahi} who proved it for tempered representations of
$\GL_{n}(k)$ for $k$ an archimedean local field. It was finally proved by Baruch~\cite{Baruch} over archimedean local
fields in general through a qualitative approach by studying invariant distributions. The restriction to a mirabolic
subgroup of a general irreducible unitary representation of $\GL_{n}(\mathbb{R})$ or $\GL_{n}(\mathbb{C})$ is determined
by combining ~\cite{Sahi} (which sets up a strategy to attack this problem and treats tempered representations),
~\cite{Sahi2} (which treats Stein complementary series), ~\cite{Sahi-Stein} (which treats Speh representations) and
~\cite{AGS} (which treats Speh complementary series). In the literature, there is another related/similar work by
Rossi-Vergne~\cite{Rossi-Vergne} concerning the restriction to a minimal parabolic subgroup of holomorphic
(or anti-holomorphic) discrete series of a Hermitian simple Lie group. As for the restriction of irreducible unitary
representations of $\Spin(N,1)$ ($N\geq 2$) to a minimal parabolic subgroup, we note that the branching law is known
in the literature only when $N=2$ or $3$ by Martin~\cite{Martin}, and when $N=4$ by Fabec~\cite{Fabec}. Note that
Fourier transform is used in all of these works. On the geometry side, we describe explicitly the moment map image
for any coadjoint orbit of $G=\Spin(N,1)$. For the mirabolic subgroup of $\GL_{n}(\mathbb{R})$ (or $\GL_{n}(\mathbb{C})$),
similar calculation was done in \cite{Liu-Yu}. Kobayashi-Nasrin~\cite{Kobayashi-Nasrin} studied the moment map image
which corresponds to the restriction of holomorphic discrete series of scalar type with respect to holomorphic symmetric
pairs studied in \cite{Kobayashi4}.

Representations of four groups $G=\Spin(m+1,1)$, $G_{1}=\SO_{e}(m+1,1)$, $G_{2}=\OO(m+1,1)$, $G_{3}=\SO(m+1,1)$,
(resp.\ their parabolic subgroups, maximal compact subgroups, etc.) are studied (resp.\ arise) in this paper. We
remark on the advantages of studying representations of each of these groups $G$, $G_{1}$, $G_{2}$: the advantage
of studying representations of $G_{2}$ is using the matrix $s=\diag\{I_{m},-1,1\}$ to define and study intertwining
operators; the advantage of studying representations of $G$ is applying the Zuckerman translation principle; then,
representations of $G_{1}$ serve as a bridge connecting representations of $G_{2}$ and representations of $G$.
The group $G_{3}$ and its representations are used only in Appendix~\ref{S:GGP}.

Readers might be curious if methods and results of this paper can be generalized in more general setting. On the
representation theory side, we remark that the decomposition of unitary principal series as done in Appendix~\ref{S:principalSeries} can
be generalized well. The way of treating representations with trivial infinitesimal character and some complementary
series can be generalized partially. However, the comparing of $L^2$ spectrum and smooth quotient as done in
$\S$\ref{SS:CW-Cloux} and $\S$\ref{SS:ResPS} can be hardly generalized. On the geometry side, the method of moment
map calculation in this paper for a coadjoint orbit $\mathcal{O}_{g}=G\cdot f$ ($f\in\mathfrak{g}^{\ast}$) is
applicable whenever the double coset space $G^{f}\backslash G/P$ has a simple description, which is the case if
$\dim P$ is as large as possible and $f$ is as singular as possible.

The paper is organized as follows. In Section~\ref{S:repP} we introduce notation used throughout the paper. We give a
classification of irreducible unitary representations of $P$ and coadjoint $P$-orbits. In Section~\ref{S:resP} we obtain
branching laws of irreducible unitary representations of $G$ when restricted to $P$. We require an explicit calculation
of the Fourier transform of a vector in the lowest $K$-type of discrete series, which will be done in Appendix~\ref{S:trivial}.
Sections~\ref{S:elliptic} and \ref{S:non-elliptic} are devoted to the description of the moment map $q\colon\mathcal{O}
\to\mathfrak{p}^{\ast}$. In Section~\ref{S:Duflo} we verify Conjecture~\ref{C:Duflo} in our setting. In Appendices~\ref{S:principalSeries} and
\ref{S:trivial} we see the Fourier transformed picture more explicitly for particular representations: unitary principal series, those
with infinitesimal character $\rho$ and some complementary series. This yields the $L^2$-models and the decomposition
into irreducible $P$-representations of these representations. In Appendix~\ref{S:MN.K}, we show that $\bar{\pi}|_{MN}$ is
determined by the $K$ type of $\bar{\pi}$ for any irreducible unitarizable representation $\pi$ of $G$. In Appendix~\ref{S:GGP},
we explain that branching laws shown in this paper are related to a case of Bessel model of the local
Gan-Gross-Prasad conjecture (\cite{Gan-Gross-Prasad},\cite{Gan-Gross-Prasad2}).

\smallskip

\noindent\textbf{Acknowledgements.} We would like to thank Professors Michel Duflo and Pierre Torasso for suggesting this
problem. The authors thank Professor David Vogan for helpful suggestions and for providing the reference \cite{duCloux}.
Yoshiki Oshima would like to thank Professor Bent {\O}rsted for helpful comments about Appendix~\ref{S:trivial}. Jun Yu
would like to thank Professor David Vogan for encouragement and to thank Chengbo Zhu, Wen-wei Li, Lei Zhang for helpful
communications. Yoshiki Oshima was partially supported by JSPS Kakenhi Grant Number JP16K17562. Jun Yu was partially
supported by the NSFC Grant 11971036.

\section{Preliminaries}\label{S:repP}

\subsection{Notation and conventions}\label{SS:notation}


{\it Indefinite orthogonal and spin groups of real rank one.} Fix a positive integer $m$. Let $I_{m+1,1}$ be
the $(m+2)\times (m+2)$-matrix given as \[I_{m+1,1}=\begin{pmatrix}I_{m+1}&\\&-1\end{pmatrix}.\]
Put \begin{align*}
&G_{2}=\OO(m+1,1)=\{X\in M_{m+2}(\mathbb{R}):XI_{m+1,1}X^{t}=I_{m+1,1}\},\\
&G_{3}=\SO(m+1,1)=\{X\in\OO(m+1,1):\det X=1\},\\
&G_{1}=\SO_{e}(m+1,1),\\& G=\Spin(m+1,1),
\end{align*}
where $\SO_{e}(m+1,1)$ is the identity component of $\OO(m+1,1)$ (and of $\SO(m+1,1)$), and $\Spin(m+1,1)$ is
a universal covering group of $\SO_{e}(m+1,1)$. The Lie algebras of $G,G_{1},G_{2},G_{3}$ are all equal to
\begin{equation*}
\mathfrak{g}
=\mathfrak{so}(m+1,1)=\{X\in\mathfrak{gl}(m+2,\mathbb{R}): XI_{m+1,1}+I_{m+1,1}X^{t}=0\}.
\end{equation*}


\smallskip

{\it Cartan decomposition.} Write
\begin{align*}
&K=\Spin(m+1),\\
&K_{1}=\{\diag\{Y,1\}:Y\in\SO(m+1)\},\\
&K_{2}=\{\diag\{Y,t\}:Y\in\OO(m+1),t\in\{\pm{1}\}\},\\
&K_{3}=\{\diag\{Y,t\}:Y\in\OO(m+1),t=\det Y\}.
\end{align*}
Then, $K,K_{1},K_{2},K_{3}$ are maximal compact subgroups of $G,G_{1},G_2,G_{3}$ respectively.
Their Lie algebras are equal to \[\mathfrak{k}=\{\diag\{Y,0\}: Y\in\mathfrak{so}(m+1)\}.\]
Write \[\mathfrak{s}=
\Bigl\{\begin{pmatrix}0_{(m+1)\times (m+1)}&\alpha^{t}\\ \alpha&0\end{pmatrix}
:\alpha\in M_{1\times (m+1)}(\mathbb{R})\Bigr\}.\]
Then, $\mathfrak{g}=\mathfrak{k}\oplus\mathfrak{s}$,
 which is a Cartan decomposition for $\mathfrak{g}$.
The corresponding Cartan involution $\theta$ of $G_{1}$ (or $G_{2}$, $G_{3}$) is given by
$\theta=\Ad(I_{m+1,1})$.

\medskip

{\it Restricted roots and Iwasawa decomposition.}
Put \begin{equation*}
H_0=\begin{pmatrix}0_{m\times m}&0_{m\times 1}&0_{m\times 1}\\0_{1\times m}&0&1\\0_{1\times m}&1&0\\
\end{pmatrix}\textrm{ and }\mathfrak{a}=\mathbb{R}\cdot H_{0},\end{equation*}
which is a maximal abelian subspace in $\mathfrak{s}$.
Define $\lambda_{0}\in\mathfrak{a}^*$ by
$\lambda_{0}(H_{0})=1$.
Then, the restricted root system $\Delta(\mathfrak{g}_{\mathbb{C}},\mathfrak{a})$ consists of two roots
$\{\pm{\lambda_0}\}$. Let $\lambda_0$ be a positive restricted root.
Then the associated positive nilpotent part is
\begin{equation*}
\mathfrak{n}
=\Biggl\{ \begin{pmatrix}0_{m\times m}&-\alpha^{t}&\alpha^{t}\\
\alpha&0&0\\\alpha&0&0\\ \end{pmatrix}:
\alpha\in M_{1\times m}(\mathbb{R})\Biggr\}.
\end{equation*}
Let $\rho'$ be half the sum of positive roots in $\Delta(\mathfrak{n},\mathfrak{a})$.
Then
\[\rho'=\frac{m}{2}\lambda_{0}\textrm{ and } \rho'(H_0)=\frac{m}{2}.\]
One has the {\it Iwasawa decomposition}
$\mathfrak{g}=\mathfrak{k}\oplus\mathfrak{a}\oplus\mathfrak{n}$.

\medskip

{\it Standard parabolic and opposite parabolic subalgebras.}
Let \begin{equation*}
\mathfrak{m}=
Z_{\mathfrak{k}}(\mathfrak{a})=\{\diag\{Y,0_{2\times 2}\}:Y\in\mathfrak{so}(m,\mathbb{R})\}.
\end{equation*}
Write
$\mathfrak{p}=\mathfrak{m}+\mathfrak{a}+\mathfrak{n}$,
 which is a parabolic subalgebra of $\mathfrak{g}$.
We have the opposite nilradical
\begin{equation*}\label{Eq:nbar}
\bar{\mathfrak{n}}=
\Biggl\{\begin{pmatrix}0
_{m\times m}&\alpha^{t}&\alpha^{t}\\-\alpha&0&0\\
\alpha&0&0\\
\end{pmatrix}:\alpha\in M_{1\times m}(\mathbb{R})\Biggr\}
\end{equation*}
 and the opposite parabolic subalgebra
$\bar{\mathfrak{p}}=\mathfrak{m}+\mathfrak{a}+\bar{\mathfrak{n}}$.

\medskip

{\it Subgroups.} Let $A$ (or $A_{1}$, $A_{2}$, $A_{3}$), $N$ (or $N_{1}$, $N_{2}$, $N_{3}$), $\bar{N}$
(or $\bar{N}_{1}$, $\bar{N}_{2}$, $\bar{N}_{3}$) be analytic subgroups of $G$ (or $G_{1}$, $G_{2}$, $G_{3}$)
with Lie algebras $\mathfrak{a},\mathfrak{n},\bar{\mathfrak{n}}$ respectively.
In particular,
\begin{equation*}
A_{1}=\Biggl\{\begin{pmatrix}I_{m}&&\\
&r&s\\ &s&r\\
\end{pmatrix}:r,s\in\mathbb{R},\ r^{2}-s^{2}=1,\ r>0\Biggr\}.
\end{equation*}
By the two-fold covering
$G\twoheadrightarrow G_{1}$ and the inclusions $G_{1}\subset G_{2}$ and $G_{1}\subset G_{3}$ we identify
$A$, $A_{2}$, $A_{3}$ with $A_{1}$, identify $N$, $N_{2}$, $N_{3}$ with $N_{1}$, and identify $\bar{N}$,
$\bar{N}_{2}$, $\bar{N}_{3}$ with $\bar{N}_{1}$.

Put \[M=Z_{K}(\mathfrak{a}),\quad M_{1}=Z_{K_1}(\mathfrak{a}),\quad M_{2}=Z_{K_{2}}(\mathfrak{a}),
\quad M_{3}=Z_{K_{3}}(\mathfrak{a}).\] Set \[P=MAN,\quad P_{1}=M_{1}AN,\quad P_{2}=M_{2}AN,
\quad P_{3}=M_{3}AN\] and \[\bar{P}=MA\bar{N}, \quad \bar{P}_{1}=M_{1}A\bar{N},
\quad \bar{P}_{2}=M_{2}A\bar{N},\quad \bar{P}_{3}=M_{3}A\bar{N}.\]
Then, the Lie algebras of $M$ (or $M_{1}$, $M_2$, $M_{3}$), $P$ (or $P_{1}$, $P_2$, $P_{3}$), $\bar{P}$
(or $\bar{P}_{1}$, $\bar{P}_2$, $\bar{P}_{3}$) are equal to $\mathfrak{m}$, $\mathfrak{p}$, $\bar{\mathfrak{p}}$
respectively. Note that
\begin{equation*}
M=\Spin(m),\  M_{1}=\SO(m),\  M_{2}=\OO(m)\times\Delta_{2}(\OO(1)),\  M_{3}=\SO(m)\times
\Delta_{2}(\OO(1)),\end{equation*} where $\Delta_{2}(\OO(1))=\{\diag\{t,t\}:t=\pm{1}\}\subset\OO(2)$.

For $\Spin(2,1)\cong\SL_{2}(\mathbb{R})$,
 the Langlands classification of irreducible
 $(\mathfrak{g},K)$-modules requires different parametrization due to $M$ is disconnected.
For convenience in this paper we treat only $\Spin(m+1,1)$ ($m>1$),
though the case of $m=1$ is even easier (which was treated in \cite{Martin}).

\medskip

{\it Nilpotent elements.} For a row vector
 $\alpha\in \bbR^{m}$ write
\begin{equation*}
X_{\alpha}=\begin{pmatrix}
0_{m\times m}&-\alpha^{t}&\alpha^{t}\\ \alpha&0&0\\ \alpha&0&0\\
\end{pmatrix},\quad
\bar{X}_{\alpha}=
\begin{pmatrix}
0_{m\times m}&\alpha^{t}&\alpha^{t}\\ -\alpha&0&0\\ \alpha&0&0\\
\end{pmatrix}.
\end{equation*}
The following Lie brackets will be used later
\begin{align}\label{Eq:brackets}
&[X_{\alpha},\bar{X}_{\beta}]=
\diag\{2(\alpha^t\beta-\beta^t\alpha),0_{2\times 2}\}+2\alpha\beta^t H_0, \\ \nonumber
&[\diag\{Y,0_{2\times 2}\},\,X_{\alpha}]=X_{\alpha Y^t},\\ \nonumber
&[\diag\{Y,0_{2\times 2}\},\,\bar{X}_{\beta}]=\bar{X}_{\beta Y^t},\\ \nonumber
&[H_0,X_{\alpha}]=X_{\alpha},\\ \nonumber
&[H_0,\bar{X}_{\beta}]=-\bar{X}_{\beta},
\end{align}
where $\alpha,\beta\in \bbR^{m}$
 and $Y\in\mathfrak{so}(m,\mathbb{R})$.

Put
$n_{\alpha}=\exp(X_{\alpha})$ and $\bar{n}_{\alpha}=\exp(\bar{X}_{\alpha})$.
Then by $n_{\alpha}=I+X_{\alpha}+\frac{1}{2}X_{\alpha}^2$ and
$\bar{n}_{\alpha}=I+\bar{X}_{\alpha}+\frac{1}{2}\bar{X}_{\alpha}^2$ we have
\begin{equation*}
n_{\alpha}=
\begin{pmatrix}I_{m}&-\alpha^{t}&\alpha^{t}\\
 \alpha&1-\frac{1}{2}|\alpha|^2&\frac{1}{2}|\alpha|^2\\
 \alpha&-\frac{1}{2}|\alpha|^{2}&1+\frac{1}{2}|\alpha|^2\\
\end{pmatrix}
\end{equation*} and
\begin{equation*}
\bar{n}_{\alpha}=
\begin{pmatrix}I_{m}&\alpha^{t}&\alpha^{t}\\
 -\alpha&1-\frac{1}{2}|\alpha|^2&-\frac{1}{2}|\alpha|^2\\
 \alpha&\frac{1}{2}|\alpha|^2&1+\frac{1}{2}|\alpha|^2\\
\end{pmatrix}.
\end{equation*}
Via the maps $\alpha\mapsto X_{\alpha}$ and $\alpha\mapsto n_{\alpha}$,
 one identifies $\mathfrak{n}$ and $N$ with the Euclidean space $\mathbb{R}^{m}$.
We have $\theta(n_{\alpha})=\bar{n}_{-\alpha}$.

\medskip

{\it Invariant bilinear form.}
For $X,Y\in\mathfrak{g}$, define
\begin{equation}\label{Eq:form}
(X,Y)=\frac{1}{2}\tr(XY).
\end{equation}
Then $(\cdot,\cdot)$ is a nondegenerate symmetric bilinear form on $\mathfrak{g}$,
 which is invariant under the adjoint action of $G$, $G_{1}$ and $G_{2}$.
Define $\iota\colon \mathfrak{g}\to \mathfrak{g}^{\ast}$ by
\begin{equation}\label{Eq:identification1}
\iota(X)(Y)=(X,Y)\ (\forall Y\in\mathfrak{g}).
\end{equation}
Then, $\iota$ is an isomorphism of $G_{2}$ (or $G$, $G_{1}$) modules.
Define $\pr\colon \mathfrak{g}\to \mathfrak{p}^{\ast}$ by
\begin{equation}\label{Eq:identificaiton2}
\pr(X)(Y)=(X,Y)\
(\forall Y\in\mathfrak{p}).
\end{equation}
Then, the kernel of $\pr$ is $\mathfrak{n}$, and $\pr$ gives a $P$ (or
$P_{1}$, $P_{2}$) module isomorphism \[\mathfrak{g}/\mathfrak{n}\cong\mathfrak{p}^{\ast}.\]
Since $\bar{\mathfrak{p}}$ is a complement of $\mathfrak{n}$ in $\mathfrak{g}$,
 we can identify $\bar{\mathfrak{p}}$ with
$\mathfrak{p}^{\ast}$ by
$X\mapsto\pr(X)\ (X\in\bar{\mathfrak{p}})$.

\medskip

{\it Roots and weights.}
Let $n':=\lfloor \frac{m+1}{2}\rfloor$.
For $\vec{a}=(a_1,\dots,a_{n'})\in \bbR^{n'}$,
 let
\begin{align}\label{Eq:ta}
t_{\vec{a}}
:=\begin{pmatrix}
0&a_1&& &&\\
-a_{1}&0&&&&\\
&&\ddots&&&\\
&&&0&a_{n'}&\\
&&&-a_{n'}&0&\\
&&&&&0_{(m+2-2n')\times (m+2-2n')}\\
\end{pmatrix}.
\end{align}
Then
\[\mathfrak{t}=\{t_{\vec{a}} : a_1,\dots,a_{n'}\in \bbR\}\]
 is a maximal abelian subalgebra of $\mathfrak{k}=\Lie K$.
Write $T$ for the corresponding maximal torus in $K$.
Define $\epsilon'_i \in \mathfrak{t}_{\bbC}^*$ by
\[\epsilon'_i\colon t_{\vec{a}} \mapsto \mathbf{i}a_i.\]
The root system $\Delta(\mathfrak{k}_{\bbC},\mathfrak{t}_{\bbC})$
 is given by
\begin{align*}
&\{\pm{\epsilon'_{i}}\pm{\epsilon'_{j}}, \pm{\epsilon'_{k}}:1\leq i<j\leq n',\ 1\leq k\leq n'\}
\ \text{ if $m$ is even and},\\
&\{\pm{\epsilon'_{i}}\pm{\epsilon'_{j}}:1\leq i<j\leq n'\}
\ \text{ if $m$ is odd}.
\end{align*}
We denote the weight
 $c_1\epsilon'_1+\cdots +c_{n'}\epsilon'_{n'}\in \mathfrak{t}_{\bbC}^*$
 by $(c_1,\dots,c_{n'})$.
The similar notation will be used for
 elements in $(\mathfrak{t}\cap \mathfrak{m})_{\bbC}^*$
 and $(\mathfrak{t}\cap \mathfrak{m}')_{\bbC}^*$,
 where $\mathfrak{m}'$ denotes the Lie algebra of
 the group $M'$ defined in \eqref{Eq:M'}.

\begin{remark}\label{R:weight-orbit}
The bilinear form \eqref{Eq:form} on $\mathfrak{t}$ is given as
 $(t_{\vec{a}},t_{\vec{b}})=-\vec{a}\cdot \vec{b}^t$.
Hence by using the isomorphism
 $\iota\colon \mathfrak{g}_{\bbC}\to \mathfrak{g}_{\bbC}^*$
 defined as \eqref{Eq:identification1},
 we have for example
 $\epsilon'_1 = \mathbf{i}\cdot \iota(t_{(-1,0,\dots,0)})|_{\mathfrak{t}}$.
\end{remark}

Define $T_{s}:=(T\cap M)\times A$.
Then $T\cap M$ is  a maximal torus of $M$ and
 $T_{s}$ is a Cartan subgroup of $G$.
Let $n:=\lfloor \frac{m+2}{2}\rfloor$.
Note that $m=2n-2$ and $n=n'+1$ if $m$ is even;
 $m=2n-1$  and $n=n'$ if $m$ is odd.
Define $\epsilon_i\in (\mathfrak{t}_s)_{\bbC}^*$ by
\begin{align*}
&\epsilon_i=\epsilon'_i \text{ on $\mathfrak{t}\cap\mathfrak{m}$},
\quad \epsilon_i=0 \text{ on $\mathfrak{a}$}\ \text{ for $1\leq i<n$},\\
&\epsilon_n=0 \text{ on $\mathfrak{t}\cap\mathfrak{m}$},
\quad \epsilon_n=\lambda_0 \text{ on $\mathfrak{a}$}.
\end{align*}
The root system
 $\Delta=\Delta(\mathfrak{g}_{\bbC},(\mathfrak{t}_{s})_{\bbC})$
 is given by
\begin{align*}
&\{\pm{\epsilon_{i}}\pm \epsilon_{j}: 1\leq i<j\leq n\}
\ \text{ if $m$ is even and},\\
&\{\pm{\epsilon_{i}}\pm \epsilon_{j}, \pm{\epsilon_{k}}
 : 1\leq i<j\leq n,\ 1\leq k\leq n\}
\ \text{ if $m$ is odd},
\end{align*}
 where
 \[\{\pm{\epsilon_{i}}\pm{\epsilon_{j}},\pm{\epsilon_{k}}
 :1\leq i<j\leq n-1,\ 1\leq k \leq n-1\}\]
 are roots of $MA$.
Choose a positive system
\begin{align*}
&\Delta^+=\{\epsilon_{i}\pm \epsilon_{j}:1\leq i<j\leq n\}
\ \text{ if $m$ is even and},\\
&\Delta^+=\{\epsilon_{i}\pm \epsilon_{j}, \epsilon_{k}:1\leq i<j\leq n,\ 1\leq k\leq n\}
\ \text{ if $m$ is odd}.
\end{align*}
Then the corresponding simple roots are $\{\epsilon_{1}-\epsilon_{2},\dots,\epsilon_{n-1}-\epsilon_{n},
\epsilon_{n-1}+\epsilon_{n}\}$ for even $m$ and $\{\epsilon_{1}-\epsilon_{2},\dots,\epsilon_{n-1}-
\epsilon_{n},\epsilon_{n}\}$ for odd $m$. A weight of $T_{s}$ is of the form \[\gamma=c_1\epsilon_1+
\cdots+c_n\epsilon_n,\] which we denote by $(c_{1},\dots,c_{n})$. Put \[\mu=(c_{1},\dots,c_{n-1})
=c_1\epsilon_1+\cdots +c_{n-1}\epsilon_{n-1},\] which vanishes on $\mathfrak{a}$ and may be regarded
as a weight of $\mathfrak{t}\cap \mathfrak{m}$; put \[\nu=c_{n}\epsilon_{n},\] which vanishes on
$\mathfrak{t}\cap \mathfrak{m}$ and may be regarded as a weight of $\mathfrak{a}$. Then $\gamma=
(\mu,\nu)=\mu+\nu$.


The vector $\rho:=\frac{1}{2}\sum_{\alpha\in\Delta^{+}}\alpha$ is given as
\[\Bigl(n-\frac{1}{2},\dots,\frac{3}{2},\frac{1}{2}\Bigr)
\text{ for $m$ odd}; \
 (n-1,\dots,1,0)
\text{ for $m$ even}.\]

\medskip

{\it Reflections.} For a row vector
 $0\neq x\in\mathbb{R}^{m}$, write
\begin{equation}\label{Eq:rx}
r_{x}=I_{m}-
\frac{2}{|x|^2}x^{t}x\in\OO(m),
\end{equation}
which is a reflection.
The action of $r_{x}$ on $\mathbb{R}^{m}$ is given by
\[r_{x}(y)=y-\frac{2yx^t}{|x|^2}x\quad (\forall y\in\mathbb{R}^{m}).\]
Let $r_{x}$ also denote the element \[\diag\{r_{x},I_{2}\}\in G_{2}=\OO(m+1,1).\]

Write
\begin{equation}\label{Eq:s}s=\diag\{I_{m},-1,1\}\in\OO(m+1,1).
\end{equation}
For $x\in\mathbb{R}^{m}$, write
\begin{align}\label{Eq:sx}
s_{x}=
\begin{pmatrix}I_{m}-\frac{2x^{t}x}{1+|x|^2}&
-\frac{2x^{t}}{1+|x|^2}&0\\
\frac{2x}{1+|x|^2}&\frac{1-|x|^2}{1+|x|^2}&0\\
0&0&1\\
\end{pmatrix}\in\OO(m+1,1).
\end{align}

\medskip

{\it Unitarily induced representations of $P$, $P_{1}$, $P_{2}$ and $P_{3}$.}
Through the map
\[Y\mapsto\begin{pmatrix}Y&\\&I_{2}\\\end{pmatrix},\]
one identifies $M_{1}$ with $\SO(m)$.
With this identifications, the adjoint action of $M_{1}A$ on $\mathfrak{n}$ is given by
\[\Ad(Y,a)X_\alpha=e^{\lambda_0(\log a)} X_{\alpha Y^t}\ (\forall \alpha\in \bbR^{m},
\ \forall Y\in\SO(m),\ \forall a\in A).\] Moreover, identify $\mathbb{R}^{m}$ with
$\mathfrak{n}^{\ast}$ by $\xi(X_{\alpha})=\xi\alpha^t$. Then, the coadjoint action of $M_{1}A$
on $\mathfrak{n}^{\ast}$ is given by \begin{equation}\label{Eq:MAaction}
\Ad^{\ast}(Y,a)\xi = e^{-\lambda_0(\log a)} \xi Y^t\
(\forall \xi\in \bbR^{m},\ \forall Y\in\SO(m),\ \forall a\in A).
\end{equation}

Let \[\xi_0=(0,\dots,0,1)\in\mathfrak{n}^{\ast}.\] Put \[M'_{1}=\Stab_{M_{1}A}\xi_0.\]
Then,\begin{equation*}
M'_{1}=\Bigl\{\begin{pmatrix}Y&\\ &I_{3}\\ \end{pmatrix}:Y\in\SO(m-1)\Bigr\}.\end{equation*}
For a (not necessarily irreducible) unitary representation $(\tau,V_{\tau})$ of $M'_{1}$, let
\begin{equation*}
I_{P_{1},\tau}=\Ind_{M'_{1}\ltimes N}^{M_{1}AN}(\tau\otimes e^{\mathbf{i}\xi_{0}})\end{equation*}
be a unitarily induced representation. It consists of functions
$h\colon M_{1}AN \rightarrow V_{\tau}$ with
\[h(pm'n)=(\tau\otimes e^{\mathbf{i}\xi_{0}})(m',n)^{-1}h(p)\]
 for all $(p,m',n)\in P_{1}\times M'_{1}\times N$
 and $\langle h,h\rangle<\infty$, where
\begin{equation*}
\langle h_1,h_2\rangle:=\int_{M_{1}A/M_{1}'}
\langle  h_{1}(ma),h_{2}(ma)\rangle_{\tau}\d {}_{l}ma
\end{equation*}
for $h_1,h_2\in\Ind_{M'_{1}\ltimes N}^{M_{1}AN}(\tau\otimes e^{\mathbf{i}\xi_{0}})$.
Here $\d {}_{l}m_0a$ is a left $M_{1}A$ invariant measure on $(M_{1}A)/M'_{1}$,
 and $\langle\cdot,\cdot \rangle_{\tau}$
 denotes an $M'_{1}$-invariant inner product on $V_{\tau}$.
The action of $P_{1}$ on $I_{P_{1},\tau}$ is given by
$(p\cdot h)(x)=h(p^{-1}x)$ for $h\in I_{P_{1},\tau}$ and $p,x\in P_1$.

Similarly, put \begin{equation}\label{Eq:M'}M'=\Stab_{MA}\xi_0,\quad M'_{2}=\Stab_{M_{2}A_{2}}\xi_0,
\quad M'_{3}=\Stab_{M_{3}A_{3}}\xi_0,\end{equation} and define a unitarily induced representation
$I_{P,\tau}$ (or $I_{P_{2},\tau}$, $I_{P_{3},\tau}$) from a unitary representation $\tau$ of $M'$
(or $M'_{2}$, $M'_{3}$). One has \begin{equation*}
M'\cong\Spin(m-1),\quad M'_{2}\cong\OO(m-1)\times\OO(1),\quad M'_{3}\cong\SO(m-1)\times\OO(1).
\end{equation*}

\smallskip

{\it Induced representations of $G$.} For a finite-dimensional irreducible complex
linear representation $(\sigma,V_{\sigma})$ of $M$ and a character $e^\nu$ of $A$,
 form the smoothly induced representation
\[I(\sigma,\nu)
 =\Ind_{MA\bar{N}}^{G}(\sigma\otimes e^{\nu-\rho'}\otimes\mathbf{1}_{\bar{N}})\]
 which consists of smooth functions
 $h\colon G\rightarrow V_{\sigma}$ with
\[h(gma\bar{n})=\sigma(m)^{-1}e^{(-\nu+\rho')\log a}h(g)\]
for any $(g,m,a,\bar{n})\in G\times M\times A\times\bar{N}$.
The action of $G$ on $I(\sigma,\nu)$ is given by
$(g\cdot h)(x)=h(g^{-1}x)$ for $h\in I(\sigma,\nu)$ and $g,x\in G$.
Write $I(\sigma,\nu)_K$ for the space of functions
 $h\in I(\sigma,\nu)$ such that $h|_{K}$ is a $K$-finite function.
Then, $I(\sigma,\nu)_K$ is the space of $K$-finite vectors
 in $I(\sigma,\nu)$, and it is a $(\mathfrak{g},K)$-module.

When $\nu$ is a unitary character,
 let $\bar{I}(\sigma,\nu)$ denote the space of all functions
 $h\colon G\rightarrow V_{\sigma}$ with
\[h(gma\bar{n})=\sigma(m)^{-1}e^{(-\nu+\rho')\log a}h(g)\]
 for any $(g,m,a,\bar{n})\in G\times M\times A\times\bar{N}$,
 and $h_{N}=h|_{N}\in L^{2}(N,V_{\sigma},\d n)$.
This is called a unitary principal series representation of $G$.
The invariant inner product on $\bar{I}(\sigma,\nu)$ is defined by:
\begin{equation*}
\langle f_1,f_2\rangle=\int_{N}(f_{1}(n),
f_{2}(n))\d n
\end{equation*}
 for $f_{1},f_{2}\in L^{2}(N,V_{\sigma},\d n)$.

When $\sigma$ factors through $M_{1}$, the smoothly induced representation $I(\sigma,\nu)$ factors through $G_{1}$,
and the $(\mathfrak{g},K)$-module $I(\sigma,\nu)_K$ factors through a $(\mathfrak{g},K_{1})$-module. If $\nu$ is a unitary
character, then the unitarily induced representation $\bar{I}(\sigma,\nu)$ factors through $G_{1}$.

In the above, let $\mu$ denote the highest weight of $\sigma$.
Then for simplicity we denote
$I(\sigma,\nu), I(\sigma,\nu)_K, \bar{I}(\sigma,\nu)$
 by $I(\mu,\nu),I(\mu,\nu)_K, \bar{I}(\mu,\nu)$, respectively.

For a finite-dimensional irreducible complex linear representation $(\sigma,V_{\sigma})$ of $M_{2}$ and a character
$e^\nu$ of $A$, we define the smoothly induced representation \[I(\sigma,\nu)=\Ind_{M_{2}A\bar{N}}^{G_{2}}(\sigma\otimes e^{\nu-\rho'}\otimes\mathbf{1}_{\bar{N}})\] similarly.

\medskip

{\it $(\mathfrak{g},K)$-modules and $G$-representations.}
For a $(\mathfrak{g},K)$-module $V$,
 we denote by $V^{\sm}$ a Casselman-Wallach globalization of $V$
 (\cite{Casselman} and \cite{Wallach}).
If $V$ is unitarizable, we denote by $\bar{V}$
 for a Hilbert space completion of $V$.
For $(\mathfrak{g},K_{1})$-modules (or $(\mathfrak{g},K_{2})$-modules)
 and $G_{1}$-representations (or $G_{2}$-representations),
 we take similar conventions.

\medskip

{\it Irreducible finite-dimensional representations.}
Write $F_{\lambda}$ (resp.\ $V_{K,\lambda}$, $V_{M,\mu}$, $V_{M',\mu}$)
 for an irreducible finite-dimensional
 representation of $G$ (resp.\ $K$, $M$, $M'$) with highest weight $\lambda$,
 (resp.\ $\lambda$, $\mu$, $\mu$).

\medskip

{\it $L^{p}$ space.} We take Fourier transform on $N$ (or $\mathfrak{n}$). For a finite-dimensional Hilbert space $V$,
for brevity let $L^{p}$ denote the space of $V$-valued functions $h$ on $N$ (or $\mathfrak{n}$) such that $|h(x)|$ is
an $L^{p}$ integrable function.

\subsection{Iwasawa decomposition and Bruhat decomposition}\label{SS:group-algebra}

By a direct calculation,
 one shows the following opposite Iwasawa decomposition for elements in $N$.

\begin{lemma}\label{L:Iwasawa}
The opposite Iwasawa decomposition of a general element in $N$ is given by
\begin{equation}\label{Eq:Iwasawa3}
n_{x}=
\begin{pmatrix}
I_{m}-\frac{2x^{t}x}{1+|x|^2}&-\frac{2x^{t}}{1+|x|^2}&0\\
\frac{2x}{1+|x|^2}&\frac{1-|x|^2}{1+|x|^2}&0\\0&0&1\\
\end{pmatrix}\exp(-\log(1+|x|^2)H_{0})\bar{n}_{\frac{x}{1+|x|^2}}.
\end{equation}
\end{lemma}

One can write (\ref{Eq:Iwasawa3}) as
\[n_{x}=s_{x}\exp(-\log(1+|x|^2)H_{0})\bar{n}_{\frac{x}{1+|x|^2}},\]
where $s_{x}$ is as in (\ref{Eq:sx}).

By the Bruhat decomposition, \[G_{2}=NM_{2}A\bar{N}\sqcup sM_{2}A\bar{N}.\] By this, $sn_{x}\in NM_{2}A\bar{N}$ for any
$0\neq x\in\mathbb{R}^{m}$. By direct calculation one shows the following decomposition.

\begin{lemma}\label{L:barn-iwasawa}
For $0\neq x\in\mathbb{R}^{m}$, we have
\begin{equation*}
sn_{x}=n_{\frac{x}{|x|^{2}}}r_{x}e^{-(2\log|x|)H_{0}}\bar{n}_{\frac{x}{|x|^{2}}},
\end{equation*}
where $r_{x}$ is as in (\ref{Eq:rx}).
\end{lemma}

\subsection{Irreducible unitary representations of $P$}\label{SS:rep-P}

The classification of irreducible unitary representations of $P$ (or $P_{1},P_{2}$) is obtained by using Mackey's
little group method (see e.g.\ \cite{Wolf}). In the following we illustrate the classification of infinite-dimensional
irreducible unitary representations of $P$ (or $P_{1},P_{2}$).

\begin{proposition}\label{P:Mackey}
Any infinite-dimensional irreducible unitary representation of $P$ (or $P_{1},P_{2}$) is isomorphic to $I_{P,\tau}$
(or $I_{P_{1},\tau},I_{P_{2},\tau}$) for a unique (up to isomorphism) irreducible finite-dimensional unitary
representation $\tau$ of $M'$ (or $M'_{1},M'_{2}$).
\end{proposition}

\begin{proof}
We sketch a proof for $P$. The proof for $P_{1}$ (or $P_{2}$) is similar. Let $\pi$ be an irreducible unitary
representation of $P$. If $\pi|_{N}$ is trivial, then $\pi$ factors through $P\to MA$ and is finite-dimensional.
Assume that $\pi|_{N}$ is non-trivial, then the support of the the spectrum of $\pi|_{N}$ is not equal to $\{0\}$.
As the spectrum of $\pi|_{N}$ is an $MA$-stable subset of $\mathfrak{n}^{\ast}$ and $MA$ acts transitively
on $\mathfrak{n}^{\ast}-\{0\}$, $\xi_0$ is in the support. By Mackey's method, one then shows $\pi\cong I_{P,\tau}$
for a unique finite-dimensional irreducible unitary representation $\tau$ of $M'$ up to an isomorphism.
\end{proof}

\subsection{Abstract classification of coadjoint orbits in $\mathfrak{p}^{\ast}$}\label{SS:P-orbit}

Write $L=MA$ and $\mathfrak{l}=\mathfrak{m}\oplus\mathfrak{a}$. Then, $P=N\rtimes L$ is a Levi decomposition of $P$.
Write $L_{1}=M_{1}A\subset P_{1}$. Then, $P_{1}=N\rtimes L_{1}$.
There are exact sequences of $P$-modules
\[0\rightarrow\mathfrak{n}\rightarrow\mathfrak{p}\rightarrow\mathfrak{l}\rightarrow 0
\ \text{ and }\
0\rightarrow\mathfrak{l}^{\ast}\rightarrow\mathfrak{p}^{\ast}\rightarrow\mathfrak{n}^{\ast}\rightarrow 0.\]
Note that the action of $P$ on $\mathfrak{p}$ (or $\mathfrak{p}^{\ast}$) factors through $P_{1}$.
We have
\[L_{1}=M_1 A \cong \SO(m)\times\mathbb{R}_{>0},\quad
 \mathfrak{n}^{\ast}\cong \mathbb{R}^{m},\]
 and $L_{1}$ acts on $\mathfrak{n}^{\ast}$
 as in \eqref{Eq:MAaction}.
Thus, $\{0\}$ and $\mathfrak{n}^{\ast}-\{0\}$ are the only two $L$-orbits in
$\mathfrak{n}^{\ast}$.

Write $$\psi_{n}:\mathfrak{p}^{\ast}\rightarrow\mathfrak{n}^{\ast}\textrm{ and }\psi_{l}:\mathfrak{p}^{\ast}\rightarrow
\mathfrak{l}^{\ast}$$ for projection maps corresponding to the inclusions $\mathfrak{n}\hookrightarrow\mathfrak{p}$ and $\mathfrak{l}\hookrightarrow\mathfrak{p}$.
Write
\[\phi_{l}:\mathfrak{l}^{\ast}\rightarrow\mathfrak{p}^{\ast}\textrm{ and }
\phi_{n}:\mathfrak{n}^{\ast}\rightarrow\mathfrak{p}^{\ast}\]
 for inclusions corresponding to projections $\mathfrak{p}
\rightarrow\mathfrak{l}$ and $\mathfrak{p}\rightarrow\mathfrak{n}$ from $\mathfrak{p}=\mathfrak{l}+\mathfrak{n}$. Then,
$\psi_{n}$ and $\phi_{l}$ are $P$-equivariant maps. The maps $\psi_{l}$ and $\phi_{n}$ are $L$-equivariant, but not
$P$-equivariant.

For any $\xi\in\mathfrak{n}^{\ast}$, put $\tilde{\xi}=\phi_{n}(\xi)\in\mathfrak{p}^{\ast}$. Then,
\[\psi_{n}^{-1}(P\cdot \xi) =\mathfrak{l}^{\ast}+P\cdot\tilde{\xi}.\]
Write $P^{\xi}=\Stab_{P}(\xi)$ and $L^{\xi}=\Stab_{L}(\xi)$.
Then,
$P^{\xi}=N\rtimes L^{\xi}$.

\begin{proposition}\label{P:orbit-reduction}
In the above setting, we have
\[(\mathfrak{l}^{\ast}+P\cdot\tilde{\xi})/P
\cong(\mathfrak{l}^{\ast}+\tilde{\xi})/P^{\xi}
\cong(\mathfrak{l}^{\ast}/\ad^*(\mathfrak{n})(\tilde{\xi}))/L^{\xi}\cong(\mathfrak{l}^{\xi})^{\ast}/L^{\xi}.\]
\end{proposition}

\begin{proof}
We have
\[\mathfrak{l}^{\ast}+\tilde{\xi}\subset\mathfrak{l}^{\ast}+P\cdot\tilde{\xi}
\ \text{ and }\
\mathfrak{l}^{\ast}+P\cdot \tilde{\xi}=P(\mathfrak{l}^{\ast}+\tilde{\xi}).\]
On the other hand, for any $g\in P$,
\[g(\mathfrak{l}^{\ast}+\tilde{\xi})
\cap(\mathfrak{l}^{\ast}+\tilde{\xi})\neq\emptyset\]
 if and only if $g\in P^{\xi}$. Thus,
\begin{equation*}
(\mathfrak{l}^{\ast}+P\cdot\tilde{\xi})/P\cong(\mathfrak{l}^{\ast}+\tilde{\xi})/P^{\xi}.
\end{equation*}

As $\mathfrak{n}$ is abelian, it acts trivially on $\mathfrak{n}^{\ast}$.
Thus, $$\ad^*(\mathfrak{n})
(\mathfrak{p}^{\ast})\subset\mathfrak{l}^{\ast}.$$ Apparently, $\mathfrak{n}$ acts trivially on $\mathfrak{l}^{\ast}$.
By these, $N$ acts on $\mathfrak{l}^{\ast}+\tilde{\xi}$ through translations: for $X\in\mathfrak{n}$ and $\eta\in \mathfrak{l}^{\ast}$,
\begin{equation*}
\exp(X)\cdot(\eta+\tilde{\xi})=(\eta+\ad^*(X)\tilde{\xi})+\tilde{\xi}.
\end{equation*}
Since $P^{\xi}=L^{\xi}N$, we get
\begin{equation*}
(\mathfrak{l}^{\ast}+\tilde{\xi})/P^{\xi}\cong(\mathfrak{l}^{\ast}/\ad^*(\mathfrak{n})(\tilde{\xi}))/L^{\xi}.
\end{equation*}

For $X\in\mathfrak{l}$ and $Y\in\mathfrak{n}$, \[(\ad^*(X)(\xi))(Y)=(\ad^*(X)(\tilde{\xi}))(Y)=-\tilde{\xi}([X,Y])
=-(\ad^*(Y)(\tilde{\xi}))(X).\] By this, \[X\in\mathfrak{l}^{\xi}\Leftrightarrow\tilde{\xi}|_{\ad(X)(\mathfrak{n})}
=0\Leftrightarrow(\ad^*(\mathfrak{n})(\tilde{\xi}))(X)=0.\] Thus, the null space of $\mathfrak{l}^{\xi}$
($\subset\mathfrak{l}$) in $\mathfrak{l}^{\ast}$ is $\ad^{\ast}(\mathfrak{n})\tilde{\xi}$. Hence,
\begin{equation}\label{Eq:orbit-reduc3}\mathfrak{l}^{\ast}/\ad^*(\mathfrak{n})(\tilde{\xi})
\cong(\mathfrak{l}^{\xi})^{\ast}.\end{equation} This finishes the proof of the proposition.
\end{proof}


Note that $\mathfrak{l}=\mathfrak{m}\oplus\mathfrak{a}$ with $\mathfrak{m}$ a compact semisimple Lie algebra
and $\mathfrak{a}$ a one-dimensional abelian Lie algebra. Thus, $\mathfrak{l}^{\ast}=\mathfrak{m}^{\ast}
\oplus\mathfrak{a}^{\ast}$.
Let
\[\mathfrak{l}^{\ast}_{\xi}
=\{\eta\in \mathfrak{l}^{\ast} \mid (\eta,\ad^*(\mathfrak{n})(\tilde{\xi}))=0\}.
\]
Here, $(\cdot,\cdot)$ is the invariant form on $\mathfrak{l}^{\ast}$ induced by the
 restriction on $\mathfrak{l}$ of the invariant form defined in (\ref{Eq:form}).
We note that $L^{\xi}\cdot \tilde{\xi}=\tilde{\xi}$ and hence
 $L^{\xi}\cdot \ad^*(\mathfrak{n})(\tilde{\xi})=\ad^*(\mathfrak{n})(\tilde{\xi})$.

\begin{lemma}\label{L:p-standard1}
We have $\mathfrak{l}^{\ast}=\mathfrak{l}^{\ast}_{\xi}\oplus\ad^*(\mathfrak{n})(\tilde{\xi})$, and
$\mathfrak{l}^{\ast}_{\xi}\cong(\mathfrak{l}^{\xi})^{\ast}$ as $L^{\xi}$-modules.
\end{lemma}

\begin{proof}
When $\xi=0$, we have $L^{\xi}=L$, $\ad^*(\mathfrak{n})(\tilde{\xi})=0$
 and $\mathfrak{l}^{\ast}_{\xi}=\mathfrak{l}^{\ast}$.
Then, the two statements in the lemma are clear.

When $\xi\neq 0$, one shows by a direct calculation that
\begin{equation*}
\mathfrak{a}^{\ast}\subset\ad^{\ast}(\mathfrak{n})(\tilde{\xi}).
\end{equation*}
Then, the restriction of the induced form on $\mathfrak{l}^{\ast}$ of $(\cdot,\cdot)$
is nondegenerate when restricted to $\ad^{\ast}(\mathfrak{n})(\tilde{\xi})$. Hence,
$\mathfrak{l}^{\ast}=\mathfrak{l}^{\ast}_{\xi}\oplus\ad^{\ast}(\mathfrak{n})(\tilde{\xi})$.
This is clearly a decomposition of $L^{\xi}$-modules. Combining with \eqref{Eq:orbit-reduc3}, we get
\begin{equation*}
\mathfrak{l}^{\ast}_{\xi}\cong\mathfrak{l}^{\ast}/\ad^{\ast}(\mathfrak{n})(\tilde{\xi})\cong
(\mathfrak{l}^{\xi})^{\ast}\qedhere
\end{equation*}
\end{proof}

\begin{lemma}\label{L:p-standard3}\begin{enumerate}
\item[(1)]Every $P$-orbit in $\mathfrak{l}^{\ast}+P\cdot\tilde{\xi}$ has a representative of the form
$\eta+\tilde{\xi}$, where $\eta\in\mathfrak{l}^{\ast}_{\xi}$.
\item[(2)]Two elements $\eta+\tilde{\xi}$ and $\eta'+\tilde{\xi}$ $(\eta,\eta'\in\mathfrak{l}^{\ast}_{\xi})$
are in the same $P$-orbit if and and only if $\eta$ and $\eta'$ are in the same $L^{\xi}$-orbit.
\item[(3)]For an element $\eta+\tilde{\xi}(\in\mathfrak{p}^{\ast})$ with $\eta\in\mathfrak{l}^{\ast}_{\xi}\cong
(\mathfrak{l}^{\xi})^{\ast}$, if $\xi\neq 0$, then $\Stab_{P}(\eta+\tilde{\xi})=\Stab_{L^{\xi}}(\eta)$;
if $\xi=0$, then $\Stab_{P}(\eta+\tilde{\xi})=\Stab_{L}(\eta)N$.
\end{enumerate}
\end{lemma}

\begin{proof}
(1) and (2) follow from Proposition~\ref{P:orbit-reduction} and Lemma~\ref{L:p-standard1}.

(3) Suppose that $g\cdot (\eta+\tilde{\xi})=(\eta+\tilde{\xi})$ where $\eta\in\mathfrak{l}^{\ast}_{\xi}$ and
$g\in P$. First, we have $g\in P^{\xi}=L^{\xi}N$ by projecting to $\mathfrak{n}^{\ast}$. Write $g=g_1g_2$ with
$g_{1}\in L^{\xi}$ and $g_{2}\in N$. Since $g_2\cdot \eta=\eta$, we have \[g\cdot(\eta+\tilde{\xi})=
g_1\cdot\eta+ g_1g_2\cdot \tilde{\xi}=g_1\cdot \eta + g_1g_2g_1^{-1}\cdot \tilde{\xi}.\]
Hence $g_1\cdot \eta=\eta$ and $g_1g_2g_1^{-1}\tilde{\xi}=\tilde{\xi}$.
Then we get $g_{1}\in\Stab_{L^{\xi}}(\eta)$, which shows the statement when $\xi=0$.
When $\xi\neq 0$, we have $L\cdot \xi=\mathfrak{n}^*-\{0\}$
 which implies $\dim(\mathfrak{l}\cdot \xi) = \dim\mathfrak{n}$.
Then by Lemma \ref{L:p-standard1},
$\dim \mathfrak{n} =\dim(\ad^{\ast}(\mathfrak{n})(\tilde{\xi}))$
 and hence the map
\[\mathfrak{n}\rightarrow\mathfrak{l}^{\ast},\quad Y\mapsto
\ad^{\ast}(Y)(\tilde{\xi})\] is injective.
Thus, $g_{1}g_{2}g_{1}^{-1}=1$ and $g_{2}=1$. Therefore,
$\Stab_{P}(\eta+\tilde{\xi})=\Stab_{L^{\xi}}(\eta)$.
\end{proof}

For a given element $\xi\in\mathfrak{n}^{\ast}$, Proposition \ref{P:orbit-reduction} reduces the classification of $P$-orbits
intersecting with $\mathfrak{l}^{\ast}+\tilde{\xi}$ to the classification of coadjoint $L^{\xi}$-orbits in
$(\mathfrak{l}^{\xi})^{\ast}$. Moreover, for an element $\eta\in(\mathfrak{l}_{n})^{\ast}$, we showed that $\eta+\tilde{\xi}$
is a canonical form in Lemma \ref{L:p-standard3} (2); in Lemma \ref{L:p-standard3} (3), we calculated $\Stab_{P}(\eta+\tilde{\xi})$
in terms of $\Stab_{L^{\xi}}(\eta)$. Note that, when $\xi\neq 0$, one has $L^{\xi}\cong\Spin(m-1)$.

\begin{definition}\label{D:depth}
We say a coadjoint $P$-orbit $\mathcal{O}$ in $\mathfrak{p}^{\ast}$ has {\it depth zero} if it is contained
in $\mathfrak{l}^{\ast}$, otherwise we say $\mathcal{O}$ has {depth one}.
\end{definition}

\subsection{Explicit parametrization of coadjoint orbits in $\mathfrak{p}^{\ast}$}\label{SS:P-orbit2}

We now give a parametrization of depth one coadjoint $P$-orbits.
For $Y\in\mathfrak{so}(m)$, $\beta\in\mathbb{R}^{m}$ and $a\in\mathbb{R}$, put
\begin{equation*}
X_{Y,\beta,a}=
\begin{pmatrix}
Y&\beta^{t}&\beta^{t}\\
-\beta&0&a\\
\beta&a&0\\
\end{pmatrix}\in\overline{\mathfrak{p}}.
\end{equation*}
Then, for any
$X=\begin{pmatrix}
Y&\beta_{1}^{t}&\beta_2^{t}\\
-\beta_{1}&0&a\\
\beta_{2}&a&0\\
\end{pmatrix}\in\mathfrak{g},$
one has
$\pr(X)=\pr(X_{Y,\frac{\beta_{1}+\beta_{2}}{2},a})$;
for any
$f\in\mathfrak{p}^{\ast}$, there exists a unique triple
\[(Y,\beta,a)\in\mathfrak{so}(m)\times \mathbb{R}^{m}\times\mathbb{R}\]
 such that
$f=\pr(X_{Y,\beta,a})$.

\begin{lemma}\label{L:p-standard4}
For $0\neq\beta\in\mathbb{R}^{m}$, put
$\xi=\psi_{n}(\pr(X_{0,\beta,0}))\in\mathfrak{n}^{\ast}$.
In order that $\pr(X_{Y,0,a})\in\mathfrak{l}^{\ast}_{\xi}$
 it is necessary and sufficient that $a=0$ and $Y\beta^{t}=0$.
\end{lemma}

\begin{proof}
We have
\begin{align*}
\pr(X_{Y,0,a})\in\mathfrak{l}^{\ast}_{\xi}
\Leftrightarrow
(\pr(X_{Y,0,a}), \ad^*(\mathfrak{n})(\tilde{\xi}))=0
\Leftrightarrow
([X_{Y,0,a}, X_{0,\beta,0}], \mathfrak{n})=0.
\end{align*}
Since $[X_{Y,0,a}, X_{0,\beta,0}]=X_{0,\beta Y^t - a\beta,0}\in\bar{\mathfrak{n}}$
 by \eqref{Eq:brackets},
 $\pr(X_{Y,0,a})\in\mathfrak{l}^{\ast}_{\xi}$
 is equivalent to $\beta Y^t - a\beta=0$.
If $\beta Y^t - a\beta=0$,
 then $(\beta Y^t - a\beta)\beta^t = - a\beta\beta^t=0$.
Hence $a=0$ and $\beta Y^t=0$.
The lemma follows.
\end{proof}

\begin{lemma}\label{p:standard5}
For a general triple $(Y,\beta,a)\in\mathfrak{so}(m)\times\mathbb{R}^{m}\times\mathbb{R}$ with
 $\beta\neq 0$, put $\xi=\psi_{n}(\pr(X_{Y,\beta,a}))\in\mathfrak{n}^{\ast}$.
Then, there exists a unique $\gamma\in\mathbb{R}^{m}$
 such that $\Ad^*(n_{\gamma})(\pr(X_{Y,\beta,a}))\in\mathfrak{l}^{\ast}_{\xi}+\tilde{\xi}$.
Moreover,
\[\Ad^*(n_{\gamma})(\pr(X_{Y,\beta,a}))=
 \pr(X_{Y-\frac{1}{|\beta|^2}(Y\beta^{t}\beta-\beta^{t}\beta Y^{t}),\beta,0}).\]
\end{lemma}

\begin{proof}
For $\gamma\in \mathbb{R}^{m}$, we calculate by using \eqref{Eq:brackets}
\[\pr(\Ad(n_{\gamma})(X_{Y,\beta,a}))
=\pr(X_{Y+2\gamma^t\beta-2\beta^t\gamma,\,\beta,\,a+2\gamma\beta^t}).\]
By Lemma \ref{L:p-standard4}, in order that
 $\Ad^*(n_{\gamma})(\pr(X_{Y,\beta,a}))\in\mathfrak{l}^{\ast}_{\xi}+\tilde{\xi}$,
 it is necessary and sufficient that $a+2\gamma\beta^t=0$ and
$\beta(Y+2\gamma^t\beta-2\beta^t\gamma)^t=0$.
From these two equations, one solves that
\[\gamma=-\frac{1}{2|\beta|^2}(\beta Y^t+a\beta).\]
Then we have \[\Ad^*(n_{\gamma})(X_{Y,\beta,a})=
X_{Y-\frac{1}{|\beta|^2}(Y\beta^t\beta-\beta^t\beta Y^t),\,\beta,\,0}.\qedhere\]
\end{proof}

\begin{lemma}\label{L:class-matrix}
Assume that $\beta\neq 0$ and $Y\beta^{t}=0$. Then, the orbit $P\cdot\pr(X_{Y,\beta,0})$ is
determined by the class of $Z_{Y,\beta}$ with respect to the conjugation action of $\SO(m+1)$, where
\begin{equation*}
Z_{Y,\beta}=
\begin{pmatrix}Y&\frac{\beta^t}{|\beta|}\\
-\frac{\beta}{|\beta|}&0\\
\end{pmatrix}.
\end{equation*}
\end{lemma}

\begin{proof}
We assume $m$ is odd and $m=2n-1$.
Put
\[H'=
\begin{pmatrix}0&1\\ -1&0\\
\end{pmatrix}.\]
Then, there exists
$(W,a) \in \SO(m)\times \mathbb{R}_{>0}$ such that
\[WYW^{-1}=\diag\{x_{1}H',\dots,x_{n-1}H',0\}\]
 and
$a\beta W^t=(\underbrace{0,\dots,0}_{m-1},1)$,
 where $x_1\geq x_{2}\geq\cdots\geq x_{n-2}\geq |x_{n-1}|$.
By Lemma~\ref{L:p-standard3},
 the orbit $P\cdot \pr(X_{Y,\beta,0})$ is determined by the tuple $(x_1,\dots,x_{n-1})$.
Since $Z_{Y,\beta}$ is conjugate to
\[\diag\{x_{1}H',\dots,x_{n-1}H',H'\},\]
 the class of $Z_{Y,\beta}$ with respect to the conjugation action of $\SO(m+1)$
 is also determined by the tuple $(x_1,\dots,x_{n-1})$.
Hence, the conclusion of the proposition follows.

The case where $m$ is even is similar.
\end{proof}

Suppose that $m$ is odd and $m+1=2n$.
Then it is known that
 the $\SO(2n)$-conjugacy class of $Z_{Y,\beta}$ is determined by its singular values
and the sign of its Pfaffian (which can be $1$, $-1$ or $0$).
Here, the singular values of $Z_{Y,\beta}$
 mean the square roots of eigenvalues of $(Z_{Y,\beta})^t Z_{Y,\beta}$.
From the proof of Lemma \ref{L:class-matrix},
 we see that singular values of $Z_{Y,\beta}$ are
\[\{x_1,x_1,x_2,x_2,\dots,x_{n-1},x_{n-1},1,1\}\]
 and the singular values of $Y$ are
\[\{x_1,x_1,x_2,x_2,\dots,x_{n-1},x_{n-1},0\}.\]
The sign of the Pfaffian of $Z_{Y,\beta}$ is
 equal to the sign of $x_{n-1}$.

Next, suppose that $m$ is even and $m+1=2n-1$.
Then $Z_{Y,\beta}$ has an eigenvalue $0$
 and the singular values are
\[\{x_1,x_1,x_2,x_2,\dots,x_{n-2},x_{n-2},1,1,0\}\]
 with $x_1\geq x_2\geq \cdots \geq x_{n-2}\geq 0$.
The singular values of $Y$ are
\[\{x_1,x_1,x_2,x_2,\dots,x_{n-2},x_{n-2},0,0\}.\]
The $\SO(2n-1)$-conjugacy class of $Z_{Y,\beta}$
 is determined by the tuple $(x_1,\dots,x_{n-2})$.
The Pfaffian does not appear in this case.

\begin{remark}
It is easy to see that if a coadjoint orbit in $\mathfrak{p}^{\ast}$
 is strongly regular (see Section~\ref{S:introduction}), then it has depth one.
In the above notation the orbit $P\cdot\pr(X_{Y,\beta,0})$ is
 strongly regular if and only if $x_1>\cdots >x_{n-2}>|x_{n-1}|>0$ for odd $m$
 and $x_1>\cdots > x_{n-2}>0$ for even $m$.
\end{remark}

\section{Restriction to $P$ of irreducible representations of $G$}\label{S:resP}

\subsection{Moderate growth smooth representations of $G$ (or $P$)}\label{SS:CW-Cloux}

Let $\mathcal{C}_{K}(G)$ denote the category of Harish-Chandra modules, i.e., finitely generated admissible
$(\mathfrak{g},K)$-modules. For a $G$-representation $\pi$, let $\pi_{K}$ be the space of $K$-finite vectors in $\pi$.
Let $\mathcal{C}(G)$ denote the category of moderate growth, smooth Fr\'echet $G$-representations $\pi$ such that
$\pi_{K}\in\mathcal{C}_{K}(G)$.
The morphisms in $\mathcal{C}(G)$ are defined to be continuous intertwiners with images that are direct summands in
the category of Fr\'echet spaces.
The {\it Casselman-Wallach theorem} asserts that the functor
\[\mathcal{C}(G)\rightarrow \mathcal{C}_{K}(G),\quad  \pi\mapsto \pi_{K}\]
gives an equivalence of abelian categories.
For an object $V\in\mathcal{C}_{K}(G)$, write $V^{\sm}\in\mathcal{C}(G)$ for a
{\it Casselman-Wallach globalization} of $V$. Then $(V^{\sm})_{K}\cong V$.

In \cite{duCloux}, du Cloux studied the category of moderate growth, smooth Fr\'echet representations of a real
semi-algebraic group. We recall some results of \cite{duCloux} in our setting. Let $\mathcal{C}(P)$
(resp.\ $\mathcal{C}(M')$) denotes the category of moderate growth, smooth Fr\'echet representations of $P$
(resp.\ $M'$). The morphisms are continuous intertwiners.


Let $\mathscr{S}(\mathfrak{n})$ (resp.\ $\mathscr{S}(\mathfrak{n}^{\ast})$)
 be the Schwartz space on $\mathfrak{n}$ (resp.\ $\mathfrak{n}^{\ast}$)
 with the algebra structure by the convolution product
 (resp.\ by the usual multiplication) of functions.
The inverse Fourier transform gives an algebra isomorphism
 $\mathscr{S}(\mathfrak{n})\xrightarrow{\sim} \mathscr{S}(\mathfrak{n}^{\ast})$.
Let $\mathscr{S}(\mathfrak{n}^{\ast}-\{0\})$ be the Schwarz space on
$\mathfrak{n}^{\ast}-\{0\}$. In other words, it consists of $f|_{\mathfrak{n}^{\ast}-\{0\}}$ with
$f\in\mathscr{S}(\mathfrak{n}^{\ast})$ such that $f$ and its all (higher) derivatives vanish at
$0 (\in \mathfrak{n}^{\ast})$.

A representation $E\in \mathcal{C}(P)$ can be viewed as a moderate growth, smooth Fr\'echet representation of $N$
by restriction. Then via exponential map $\mathfrak{n}\cong N$, the Fr\'echet space $E$ becomes an
$\mathscr{S}(\mathfrak{n})$-module and then an $\mathscr{S}(\mathfrak{n}^{\ast}-\{0\})$-module by
$\mathscr{S}(\mathfrak{n}^{\ast}-\{0\})\subset \mathscr{S}(\mathfrak{n}^{\ast})\cong\mathscr{S}(\mathfrak{n})$.

We shall define a functor $\Psi\colon \mathcal{C}(P)\to \mathcal{C}(M')$ as follows. This functor is given as
$E\to E(x_0)$ in \cite[The\'or\`eme 2.5.8]{duCloux}. Recall that $\xi_0\in\mathfrak{n}^{\ast}-\{0\}$ is defined in
\S\ref{SS:notation} and the stabilizer of $\xi_0$ for the coadjoint action of $P$ on $\mathfrak{n}^{\ast}$
is $\Stab_{P}(\xi_0)=M'N$. Define the following algebra by adding the constant function $1$, which becomes the unit
of the algebra: \[\tilde{\mathscr{S}}(\mathfrak{n}^{\ast}-\{0\})=\bbC 1\oplus \mathscr{S}(\mathfrak{n}^{\ast}-
\{0\}).\]
Define an ideal $\mathfrak{m}_{\xi_0}$ by
 \[\mathfrak{m}_{\xi_0}=\{f\in \tilde{\mathscr{S}}(\mathfrak{n}^{\ast}-\{0\}) : f(\xi_0)=0\}.\]
For $E\in \mathcal{C}(P)$,
\cite[Lemme 2.5.7]{duCloux} shows that
 the subspace $\mathfrak{m}_{\xi_0}\cdot E$ is closed
 and stable by the action of $M'N$.
Hence the quotient $E/(\mathfrak{m}_{\xi_0}\cdot E)$
 is a Fr\'echet space with a natural $M'N$-action on it.
The action of $\tilde{\mathscr{S}}(\mathfrak{n}^*-\{0\})$ on $E/(\mathfrak{m}_{\xi_0}\cdot E)$
 factors through the evaluation map
 $\tilde{\mathscr{S}}(\mathfrak{n}^*-\{0\})\ni f \mapsto f(\xi_0)$.
Hence $N$ acts on $E/(\mathfrak{m}_{\xi_0}\cdot E)$ by $e^{\mathbf{i}\xi_0}$.
When we view $E/(\mathfrak{m}_{\xi_0}\cdot E)$ as a representation of $M'$ we write
\begin{equation*}
\Psi(E):=E/(\mathfrak{m}_{\xi_0}\cdot E).
\end{equation*}
By \cite[The\'or\`eme 2.5.8]{duCloux},
 $\Psi(E)\in \mathcal{C}(M')$ and then $\Psi$ defines a
functor $\mathcal{C}(P)\to \mathcal{C}(M')$.

Let $F$ be a finite-dimensional representation of $P$ such that the $N$-action is trivial. Then it is easy to see that
there is a natural isomorphism
\begin{equation}\label{Jtensor}
\Psi(E)\otimes (F|_{M'}) \cong \Psi(E\otimes F).
\end{equation}


Next, we define the induction from $M'$-representations to $P$-representations.
Let $(\tau,V_{\tau})\in \mathcal{C}(M')$.
Then $\tau\otimes e^{\mathbf{i}\xi_{0}}$ is a smooth Fr\'echet representation of $M'N$.
One defines in a natural way the smoothly induced representation
 $C^{\infty}\Ind_{M'N}^{P}(\tau\otimes e^{\mathbf{i}\xi_{0}})$.
Let $\mathscr{S}(P,V_{\tau})$ be the space of Schwartz functions on $P$
 taking values in $V_{\tau}$.
For $f\in\mathscr{S}(P,V_{\tau})$, define $\bar{f}\in C^{\infty}(P,V_{\tau})$ by
\[\bar{f}(g)=\int_{M'N}(\tau\otimes e^{\mathbf{i}\xi_{0}})(mn)f(gmn)\d m\d n.\]
Then one has $\bar{f}\in C^{\infty}\Ind_{M'N}^{P}(V_{\tau})$. Let
\[\mathscr{S}\Ind_{M'N}^{P}(\tau\otimes e^{\mathbf{i}\xi_{0}})=
\{\bar{f}:f\in\mathscr{S}(P,V_{\tau})\}.\] Then
$\mathscr{S}\Ind_{M'N}^{P}(\tau\otimes  e^{\mathbf{i}\xi_0})$ is a dense
subspace of $C^{\infty}\Ind_{M'N}^{P}(\tau\otimes  e^{\mathbf{i}\xi_0})$
 and $\mathscr{S}\Ind_{M'N}^{P}
(\tau\otimes  e^{\mathbf{i}\xi_0})\in\mathcal{C}(P)$.

Let \[\mathscr{O}_{M}(P,\tau\otimes e^{\mathbf{i}\xi_{0}})=\{f\in C^{\infty}(P,V_{\tau}): h\cdot f\in\mathscr{S}(P,V_{\tau})
\ (\forall h\in\mathscr{S}(P))\}.\]
Let
\begin{align*}
&\mathscr{O}_{M}\Ind_{M'N}^{P}(\tau\otimes e^{\mathbf{i}\xi_{0}}) \\
& =\{f\in\mathscr{O}_{M}(P,\tau\otimes e^{\mathbf{i}\xi_0}):
 f(gmn)=(\tau\otimes e^{\mathbf{i}\xi_{0}})(mn)^{-1}f(g)\ (\forall g\in P,mn\in M'N)\}.
\end{align*}
This is not a Fr\'echet space, but $P$ naturally acts on it.
Then
\[\mathscr{S}\Ind_{M'N}^{P}
(\tau\otimes e^{\mathbf{i}\xi_0})\subset\mathscr{O}_{M}\Ind_{M'N}^{P}(\tau\otimes e^{\mathbf{i}\xi_0})\subset C^{\infty}\Ind_{M'N}^{P}(\tau\otimes e^{\mathbf{i}\xi_0}).\]
Since $P/(M'N)\cong \mathfrak{n}^*-\{0\}$,
 these three spaces become $\mathscr{S}(\mathfrak{n}^*-\{0\})$-modules by multiplication.


Since $N$ is nilpotent and $M'$ is compact, the group $M'N$ is unimodular and the restriction to $M'N$
of the modulus character of $P$ is also trivial. Let $E\in\mathcal{C}(P)$. The natural map $E\to\Psi(E)$
is $M'$-intertwining and corresponds to a $P$-intertwiner
\[u\colon E\to C^{\infty}\Ind_{M'N}^{P}(\Psi(E)\otimes e^{\mathbf{i}\xi_{0}})\] by the Frobenius
reciprocity. The following is a part of \cite[Th\'eor\`eme 2.5.8]{duCloux} applying to the group $P$.

\begin{fact}\label{F:duCloux}
Let $E\in\mathcal{C}(P)$ and let $u\colon E\to C^{\infty}\Ind_{M'N}^{P}(\Psi(E)\otimes e^{\mathbf{i}\xi_{0}})$
be as above. Then \begin{align*}
&\mathscr{S}\Ind_{M'N}^{P}(\Psi(E)\otimes e^{\mathbf{i}\xi_{0}})\subset\Im(u)\subset
\mathscr{O}_{M}\Ind_{M'N}^{P}(\Psi(E)\otimes e^{\mathbf{i}\xi_{0}}),\text{ and }\\&\Ker(u)=
 \{v\in E : \mathscr{S}(\mathfrak{n}^*-\{0\}) \cdot v = 0\}.
\end{align*}
\end{fact}

We need several lemmas below.

\begin{lemma}\label{L:Frobenius}
Let $E\in\mathcal{C}(P)$ and $W\in \mathcal{C}(M')$. Let
$\varphi\colon E\hookrightarrow C^{\infty}\Ind_{M'N}^{P}(W\otimes e^{\mathbf{i}\xi_{0}})$
be an injective $P$-intertwining map that is also a homomorphism of $\mathscr{S}(\mathfrak{n}^*-\{0\})$-modules.
Then the kernel of the map $\bar{\varphi}\colon E\to W$ given by $\bar{\varphi}(v)=(\varphi(v))(e)$
equals $\mathfrak{m}_{\xi_0}\cdot E$.
\end{lemma}

\begin{proof}
Take any $f\in \mathfrak{m}_{\xi_0}$ and $v\in E$. Since $\phi$ is an $\mathscr{S}(\mathfrak{n}^{\ast}-\{0\})$-homomorphism,
$\phi(fv)=f\phi(v)$ and $\bar{\phi}(fv)=f(\xi_0)\bar{\phi}(v)$. Hence, $\Ker(\bar{\varphi})\supset
\mathfrak{m}_{\xi_0}\cdot E$. Then, $\bar{\varphi}$ descends to $\bar{\varphi}\colon \Psi(E)\to W$.
The map $\varphi$ factors as
\[E\xrightarrow{u} C^{\infty}\Ind_{M'N}^{P}(\Psi(E)\otimes e^{\mathbf{i}\xi_{0}})
 \to C^{\infty}\Ind_{M'N}^{P}(W\otimes e^{\mathbf{i}\xi_{0}}).\]
By Fact~\ref{F:duCloux},
 $\Im(u)\supset \mathscr{S}\Ind_{M'N}^{P}(\Psi(E)\otimes e^{\mathbf{i}\xi_{0}})$
 and then
$\mathscr{S}\Ind_{M'N}^{P}(\Psi(E)\otimes e^{\mathbf{i}\xi_{0}})
 \to C^{\infty}\Ind_{M'N}^{P}(W\otimes e^{\mathbf{i}\xi_{0}})$
 is injective.
Therefore, $\bar{\varphi}\colon \Psi(E)\to W$ is also injective.
\end{proof}

\begin{lemma}\label{L:exact}
Let $0\to E_1 \to E_2 \to E_3 \to 0$ be a sequence
 in $\mathcal{C}(P)$ which is exact as vector spaces.
Then the induced sequence $0\to \Psi(E_1) \to \Psi(E_2) \to \Psi(E_3) \to 0$
 in $\mathcal{C}(M')$ is also exact as vector spaces.
\end{lemma}

\begin{proof}
It is easy to see that
 $\Psi(E_1) \to \Psi(E_2) \to \Psi(E_3) \to 0$ is exact.

Assuming $E_1 \hookrightarrow E_2$ is an injective homomorphism in $\mathcal{C}(P)$,
 we will show that $\Psi(E_1) \to \Psi(E_2)$ is injective.
By Fact~\ref{F:duCloux},  we have
\begin{align*}
&u_i\colon E_i\to \mathscr{O}_{M}\Ind_{M'N}^{P}(\Psi(E_i)\otimes e^{\mathbf{i}\xi_{0}}),\\
&\Ker(u_i)=\{v\in E_i :
\mathscr{S}(\mathfrak{n}^*-\{0\}) \cdot v = 0\},\\
&\mathscr{S}\Ind_{M'N}^{P}(\Psi(E_i)\otimes e^{\mathbf{i}\xi_{0}})\subset\Im(u_i)\quad (i=1,2).
\end{align*}
By this description of $\Ker(u_i)$,
 we have $\Ker(u_1)=E_1\cap \Ker(u_2)$ and hence
 the natural map $\Im(u_1)\to \Im(u_2)$ is injective.
By composing
\[\mathscr{S}\Ind_{M'N}^{P}(\Psi(E_1)\otimes e^{\mathbf{i}\xi_{0}})\subset\Im(u_1)
\hookrightarrow \Im(u_2)
\subset \mathscr{O}_{M}\Ind_{M'N}^{P}(\Psi(E_2)\otimes e^{\mathbf{i}\xi_{0}}),\]
we obtain an injective map \[\mathscr{S}\Ind_{M'N}^{P}(\Psi(E_1)\otimes e^{\mathbf{i}\xi_{0}})
\hookrightarrow \mathscr{O}_{M}\Ind_{M'N}^{P}(\Psi(E_2)\otimes e^{\mathbf{i}\xi_{0}}).\]
This is induced from the map $\Psi(E_1)\to \Psi(E_2)$, which must be also injective.
\end{proof}

For $E\in \mathcal{C}(P)$, $\Psi(E)$ is the maximal Hausdorff quotient of $E$ on which $\mathfrak{n}$ acts by
$\mathbf{i}\xi_0$ in the following sense. Let $F$ be the linear span of the set
\[\{X\cdot v - \mathbf{i}\xi_0(X) v \mid v\in E,\ X\in\mathfrak{n}\}\] and let $F^{\operatorname{cl}}$ be the
closure of $F$ in $E$. Then $F^{\operatorname{cl}}$ is closed by the $M'N$-action and
\begin{equation}\label{Eq:max_quotient}
\Psi(E)\cong E/F^{\operatorname{cl}}.
\end{equation} Take any $f\in \mathfrak{m}_{\xi_0}$ and $v\in E$ and consider the inverse Fourier transform
$\mathcal{F}(f)$ of $f$. Calculate the vector $fv$ by the definition of $\mathscr{S}(\mathfrak{n})$-action on $E$, i.e. the
integration of $\mathcal{F}(f)(n)nv$ over $n\in N$. Since the projection $p\colon E\to E/F^{cl}$ respects the $N$-action and $N$
acts by $e^{\mathbf{i}\xi_0}$ on $E/F^{cl}$, we have \[p(fv)=\int_{N} \mathcal{F}(f)(n)e^{\mathbf{i}(\xi_0,n)}p(v) dn=f(\xi_0)p(v)=0.\]
Then, the map $E\to E/F^{\operatorname{cl}}$ factors through
$\Psi(E)\to E/F^{\operatorname{cl}}$. On the other hand, since $\mathfrak{n}$ acts on $E/(\mathfrak{m}_{\xi_0}\cdot E)$
by $\mathbf{i}\xi_0$, we get a map $E/F^{\operatorname{cl}}\to\Psi(E)$, which is the inverse of the above map.
Thus \eqref{Eq:max_quotient} follows.

For $V\in \mathcal{C}(G)$, write $V|_P\in \mathcal{C}(P)$ for the representation obtained by restriction of the action
of $G$ to $P$.

\begin{lemma}\label{L:tensor}
Let $V\in \mathcal{C}(G)$ and
 let $F$ be a finite-dimensional representation of $P$.
Then there exists an isomorphism of $M'$-representations:
\[\Psi(V|_P)\otimes (F|_{M'}) \cong \Psi((V\otimes F)|_P).\]
\end{lemma}

\begin{proof}
There exists a filtration $0=F_0\subset F_1\subset \cdots \subset F_k =F$
 of $P$-subrepresentaions such that $N$ acts trivially on $F_i/F_{i-1}$.
Then \eqref{Jtensor} and Lemma~\ref{L:exact} imply
\[\Psi(V|_P)\otimes (F_i |_{M'}) \cong \Psi(V|_P\otimes F_i)\] inductively and
we obtain the conclusion of the lemma.
\end{proof}

\begin{remark}\label{R:CHM}
For $V\in \mathcal{C}(G)$, the dimension of $\Psi(V|_P)$ is finite and
 $\Psi(V|_P)\cong H_0(\mathfrak{n}, V\otimes e^{-\mathbf{i}\xi_0})$,
 namely, the subspace $F$ defined above \eqref{Eq:max_quotient} is closed.
This is proved in \cite[\S8]{CHM}.
The exactness of $\Psi$ for representations of $G$ is also proved there.
Vectors in the dual space of $\Psi(V|_P)$ is no other than Whittaker vectors.
\end{remark}

We will apply the above lemmas to study the restriction of unitary representations.
Let $V$ be a non-trivial irreducible
unitarizable $(\mathfrak{g},K)$-module.
Write $\bar{V}$ for its Hilbert space completion and $V^{\sm}$ for the
Casselman-Wallach globalization. By Proposition~\ref{P:Mackey}, an irreducible unitary representation of $P$ is either
finite-dimensional or equal to $I_{P,\tau}$ for an irreducible representation $\tau$ of $M'$. In the former case, it
factors through $P/N (\cong MA)$. Hence these are parametrized by irreducible unitary representations
$\sigma\otimes e^{\nu}$ of $MA$. Then by general theory, the restriction of $\bar{V}$ to $P$ decomposes into
irreducibles as \begin{align*}
\bar{V}|_P \cong \int^{\oplus} (\sigma\otimes e^\nu)^{m(\sigma,\nu)} d\mu
 \oplus \bigoplus_{\tau} (I_{P,\tau})^{m(\tau)}.
\end{align*}
Here, the first term on the right hand side is a direct integral of irreducible unitary representations and the second
term is a Hilbert space direct sum. We will show that actually the first term on the right hand side does not appear and
the second term is a finite sum. Since any vector $v\in\int^{\oplus}(\sigma\otimes e^\nu)^{m(\sigma,\nu)}d\mu$ is
$N$-invariant, \[V\ni v' \mapsto (v',v) \in \bbC \] defines an $\mathfrak{n}$-invariant vector of the algebraic dual space $V^*$.
Since it is known that $H^0(\mathfrak{n},V^*)$ is finite-dimensional (\cite[Corollary 2.4]{Casselman-Osborne}),
$\int^{\oplus} (\sigma\otimes e^\nu)^{m(\sigma,\nu)} d\mu$ is also finite-dimensional and in particular it only has a
discrete spectrum. Suppose that $\sigma\otimes \nu$ appears in $\bar{V}|_P$ as a direct summand. Then by the Frobenius
reciprocity, we obtain an intertwining map \[V\hookrightarrow \Ind_P^G (\sigma\otimes e^{\nu}\otimes\mathbf{1}_N)
\bigl(\cong \Ind_{\bar{P}}^{G}(\sigma\otimes e^{-\nu}\otimes\mathbf{1}_{\bar{N}})=I(\sigma,-\nu+\rho')\bigr).\]
Since $\nu\in \mathbf{i}\mathfrak{a}^*$, $V$ is isomorphic to the unique irreducible subrepresentation of $I(\sigma,-\nu+\rho')$.
By considering leading exponent of matrix coefficients of $V$, \cite[Theorem 9.1.4]{Collingwood} (which in turn is
implied by a theorem of Howe-Moore in \cite{Howe-Moore}) implies that $\nu=0$ and $V$ is trivial, which is not the
case. Hence, $\int^{\oplus}(\sigma\otimes e^\nu)^{m(\sigma,\nu)}d\mu=0$.

Let $\bar{\tau}:=\sum^{\oplus}_{\tau} \tau^{m(\tau)}$ be the Hilbert direct sum.
Let $\Ind_{M'N}^P(\bar{\tau}\otimes e^{\mathbf{i}\xi_0})$
be the unitarily ($L^2$-)induced representation. Then
\[\bar{V}|_P \cong \bigoplus_{\tau} (I_{P,\tau})^{m(\tau)}
 \cong \Ind_{M'N}^P(\bar{\tau}\otimes e^{\mathbf{i}\xi_{0}}).\]
Let $\bar{\tau}^{\infty}$ be the
 set of smooth vectors in $\bar{\tau}$ as a representation of $M'$.
Then by the Sobolev embedding theorem (on $P/M'N\cong\mathbb{R}^{m}-\{0\}$),
 the smooth vectors in $\bar{V}$ lies in
 $C^{\infty}\Ind_{M'N}^P(\bar{\tau}^{\infty} \otimes e^{\mathbf{i} \xi_{0}})$.
Hence we obtain an injective $P$-intertwining map
\[V^{\sm}\hookrightarrow
 C^{\infty}\Ind_{M'N}^P(\bar{\tau}^{\infty}\otimes e^{\mathbf{i}\xi_{0}}).\]
Now we apply Lemma~\ref{L:Frobenius}
 and use the denseness of $V^{\sm}$ in $\bar{V}$,
 we conclude that
 \[\Psi(V^{\sm}|_P)\cong\bar{\tau}^{\infty}.\]
By Remark~\ref{R:CHM} or Proposition~\ref{P:res-induced} shown later,
 $\Psi(V^{\sm}|_P)$ is always finite-dimensional.
Therefore, $\bar{\tau}^{\infty}=\bar{\tau}$. We thus obtain the following.

\begin{lemma}\label{L:unitary-J}
Suppose that $V$ is a non-trivial irreducible unitarizable $(\mathfrak{g},K)$-module. Then
\begin{align*}
\bar{V}|_{P}\cong \Ind_{M'N}^{P}(\Psi(V^{\sm}|_P)\otimes e^{\mathbf{i}\xi_{0}}).
\end{align*}
\end{lemma}

\subsection{Restrictions of principal series representations}\label{SS:ResPS}

We calculate $\Psi(I(\sigma,\nu)|_P)$, where $I(\sigma,\nu)$ is a (not necessarily unitary) principal series
representation of $G$.

\begin{proposition}\label{P:res-induced}
Let $I(\sigma,\nu)=\Ind_{\bar{P}}^{G}(V_{\sigma}\otimes e^{\nu-\rho'}\otimes \mathbf{1}_N)$
 be a principal series representation.
Then \[\Psi(I(\sigma,\nu)|_P)\cong V_{\sigma}|_{M'}.\]
\end{proposition}

\begin{remark}
Proposition~\ref{P:res-induced} is proved in \cite[Lemma 8.5]{CHM}
 by using the Bruhat filtration.
However, we include the proof here because the argument of the Fourier transform
 below will be used for concrete calculations in Appendix~\ref{S:trivial}.
\end{remark}

In order to prove the proposition, we consider the restriction of the functions
 $f\in I(\sigma,\nu)$ to $N$ and take inverse Fourier transform.

For $f\in I(\sigma,\nu)$, let $f_{N}=f|_{N}$. We have the map \[I(\sigma,\nu)\to C^{\infty}(N,V_{\sigma}),
\quad f\mapsto f_{N}.\] The action of $P=MAN$ on $I(\sigma,\nu)$ is compatible with the following $P$-action
on $C^{\infty}(N,V_{\sigma})$: for $F\in C^{\infty}(N,V_{\sigma})$ and $n\in N$,
\begin{align}\label{Eq:P-action} \nonumber
&(n'\cdot F)(n)=F(n'^{-1}n)\quad (n'\in N); \\
&(a\cdot F)(n)=e^{(\nu-\rho')\log a}F(a^{-1}na)\quad (a\in A); \\ \nonumber
&(m_0\cdot F)(n)=\sigma(m_0)F(m_0^{-1}nm_0)\quad (m_0\in M).
\end{align}

Next, define the inverse Fourier transform of $f_N \in C^{\infty}(N,V_{\sigma})$.
If $f_N$ is $L^1$, then its inverse Fourier transform is defined as
 a function on $\mathfrak{n}^*$ as
\begin{equation}\label{Eq:Ff}
\widehat{f_{N}}(\xi)=\mathcal{F}(f_{N})(\xi)=
(2\pi)^{-\frac{m}{2}}\int_{\bbR^{m}}e^{\mathbf{i}(\xi,x)}f(n_{x})\d x.
\end{equation}
In general, for $f\in I(\sigma,\nu)$,
 the function $f_{N}$ on $N\cong \bbR^m$ is of at most polynomial growth at infinity.
Hence $f_{N}$ is a tempered distribution.
The map \eqref{Eq:Ff} extends for tempered distributions and
 we obtain $\widehat{f_{N}}(\xi)$ as distributions on $\mathfrak{n}^*$.
The action of $G$ on the Fourier transformed picture is defined as
\begin{equation*}
g(\widehat{f_{N}})=\widehat{(gf)_{N}}
\end{equation*}
 for $f\in I(\sigma,\nu)$ and $g\in G$.
Then the $P$-action is given as follows:
 for $f\in I(\sigma,\nu)$ and $\xi\in\mathfrak{n}^{\ast}$,
\begin{align}\label{Eq:P-action2}\nonumber
&(n_{x}\cdot \widehat{f_{N}})(\xi)
 =e^{\mathbf{i}(\xi,x)}\widehat{f_{N}}(\xi)\quad (x\in \bbR^{m});\\
&(a\cdot \widehat{f_{N}})(\xi)=e^{(\nu+\rho')\log a}
\widehat{f_{N}}(\Ad^{\ast}(a^{-1})\xi) \quad (a\in A);\\ \nonumber
&(m_0\cdot \widehat{f_{N}})(\xi)=\sigma(m_0)\widehat{f_{N}}(\Ad^{\ast}(m_0^{-1})\xi)
 \quad (m_0\in M).
\end{align}

\begin{lemma}\label{L:FourierSmooth}
Let $f\in I(\sigma,\nu)$.
Then the restriction of $\widehat{f_{N}}$ to
 $\mathfrak{n}^*-\{0\}$ is a $C^{\infty}$-function.
\end{lemma}

\begin{proof}
We first prove a similar claim for the group $G_2=\OO(m+1,1)$.
Let $I(\sigma,\nu)$ be a principal series representation of $G_2$
 for an irreducible representation $\sigma$ of $M_2$ and
 take $f\in I(\sigma,\nu)$.
To prove that $\widehat{f_N}|_{\mathfrak{n}^*-\{0\}}$ is a smooth function,
 we need to see the behavior of $f(n_x)$ as $x\to \infty$.
This is equivalent to the behavior of $f(s n_x)$ near $x=0$,
 where $s=\diag\{I_m,-1,1\}$.
Put $F(x):=f(sn_x)$ for $x\in \bbR^n$.
By Lemma~\ref{L:barn-iwasawa},
\[
F(x)=f\bigl(n_{\frac{x}{|x|^2}}r_x e^{-(2\log |x|)H_0}\bar{n}_{\frac{x}{|x|^2}}\bigr)
= |x|^{2(\nu-\rho')(H_0)} \sigma(r_x) f\bigl(n_{\frac{x}{|x|^2}}\bigr).
\]
Since $F$ is smooth, $|f(n_x)|$ is bounded by $C|x|^{2(-\nu+\rho')(H_0)}$
 as $x\to \infty$ for some constant $C>0$.

The $G$-action on $I(\sigma,\nu)$ differentiates to the $\mathfrak{g}$-action.
Take $X_y\in \mathfrak{n}$ for $y\in \bbR^m$ and consider the function
 $X_y\cdot f\in I(\sigma,\nu)$.
We have
\[(X_y\cdot f)(sn_x) = \frac{d}{dt}\Bigl|_{t=0} f(n_{ty}^{-1} s n_x)\]
By Lemma~\ref{L:barn-iwasawa} again,
\[n_{ty}^{-1} s n_x
=n_{-ty+\frac{x}{|x|^2}} r_x e^{-(2\log |x|)H_0}\bar{n}_{\frac{x}{|x|^2}}.\]
Putting $z:=-ty+\frac{x}{|x|^2}$, we have
\begin{align*}
&\quad n_{-ty+\frac{x}{|x|^2}} r_x e^{-(2\log |x|)H_0}\bar{n}_{\frac{x}{|x|^2}}\\
&=s n_{\frac{z}{|z|^2}} r_z e^{-(2\log |z|)H_0}\bar{n}_{\frac{z}{|z|^2}}
r_x e^{-(2\log |x|)H_0}\bar{n}_{\frac{x}{|x|^2}}\\
&\in s n_{\frac{z}{|z|^2}} r_z r_x e^{-(2\log |z|+2\log |x|)H_0}\bar{N}.
\end{align*}
Hence
\[
(X_y\cdot f)(sn_x)
 = \frac{d}{dt}\Bigl|_{t=0} |z|^{2(\nu-\rho')(H_0)} |x|^{2(\nu-\rho')(H_0)}
 \sigma(r_z r_x) F\Bigl(\frac{z}{|z|^2}\Bigr).
\]
Note that $r_z=r_{-t|x|^2 y + x}$.
We calculate
\begin{align*}
&\frac{d}{dt}\Bigl|_{t=0} |z|^{2(\nu-\rho')(H_0)} |x|^{2(\nu-\rho')(H_0)}
= - 2(\nu-\rho')(H_0) (y,x), \\
&\frac{d}{dt}\Bigl|_{t=0} r_z r_x
= 2 (y^t x - x^t y), \\
&\frac{d}{dt}\Bigl|_{t=0} F\Bigl(\frac{z}{|z|^2}\Bigr)
= \bigl(2 (x,y)x - y|x|^2\bigr) (\nabla_{y} F)(x).
\end{align*}
Combining above equations, we see that if $F(x)$ vanishes at $x=0$ of order $k$,
 then $(X_y\cdot f)(sn_x)$ vanishes at $x=0$ of order $k+1$.
Hence $(X_y\cdot f)(n_x)$ is bounded by $C|x|^{2(-\nu+\rho')(H_0)-1}$ for some $C$.
By repeating this, $(X_y^l\cdot f)(n_x)$ is bounded by $C|x|^{2(-\nu+\rho')(H_0)-l}$ for $l>0$.
Then for any $k>0$, if $l$ is sufficiently large,
 then $P(x)(X_y^l \cdot f)(n_x)$ is in $L^1$ for any polynomial $P(x)$ of degree $k$.
Therefore, its inverse Fourier transform is continuous.
By
\begin{align}
\label{Eq:Mult-Differ}
\mathcal{F}(x_{j}h)=-\mathbf{i}\partial_{\xi_j} \mathcal{F}(h),\quad
\mathcal{F}(\partial_{x_j}h) =-\mathbf{i}\xi_j \mathcal{F}(h),
\end{align}
we have
\[
\mathcal{F} (P(x)(X_y^l \cdot f)(n_x))
= P(-\mathbf{i}\partial_{\xi}) (-\mathbf{i}\xi,y)^l \cdot \widehat{f_N}(\xi).
\]
Hence $\widehat{f_N}(\xi)$ is $C^k$ in $(\xi,y)\neq 0$.
Since $k$ and $y$ are arbitrary, we proved that $\widehat{f_N}$ is $C^{\infty}$ on
 $\mathfrak{n}^*-\{0\}$.

To prove the claim for $G=\Spin(m+1,1)$, fix $m_0\in M_2$ such that
 $m_0s\in M_1=\SO(m)$ and use a lift of $m_0s$ in $M=\Spin(m)$ instead of $s$
 in the above argument.
\end{proof}

Recall $\xi_0=(0,\dots,0,1)\in \mathfrak{n}^*$.
For $h\in C^{\infty}(\mathfrak{n}^{\ast}-\{0\},V_{\sigma})$,
 define a function $h_{at,\nu}$ on $P$ by
\begin{align*}
h_{at,\nu}(p)
=(p^{-1}\cdot h)(\xi_0) \quad (p\in P).
\end{align*}
More concretely,
\begin{align}\label{Eq:anti-trivialization}
&\quad h_{at,\nu}(p)\\ \nonumber
&=e^{-\mathbf{i}(\xi_0,x)}
 e^{(-\nu-\rho')\log a}\sigma(m_0)^{-1}h(\Ad^*(m_0a)\xi_{0}) \\ \nonumber
&=e^{-\mathbf{i}(\xi_0,x)}
 |\Ad^*(m_0a)(\xi_{0})|^{\frac{2\nu(H_0)+m}{2}}
 \sigma(m_0)^{-1}h(\Ad^*(m_0a)\xi_{0})
\end{align}
for $p=m_0an_{x}\in P$.
We call $h_{at,\nu}$ the anti-trivialization of $h$.
The term `anti-trivialization' comes from:
 $h$ is a function on $\mathfrak{n}^{\ast}-\{0\}$,
 i.e., a section of the trivial bundle on $P/M'N\cong\mathfrak{n}^{\ast}-\{0\}$, and $h_{at,\nu}$
 is a section of the vector bundle
 $P\times_{M'N}(\sigma|_{M'}\otimes e^{\mathbf{i} \xi_{0}})$
 on $P/M'N$.

\begin{lemma}\label{L:anti-trivialization}
The image of the map
 $C^{\infty}(\mathfrak{n}^{\ast}-\{0\},V_{\sigma})\ni h\mapsto h_{at,\nu}$
 is equal to the representation space of the smoothly induced representation
 $C^{\infty}\Ind_{M'N}^{P}(\sigma|_{M'}\otimes e^{\mathbf{i}\xi_{0}})$.
The map $h\mapsto h_{at,\nu}$ respects the actions of $P$
 and $\mathscr{S}(\mathfrak{n}^*-\{0\})$.
\end{lemma}

\begin{proof}
For any $m'\in M'$, $n_{x}\in N$ and $p\in P$, we have
\begin{align*}
&\quad h_{at,\nu}(pm'n_{x})\\
&=(n_{x}^{-1}(m')^{-1}p^{-1}\cdot h)(\xi_0)\\
&=e^{-\mathbf{i}(\xi_0,x)}
 \sigma(m')^{-1}(p^{-1}\cdot h)(\Ad^*(m')\xi_{0})\\
&=(\sigma\otimes e^{\mathbf{i} \xi_{0}})(m',n_{x})^{-1}h_{at,\nu}(p),
\end{align*}
 where we used $\Ad^*(m')\xi_0=\xi_0$.
This shows that $h_{at,\nu}$ is a section of the vector bundle
 $P\times_{M'N}(\sigma|_{M'}\otimes  e^{\mathbf{i}\xi_{0}})$.

It directly follows from the definition of $h_{at,\nu}$ that
 the map $h\mapsto h_{at,\nu}$ respects the $P$-actions.
The actions of $\mathscr{S}(\mathfrak{n}^*-\{0\})$ on
 $C^{\infty}\Ind_{M'N}^{P}(\sigma|_{M'}\otimes e^{\mathbf{i}\xi_{0}})$
 and $C^{\infty} (\mathfrak{n}^{\ast}-\{0\},V_{\sigma})$
 are given by multiplications.
Hence the map $h\mapsto h_{at,\nu}$ is a $\mathscr{S}(\mathfrak{n}^*-\{0\})$-homomorphism.

The inverse map
\[C^{\infty}\Ind_{M'N}^{P}(\sigma|_{M'}\otimes e^{\mathbf{i}\xi_{0}})
\to C^{\infty} (\mathfrak{n}^{\ast}-\{0\},V_{\sigma}),
\quad h' \mapsto h'_{t,\nu}\]
is given as follows:
 for any $\xi\in \mathfrak{n}^{\ast}-\{0\}$, choose
 $m_0a\in MA$ such that $\xi=\Ad^*(m_0a)\xi_0$ and define
\[h'_{t,\nu}(\xi)=|\xi|^{-\frac{m+2\nu(H_0)}{2}}\sigma(m_0)h'(m_0a).\]
It is easy to see that the maps $h\mapsto h_{at,\nu}$
 and $h'\mapsto h'_{t,\nu}$ are inverse to each other.
\end{proof}

\begin{proof}[Proof of Proposition~\ref{P:res-induced}]
By Lemmas~\ref{L:FourierSmooth} and \ref{L:anti-trivialization},
 we obtain a map
\[
\varphi\colon
 I(\sigma,\nu) \to C^{\infty}\Ind_{M'N}^{P}(\sigma|_{M'}\otimes e^{\mathbf{i}\xi_{0}}),
\quad f\mapsto (\widehat{f_N})_{at,\nu}
\]
which respects the actions of $P$ and $\mathscr{S}(\mathfrak{n}^*-\{0\})$.
If $f\in \Ker \varphi$, then the support of $\widehat{f_N}$ is contained in $\{0\}$,
 or equivalently, $f_N$ is a polynomial.
Hence $\mathscr{S}(\mathfrak{n}^*-\{0\})\cdot (\Ker \varphi)=0$
 and $\Psi(I(\sigma,\nu)|_P)\cong \Psi(I(\sigma,\nu)|_P/\Ker \varphi)$.
Then by Lemma~\ref{L:Frobenius}, there exists an injective map
 $\Psi(I(\sigma,\nu)|_P)\to \sigma|_{M'}$.
To show the surjectivity, take any vector $v\in \sigma|_{M'}$ and take a function
$h\in \mathscr{S}\Ind_{M'N}^{P}(\sigma|_{M'}\otimes e^{\mathbf{i}\xi_{0}})$
(or a function $h\in C^{\infty}\Ind_{M'N}^{P}(\sigma|_{M'}\otimes e^{\mathbf{i}\xi_{0}})$ which is
compactly supported modulo $M'N$) such that $h(e)=v$. Then there exists $f\in I(\sigma,\nu)$ such that
$(\widehat{f_N})_{at,\nu}=h$, which implies that the map $\Psi(I(\sigma,\nu)|_P)\to \sigma|_{M'}$ is surjective.
\end{proof}

Using Lemma \ref{L:unitary-J} and the determination of the restriction to $P$ of irreducible unitary representations
of $G$ with trivial infinitesimal character and some complementary series in Appendix \ref{S:trivial}, one can show
Proposition~\ref{P:res-induced} using translation principle similarly as that in Theorem \ref{T:branching-ds}. However, it
is much more complicated than the above proof.

\smallskip

By Casselman's subrepresentation theorem, any moderate growth,
 irreducible admissible smooth Fr\'echet representation $\pi$ of $G$ is
 a subrepresentation of a principal series representation $I(\sigma,\nu)$.
Then by Lemma~\ref{L:exact} and Proposition~\ref{P:res-induced},
 $\Psi(\pi|_P)\subset \Psi(I(\sigma,\nu)|_P)\cong\sigma|_{M'}$.
In particular, $\Psi(\pi|_P)$ is finite-dimensional.

Let $K(G)$ (resp.\ $K(M')$) be the Grothendieck group of
 the category of Harish-Chandra modules
 (resp.\ the category of finite-dimensional representations of $M'$).
By Lemma~\ref{L:exact},
 $\mathcal{C}(G)\ni \pi \mapsto \Psi(\pi|_{P})$ induces a homomorphism
 $\Psi \colon K(G)\to K(M')$.

\subsection{Classification of irreducible representations of $G$}\label{SS:irreducible}

In this subsection we recall the classification of irreducible
 admissible representations $\pi\in \mathcal{C}(G)$.

Suppose first that $m$ is odd and then $m=2n-1$ and $G=\Spin(2n,1)$.
The infinitesimal character $\gamma$ of $\pi$ is conjugate to
\[(\mu+\rho_{M},\nu)=\Bigl(a_{1}+n-\frac{3}{2},\cdots,a_{n-1}+\frac{1}{2},a\Bigr),\]
 where $\mu=(a_{1},\dots,a_{n-1})$ and $\nu=a\lambda_{0}$.
We have $a_{1}\geq \cdots\geq a_{n-1}\geq 0$;
 and $a_{1},\dots,a_{n-1}$ are all integers or all half-integers.
The weight $\gamma$ is integral if and only if $a-(a_{j}+\frac{1}{2})\in\mathbb{Z}$.
The singularity of integral $\gamma$ has several possibilities:
\begin{enumerate}
\item If $a\neq a_{j}+n-j-\frac{1}{2}$ for $1\leq j\leq n-1$ and $a\neq 0$,
 then $\gamma$ is regular.
Write $\Lambda_{0}$ for the set of integral regular dominant weights.
\item If $a=a_{j}+n-j-\frac{1}{2}$ for some $1\leq j \leq n - 1$, then up to conjugation
\[\gamma=\Bigl(a_{1}+n-\frac{3}{2},\dots,a_{j}+n-j-\frac{1}{2},a_{j}+n-j-\frac{1}{2},
 \dots,a_{n-1}+\frac{1}{2}\Bigr),\]
 and $\alpha_{j}=\epsilon_{j}-\epsilon_{j+1}$ is the only simple root orthogonal to $\gamma$.
Write $\Lambda_{j}$ for the set of such integral dominant weights.
\item If $a=0$, then
\[\gamma=\Bigl(a_{1}+n-\frac{3}{2},\dots,a_{n-1}+\frac{1}{2},0\Bigr).\]
$a_{j}$ ($1\leq j\leq n-1$) are half-integers, and $\alpha_{n}=\epsilon_{n}$
 is the only simple root orthogonal to $\gamma$.
Write $\Lambda_{n}$ for the set of such integral dominant weights.
\end{enumerate}


To describe irreducible representations with the infinitesimal character $\gamma$,
 we introduce several notation for every type of $\gamma$.

For a weight
\[\gamma=\Bigr(a_{1}+n-\frac{1}{2},\cdots,a_{n-1}+\frac{3}{2},a_{n}+\frac{1}{2}\Bigr)
 \in\Lambda_{0}\]
with $a_1\geq \cdots \geq a_{n-1}\geq a_n\geq 0$, let
\[\mu_{j}=(a_{1}+1,\cdots,a_{j}+1,a_{j+2},\cdots,a_{n})
 \textrm{ and }\nu_{j}=\Bigl(a_{j+1}+n-\frac{1}{2}-j\Bigr)\lambda_{0}\]
for $0\leq j\leq n-1$.
Put
\[I_{j}^{\pm}(\gamma)= I(\mu_j,\pm\nu_j)
 =\Ind_{MA\bar{N}}^{G}(V_{M,\mu_{j}}\otimes e^{\pm{\nu_j-\rho'}}
\otimes\mathbf{1}_{\bar{N}}).\]
For each $j$, there are nonzero intertwining operators
\[J_{j}^{+}(\gamma)\colon I_{j}^{+}(\gamma)
\rightarrow I_{j}^{-}(\gamma)
\text{ and }
J_{j}^{-}(\gamma)\colon I_{j}^{-}(\gamma)\rightarrow I_{j}^{+}(\gamma).\]
Write $\pi_{j}(\gamma)$ (resp.\ $\pi'_{j}(\gamma)$)
 for the image of $J_{j}^{-}(\gamma)$ (resp.\ $J_{j}^{+}(\gamma)$).
Put \[\lambda^+=
(a_{1}+1,\dots,a_{n-1}+1,a_{n}+1)\textrm{ and }
\lambda^-=(a_{1}+1,\dots,a_{n-1}+1,-(a_{n}+1)).\] Write $\pi^{+}(\gamma)$ for a
discrete series with the lowest $K$-type $V_{K,\lambda^+}$,
 and write $\pi^{-}(\gamma)$ for a discrete series with the lowest $K$-type $V_{K,\lambda^-}$.

Let $1\leq j\leq n-1$. For a weight
\[\gamma=\Bigl(a_{1}+n-\frac{3}{2},\dots,a_{j}+n-j-\frac{1}{2},
a_{j}+n-j-\frac{1}{2},\dots,a_{n-1}+\frac{1}{2}\Bigr)\in\Lambda_{j},\]
write \[\mu=(a_{1},\dots,a_{n-1})\textrm{ and }
 \nu=\Bigl( a_{j}+n-j-\frac{1}{2}\Bigr)\lambda_0.\]
Put
\[\pi(\gamma)=I(\mu,\nu).\]

For a weight
\[\gamma=\Bigl(a_{1}+n-\frac{3}{2},a_{2}+n-\frac{5}{2},\dots,a_{n-1}+\frac{1}{2},0\Bigr)
 \in\Lambda_{n},\]
write
\[\mu=(a_{1},\dots,a_{n-1}) \text{ and } I(\gamma)=I(\mu,0).\]
By Schmid's identity (\cite[Theorem 12.34]{Knapp})
 $I(\gamma)$ is the direct sum of two limits of discrete series
(\cite[Theorem 12.26]{Knapp}).
Write $\pi^{+}(\gamma)$ (resp.\ $\pi^{-}(\gamma)$)
 for a limits of discrete series with
 the lowest $K$-type $V_{K,\lambda^+}$ (resp.\ $V_{K,\lambda^-}$), where
\[\lambda^+=\Bigl(a_{1},a_{2},\dots,a_{n-1},\frac{1}{2}\Bigr)\textrm{ and }
\lambda^-=\Bigl(a_{1},a_{2},\dots,a_{n-1},-\frac{1}{2}\Bigr).\]

For a non-integral weight
\[\gamma=\Bigr(a_{1}+n-\frac{3}{2},\cdots,a_{n-1}+\frac{1}{2},a\Bigr),\]
write \[\mu=(a_{1},\dots,a_{n-1})
 \textrm{ and } \nu=a\lambda_0.\]
Put
\[\pi(\gamma)=I(\mu,\nu).\]
Note that $I(\mu,\nu)\cong I(\mu,-\nu)$.

By using the above notation,
 the Langlands classification of irreducible representations of $G$
 is given as follows.
In Fact~\ref{F:gamma-classification},
 an irreducible representation of $G$ means
 an irreducible, moderate growth, smooth Fr\'echet representation.

\begin{fact}\label{F:gamma-classification}
\begin{enumerate}
\item
For $\gamma\in\Lambda_{0}$, any irreducible representation of $G$
 with infinitesimal character $\gamma$ is
 equivalent to one of
\[\{\pi_{0}(\gamma),\dots,\pi_{n-1}(\gamma),\pi^{+}(\gamma),\pi^{-}(\gamma)\}.\]
When $0\leq j\leq n-2$,
$\pi_{j+1}(\gamma)\cong \pi'_{j}(\gamma)$; $\pi_{0}(\gamma)$ is a finite-dimensional module;
 and $\pi'_{n-1}(\gamma)\cong\pi^{+}(\gamma)\oplus \pi^{-}(\gamma)$.
\item
 For $\gamma\in\Lambda_{j}$ ($1\leq j\leq n-1$),
 any irreducible representation of $G$
 with infinitesimal character $\gamma$ is
 equivalent to $\pi(\gamma)$.
\item
For $\gamma\in\Lambda_{n}$, any irreducible representation of $G$
 with infinitesimal character $\gamma$ is
 equivalent to  $\pi^{+}(\gamma)$ or $\pi^{-}(\gamma)$.
\item
For a non-integral weight $\gamma$,
 any irreducible representation of $G$
 with infinitesimal character $\gamma$ is
 equivalent to $\pi(\gamma)$.
\end{enumerate}
\end{fact}

Among these representations, unitarizable ones are given as follows (\cite{Hirai}).

\begin{fact}\label{F:unitarizable}
\begin{enumerate}
\item
For $\gamma\in\Lambda_{0}$, $\pi^{+}(\gamma)$ and $\pi^{-}(\gamma)$ are unitarizable
 (discrete series).
 $\pi_j(\gamma)$ is unitarizable if and only if $a_i=0$ for any $j<i\leq n$.
\item
For $\gamma\in\Lambda_{j}$ ($1\leq j\leq n-1$),
 $\pi(\gamma)$ is unitarizable if and only if $a_i=0$ for any $j\leq i\leq n-1$.
\item
For $\gamma\in\Lambda_{n}$,
 $\pi^+(\gamma)$ and $\pi^-(\gamma)$ are unitarizable
 (limit of discrete series).
\end{enumerate}
\end{fact}

\begin{fact}
For a non-integral weight $\gamma$,
 $\pi(\gamma)$ is unitarizable if and only if
 at least one of the following two conditions holds.
\begin{enumerate}
\item
 $a\in \mathbf{i}\bbR$\ (unitary principal series).
\item
 $a\in \bbR$, $|a|< n-\frac{1}{2}$, $a_i\in \bbZ$ for $1\leq i\leq n-1$
 and $a_j=0$ for any $n-|a|-\frac{1}{2}<j\leq n-1$ (complementary series).
\end{enumerate}
\end{fact}

\begin{remark}\label{R:complAq}
The unitarizable representations $\pi(\gamma)$ for $\gamma\in \Lambda_{j}$
 in Fact~\ref{F:unitarizable} (2) can be regarded as a complementary series
 and also as $A_{\mathfrak{q}}(\lambda)$ as we see below.
\end{remark}

The unitarizable $(\mathfrak{g},K)$-modules with integral infinitesimal character
 are isomorphic to Vogan-Zuckerman's derived functor module $A_{\mathfrak{q}}(\lambda)$.
General references for $A_{\mathfrak{q}}(\lambda)$
 are e.g.\ \cite{Knapp-Vogan}, \cite{Vogan}.
For $0\leq j\leq n-1$
 let $\mathfrak{q}_j$ be a $\theta$-stable parabolic subalgebra of $\mathfrak{g}_{\bbC}$
 such that the real form of the Levi component of $\mathfrak{q}_j$
 is isomorphic to $\mathfrak{u}(1)^{j}+\mathfrak{so}(2(n-j),1)$.
For the normalization of parameters, we follow the book of Knapp-Vogan \cite{Knapp-Vogan}.
In particular, $A_{\mathfrak{q}}(\lambda)$ has infinitesimal character $\lambda+\rho$.

\begin{remark}\label{R:Aqsingular}
The parameter $\lambda=(a_1,\dots,a_{j},0,\dots,0)$ for $\mathfrak{q}_j$
 is in the good range if and only if
 $a_1\geq a_2\geq \cdots \geq a_{j} \geq 0$.
It is in the weakly fair range if and only if
 $a_i+1\geq a_{i+1}$ for $1\leq i\leq j-1$ and $a_j\geq -n+j$.
When $\lambda$ is in the weakly fair range, $A_{\mathfrak{q}}(\lambda)$ is nonzero
 if and only if $a_1\geq \cdots \geq a_{j}$ and $a_{j-1} \geq -1$.
\end{remark}

Let $\gamma\in \Lambda_0$ and $0\leq j\leq n-1$ such that $\pi_j(\gamma)$ is unitarizable.
Then
\begin{align*}
\pi_j(\gamma)_K \cong A_{\mathfrak{q}_{j}}(\lambda),
\end{align*}
where $\lambda=(a_1,\dots,a_{j},0,\dots,0)$.

Let $\gamma\in \Lambda_j\ (1\leq j\leq n-1)$ and $1\leq i\leq j$.
Assume that $a_{i}=\cdots=a_n=0$.
Then
\begin{align*}
\pi(\gamma)_K \cong A_{\mathfrak{q}_{i}}(\lambda),
\end{align*}
where $\lambda=(a_1-1,\dots,a_{i-1}-1,i-j-1,0,\dots,0)$.

\bigskip

Suppose next that $m$ is even and then $m=2n-2$ and $G=\Spin(2n-1,1)$.
This case is similar to and simpler than the previous case.
The infinitesimal character $\gamma$ of $\pi\in\mathcal{C}(G)$ is conjugate to
\[(\mu+\rho_{M},\nu)=(a_{1}+n-2,a_{2}+n-3,\cdots,a_{n-1},a),\]
 where $\mu=(a_{1},\dots,a_{n-1})$ and $\nu=a\lambda_{0}$.
We have $a_{1}\geq \cdots\geq a_{n-1}\geq 0$;
 and $a_{1},\dots,a_{n-1}$ are all integers or all half-integers.
The weight $\gamma$ is integral if and only if $a-a_{j}\in\mathbb{Z}$.
The singularity of integral $\gamma$ has the following possibilities:
\begin{enumerate}
\item If $a\neq a_{j}+n-j-1$ for $1\leq j\leq n-1$, then $\gamma$ is regular.
Write $\Lambda_{0}$ for the set of integral regular dominant weights.
\item If $a=a_{j}+n-j-1$ for some $1\leq j \leq n - 1$, then up to conjugation
\[\gamma=(a_{1}+n-2,\dots,a_{j}+n-j-1,a_{j}+n-j-1,\dots,a_{n-1}).\]
Write $\Lambda_{j}$ for the set of such integral dominant weights.
\end{enumerate}

We introduce several notation for every type of $\gamma$.

For a weight
\[\gamma=(a_{1}+n-1,a_{2}+n-2,\cdots,a_{n-1}+1,a_{n})
 \in\Lambda_{0}\]
with $a_1\geq \cdots \geq a_{n-1}\geq a_n\geq 0$, let
\[\mu_{j}=(a_{1}+1,\cdots,a_{j}+1,a_{j+2},\cdots,a_{n})
 \textrm{ and }\nu_{j}=(a_{j+1}+n-j-1)\lambda_{0}\]
for $0\leq j\leq n-1$.
Put
\[I_{j}^{\pm}(\gamma)=I(\mu_{j},\pm \nu_j).\]
For each $j$, there are nonzero intertwining operators
\[J_{j}^{+}(\gamma)\colon I_{j}^{+}(\gamma)
\rightarrow I_{j}^{-}(\gamma)
\text{ and }
J_{j}^{-}(\gamma)\colon I_{j}^{-}(\gamma)\rightarrow I_{j}^{+}(\gamma).\]
Write $\pi_{j}(\gamma)$ (resp.\ $\pi'_{j}(\gamma)$)
 for the image of $J_{j}^{-}(\gamma)$ (resp.\ $J_{j}^{+}(\gamma)$).

Let $1\leq j\leq n-1$. For a weight
\[\gamma=(a_{1}+n-2,\dots,a_{j}+n-j-1,
a_{j}+n-j-1,\dots,a_{n-1})\in\Lambda_{j},\]
write
\[\mu=(a_{1},\dots,a_{n-1})\textrm{ and }
 \nu=(a_{j}+n-j-1)\lambda_0.\]
Put
\[\pi(\gamma)=I(\mu,\nu).\]

For a non-integral weight
\[\gamma=(a_{1}+n-2,a_{2}+n-3,\cdots,a_{n-1},a),\]
write \[\mu=(a_{1},\dots,a_{n-1})  \textrm{ and } \nu=a\lambda_0.\]
Put
\[\pi(\gamma)=I(\mu,\nu).\]
Note that $I(\mu,\nu)  \cong I(\mu, -\nu)$.

Using these notation, the Langlands classification is given as follows.

\begin{fact}\label{F:gamma-classification2}
\begin{enumerate}
\item
For $\gamma\in\Lambda_{0}$,
 any irreducible representation of $G$
 with infinitesimal character $\gamma$ is
 equivalent to one of
\[\{\pi_{0}(\gamma),\dots,\pi_{n-1}(\gamma)\}.\]
When $0\leq j\leq n-2$,
$\pi_{j+1}(\gamma)\cong \pi'_{j}(\gamma)$;
 $\pi_{n-1}(\gamma)\cong \pi'_{n-1}(\gamma)$; and
 $\pi_{0}(\gamma)$ is a finite-dimensional module.
If $a_n=0$, then $\pi_{n-1}(\gamma)$ is tempered.
\item
 For $\gamma\in\Lambda_{j}$ ($1\leq j\leq n-1$),
 any irreducible representation of $G$
 with infinitesimal character $\gamma$ is
 equivalent to $\pi(\gamma)$.
\item
For a non-integral weight $\gamma$,
 any irreducible representation of $G$
 with infinitesimal character $\gamma$ is
 equivalent to $\pi(\gamma)$.
\end{enumerate}
\end{fact}

Among these representations, unitarizable ones are given as follows (\cite{Hirai}).

\begin{fact}\label{F:unitarizable2}
\begin{enumerate}
\item
For $\gamma\in\Lambda_{0}$,
 $\pi_j(\gamma)$ is unitarizable if and only if $a_i=0$ for any $j<i\leq n$.
\item
For $\gamma\in\Lambda_{j}$ ($1\leq j\leq n-1$),
 $\pi(\gamma)$ is unitarizable if and only if $a_i=0$ for any $j\leq i\leq n-1$.
\end{enumerate}
\end{fact}

\begin{fact}
For a non-integral weight $\gamma$,
 $\pi(\gamma)$ is unitarizable if and only if
 at least one of the following two conditions holds.
\begin{enumerate}
\item
 $a\in \mathbf{i}\bbR$\ (unitary principal series).
\item
 $a\in \bbR$, $|a|<n-1$, $a_i\in \bbZ$ for $1\leq i\leq n-1$
 and $a_j=0$ for any $n-|a|-1<j\leq n-1$ (complementary series).
\end{enumerate}
\end{fact}

For $0\leq j\leq n-1$
 let $\mathfrak{q}_j$ be a $\theta$-stable parabolic subalgebra of $\mathfrak{g}_{\bbC}$
 such that the real form of the Levi component of $\mathfrak{q}_j$
 is isomorphic to $\mathfrak{u}(1)^{j}+\mathfrak{so}(2(n-j)-1,1)$.

Remarks~\ref{R:complAq} and \ref{R:Aqsingular} are valid without change of words.

Let $\gamma\in \Lambda_0$ and $0\leq j\leq n-1$ such that $\pi_j(\gamma)$ is unitarizable.
Then
\begin{align*}
\pi_j(\gamma)_K \cong A_{\mathfrak{q}_{j}}(\lambda),
\end{align*}
where $\lambda=(a_1,\dots,a_{j},0,\dots,0)$.

Let $\gamma\in \Lambda_j\ (1\leq j\leq n-1)$ and $1\leq i\leq j$.
Assume that $a_{i}=\cdots=a_n=0$.
Then
\begin{align*}
\pi(\gamma)_K \cong A_{\mathfrak{q}_{i}}(\lambda),
\end{align*}
where $\lambda=(a_1-1,\dots,a_{i-1}-1,i-j-1,0,\dots,0)$.

\subsection{Branching laws for the restriction from $\Spin(2n,1)$ to $P$}\label{SS:branchinglaw}

We deduce branching laws for the restriction to $P$
 of all irreducible unitary representations of $G$.

In this subsection suppose $G=\Spin(2n,1)$.
A similar result for the group $\Spin(2n-1,1)$ will be given in the next subsection.

By Fact~\ref{F:unitarizable},
 many of irreducible unitary representations of $G$
 are the completion of principal series representations.
This is the case if the infinitesimal character $\gamma$ lies in
 $\Lambda_j\ (1\leq j\leq n-1)$ or $\gamma$ is not integral.

\begin{theorem}\label{T:branching-ps}
Suppose that an irreducible unitary representation $\pi$ of $Spin(2n,1)$ is
 isomorphic to the completion of a principal series representation
 $I(\mu,\nu)$, where $\mu=(a_1,\dots,a_{n-1})$ and $a_1\geq a_2\geq \dots\geq a_{n-1}\geq 0$.
Then
\[ \pi|_{P}\cong \bigoplus_{\tau}
 \Ind_{M'N}^P(V_{M',\tau}\otimes e^{\mathbf{i}\xi_0}), \]
where $\tau=(b_1,\dots,b_{n-1})$ runs over tuples such that
 \[a_{1}\geq b_1 \geq a_{2}\geq b_2 \geq
  \cdots \geq a_{n-1} \geq |b_{n-1}|\] and $b_{i}-a_{1}\in\mathbb{Z}$.
\end{theorem}

\begin{proof}
By Lemma~\ref{L:unitary-J} and Proposition~\ref{P:res-induced},
 the theorem follows from the well-known branching law from $M=\Spin(2n-1)$ to $M'=\Spin(2n-2)$
 (see e.g.\ \cite[Theorem 8.1.3]{Goodman-Wallach}).
\end{proof}

Next, let $\gamma\in \Lambda_0$, namely, $\gamma$ is a regular integral weight.
Recall that in \S\ref{SS:irreducible}
 we defined $\pi_j(\gamma)$ to be the image of the intertwining operator
 $J_j^{-}(\gamma)\colon I_j^-(\gamma) \to I_j^+(\gamma)$.
We give branching laws for $\pi_j(\gamma)$ for $1\leq j\leq n-1$
 when it is unitarizable.

\begin{theorem}\label{T:branching-regular}
Let $1\leq j\leq n-1$ and let
\[\gamma=\Bigl(a_1+n-\frac{1}{2},\dots,a_{j}+n-j+\frac{1}{2},
 n-j-\frac{1}{2},\dots,\frac{1}{2} \Bigr),\]
where $a_1\geq \cdots \geq a_{j}\geq 0$ are integers.
Then
\[ \bar{\pi}_j(\gamma)|_{P}\cong
\bigoplus_{\tau}
 \Ind_{M'N}^P(V_{M',\tau}\otimes e^{\mathbf{i}\xi_0}), \]
where $\tau=(b_1,\dots,b_{j-1},0,\dots,0)$ runs over tuples of integers such that
 \[a_{1}+1 \geq b_1 \geq a_{2}+1 \geq b_2 \geq
 \cdots \geq a_{j-1}+1 \geq b_{j-1}\geq a_{j}+1.\]
\end{theorem}

\begin{proof}
Let $0\leq i < j$.
It is known that $[\pi_i(\gamma)]+[\pi_{i+1}(\gamma)]=[I_i^+(\gamma)]$
 in the Grothendieck group $K(G)$.
Hence by Proposition~\ref{P:res-induced}
\[
\Psi([\pi_i(\gamma)])+\Psi([\pi_{i+1}(\gamma)]) = [V_{M,\mu_i}|_{M'}],
\]
where $\mu_i=(a_1+1,\dots,a_{i}+1,a_{i+2},\dots,a_{j},0,\dots,0)$.
Since $\pi_0(\gamma)$ is finite-dimensional, $\Psi([\pi_0(\gamma)])=0$.
Then by induction on $i$, we have
\[\Psi([\pi_i(\gamma)])= \bigoplus_{\tau} [V_{M',\tau}],
\]
where $\tau=(b_1,\dots,b_{j-1},0,\dots,0)$ runs over tuples of integers such that
\begin{align*}
&a_{1}+1 \geq b_1 \geq a_{2}+1 \geq \cdots \geq   b_{i-1} \geq a_{i}+1, \text{ and } \\
&a_{i+1}\geq b_{i} \geq a_{i+2} \geq b_{i+1} \geq \dots\geq a_{j}\geq  |b_{j-1}|.
\end{align*}
Hence the theorem follows from Lemma~\ref{L:unitary-J}.
\end{proof}

We have the following formula for $A_{\mathfrak{q}}(\lambda)$
 by  Theorems~\ref{T:branching-ps} and \ref{T:branching-regular}.
For $0\leq j\leq n-1$
 let $\mathfrak{q}_j$ be a $\theta$-stable parabolic subalgebra of $\mathfrak{g}_{\bbC}$
 such that the real form of the Levi component of $\mathfrak{q}_j$
 is isomorphic to $\mathfrak{u}(1)^{j}+\mathfrak{so}(2(n-j),1)$.
For a weakly fair parameter $\lambda=(a_1,\dots,a_j,0,\dots,0)$, we have
\[ \overline{A_{\mathfrak{q}_{j}}(\lambda)}|_{P}\cong
\bigoplus_{\tau}
 \Ind_{M'N}^P(V_{M',\tau}\otimes e^{\mathbf{i}\xi_0}), \]
where $\tau=(b_1,\dots,b_{j-1},0,\dots,0)$ runs over tuples of integers such that
 \[a_{1}+1 \geq b_1 \geq a_{2}+1 \geq b_2 \geq
 \cdots \geq a_{j-1}+1 \geq b_{j-1}\geq \max\{a_{j}+1,0\}.\]

The remaining representations are (limit of) discrete series representations.
The following formula is proved by the translation principle
 and the case where $\gamma=\rho$.
The proof for the case $\gamma=\rho$ involves an explicit calculation
 for the lowest $K$-type and will be later proved in Proposition~\ref{P:P-restriction2}.

\begin{theorem}\label{T:branching-ds}
Let
\[\gamma=\Bigl(a_1+n-\frac{1}{2}, a_2+n-\frac{3}{2},
 \dots, a_{n}+\frac{1}{2} \Bigr),\]
where $a_1\geq \cdots \geq a_{n}\geq -\frac{1}{2}$ are all integers
 or all half-integers.
Then
\[ \bar{\pi}^{\pm}(\gamma)|_{P}\cong
\bigoplus_{\tau}
 \Ind_{M'N}^P(V_{M',\tau}\otimes e^{\mathbf{i}\xi_0}), \]
where $\tau=(b_1,\dots,b_{n-1})$ runs over tuples such that
\[a_{1}+1 \geq b_1 \geq a_{2}+1 \geq \cdots \geq b_{n-2}
 \geq a_{n-1}+1 \geq {\mp} b_{n-1}\geq a_{n}+1\]
 and $b_{i}-a_{1}\in\mathbb{Z}$.
\end{theorem}

\begin{proof}
By Lemma \ref{L:unitary-J},
 it suffices to calculate $\Psi([\pi^{\pm}(\gamma)])$ for
 $\gamma\in\Lambda_{0}\sqcup \Lambda_{n}$.

The same argument as in the proof of Theorem~\ref{T:branching-regular} yields
\begin{align}\label{Eq:ds-pair}
\Psi([\pi^+(\gamma)])+\Psi([\pi^-(\gamma)])
=\Psi([\pi'_{n-1}(\gamma)])
=\bigoplus_{\tau} [V_{M',\tau}],
\end{align}
where $\tau=(b_1,\dots,b_{n-1})$ runs over tuples of integers such that
\begin{align*}
&a_{1}+1 \geq b_1 \geq a_{2}+1 \geq \cdots \geq  b_{n-2} \geq a_{n-2}+1\geq |b_{n-1}|\geq a_{n-1}+1.
\end{align*}
Therefore, it suffices to show that the two modules $\Psi([\pi^+(\gamma)])$ and $\Psi([\pi^-(\gamma)])$
are separated by the sign of $b_{n-1}$.

First, prove the statement for $\gamma\in\Lambda_0$
 by induction on $|\gamma|$.
When $\gamma=\rho$, the conclusion follows from Proposition~\ref{P:P-restriction2}.
Let $\gamma\not\in\Lambda_{0}-\{\rho\}$
 and assume that the conclusion holds for weights in $\Lambda_{0}$
 having norm strictly smaller than $|\gamma|$.
Write $\omega_{k}$ for the $k$-th fundamental weight, namely,
\[
\omega_{k}=(\underbrace{1,\dots,1}_{k},\underbrace{0,\dots,0}_{n-k})
\text{ for $1\leq k\leq n-1$ and }
\omega_{n}=
\Bigl(\underbrace{\frac{1}{2},\dots,\frac{1}{2}}_{n}\Bigr).
\]
Then, one finds $\gamma'\in\Lambda_{0}$ and a fundamental weight $\omega_{k}$
 such that $\gamma=\gamma'+\omega_{k}$.
By the Zuckerman translation principle (\cite{Vogan}, \cite{Zuckerman}),
 $\pi^{\pm}(\gamma)$ occurs as a composition factor of
 $\pi^{\pm}(\gamma')\otimes F_{\omega_{k}}$.
Hence by Lemma~\ref{L:tensor},
 if an irreducible $M'$-representation $V_{M',\mu}$
 with $\mu=(b_{1},\dots,b_{n-1})$ occurs in $\Psi([\pi^{+}(\gamma)])$,
 then it also occurs in $\Psi([\pi^{+}(\gamma')])\otimes [F_{\omega_{k}}|_{M'}]$.
For any irreducible $V_{M',\mu'}$ in $\Psi([\pi^{+}(\gamma')])$
 with $\mu'=(b'_{1},\dots,b'_{n-1})$, one has $b'_{n-1}\leq -1$
 by induction hypothesis and
 for any weight $\mu''$ appearing in $F_{\omega_{k}}|_{M'}$
 with $\mu''=(b''_{1},\dots,b''_{n-1})$,
 we have $b''_{n-1} \in\{1,-1,\frac{1}{2},-\frac{1}{2}\}$.
Hence $b_{n-1}\leq 0$.
Therefore, we get $b_{n-1}\leq -1$ from \eqref{Eq:ds-pair}.
The statement for $\pi^{-}(\gamma)$ is similarly proved.

Next, suppose that $\gamma\in\Lambda_{n}$.
Let $\gamma'=\gamma+\omega_{n}\in \Lambda_{0}$.
Then again by the translation principle,
 $\pi^{\pm}(\gamma)$ occurs as a composition factor of
 $\pi^{\pm}(\gamma')\otimes F_{\omega_{n}}$.
Then by using the result for $\Psi([\pi^{\pm}(\gamma')])$ proved above,
 the statement for $\Psi([\pi^{\pm}(\gamma)])$ is similarly obtained.
\end{proof}

\subsection{Branching laws for the restriction from $\Spin(2n-1,1)$ to $P$}\label{SS:branchinglaw2}

Let $G=\Spin(2n-1,1)$.
Branching laws for the restriction to $P$ are similar to the previous case where $G=\Spin(2n,1)$.

\begin{theorem}\label{T:branching-ps2}
Suppose that an irreducible unitary representation $\pi$ of $Spin(2n-1,1)$ is
 isomorphic to the completion of a principal series representation
 $I(\mu,\nu)$, where $\mu=(a_1,\dots,a_{n-1})$ and $a_1\geq \dots\geq a_{n-2}\geq |a_{n-1}|$.
Then
\[ \pi|_{P}\cong \bigoplus_{\tau}
 \Ind_{M'N}^P(V_{M',\tau}\otimes e^{\mathbf{i}\xi_0}), \]
where $\tau=(b_1,\dots,b_{n-2})$ runs over tuples such that
 \[a_{1}\geq b_1 \geq a_{2}\geq b_2 \geq
  \cdots \geq a_{n-1} \geq b_{n-2}\geq |a_{n-1}|\] and $b_{i}-a_{1}\in\mathbb{Z}$.
\end{theorem}

\begin{theorem}\label{T:branching-regular2}
Let
\[\gamma=(a_1+n-1,\dots,a_{j}+n-j, n-j-1,\dots, 0),\]
where $a_1\geq \cdots \geq a_{j}\geq 0$ are integers
 and let $1\leq j\leq n-1$.
Then
\[ \bar{\pi}_j(\gamma)|_{P}\cong
\bigoplus_{\tau}
 \Ind_{M'N}^P(V_{M',\tau}\otimes e^{\mathbf{i}\xi_0}), \]
where $\tau=(b_1,\dots,b_{j-1},0,\dots,0)$ runs over tuples of integers such that
 \[a_{1}+1 \geq b_1 \geq a_{2}+1 \geq b_2 \geq
 \cdots \geq a_{j-1}+1 \geq b_{j-1}\geq a_{j}+1.\]
\end{theorem}

For $0\leq j\leq n-1$
 let $\mathfrak{q}_j$ be a $\theta$-stable parabolic subalgebra of $\mathfrak{g}_{\bbC}$
 such that the real form of the Levi component of $\mathfrak{q}_j$
 is isomorphic to $\mathfrak{u}(1)^{j}+\mathfrak{so}(2(n-j)-1,1)$.
For a weakly fair parameter $\lambda=(a_1,\dots,a_j,0,\dots,0)$, we have
\[ \overline{A_{\mathfrak{q}_{j}}(\lambda)}|_{P}\cong
\bigoplus_{\tau}
 \Ind_{M'N}^P(V_{M',\tau}\otimes e^{\mathbf{i}\xi_0}), \]
where $\tau=(b_1,\dots,b_{j-1},0,\dots,0)$ runs over tuples of integers such that
 \[a_{1}+1 \geq b_1 \geq a_{2}+1 \geq b_2 \geq
 \cdots \geq a_{j-1}+1 \geq b_{j-1}\geq \max\{a_{j}+1,0\}.\]

\section{Moment map for elliptic coadjoint orbits}\label{S:elliptic}

In this section and the next section we calculate the projection of semisimple coadjoint orbits
 for $G$ by the natural map $\mathfrak{g}^*\to \mathfrak{p}^*$.
We treat elliptic orbits for $G=\Spin(2n,1)$ in this section.
Non-elliptic orbits and the case $G=\Spin(2n-1,1)$ will be treated in the next section.

Throughout this section we assume $m$ is odd and then $G=\Spin(2n,1)$.
In Definition~\ref{D:depth},
 we divided the coadjoint orbits $\mathcal{O}$ for $P$ into two types:
 depth zero and depth one according to
 $\mathcal{O}\subset \mathfrak{l}^*$ or not.
Then we saw at the end of \S\ref{SS:P-orbit2} that
 the depth one coadjoint orbits for $P$ are parametrized by
 singular values of the matrix $Y$ there
 and the sign of the Pfaffian of $Z_{Y,\beta}$.

\subsection{$P$-orbits in $\mathcal{O}_{f}$}\label{SS:doubleCoset}

For $a_1\geq a_2\geq\cdots\geq a_{n-1}\geq|a_{n}|\geq 0$, write
$\vec{a}=(a_{1},a_{2},\dots,a_{n}).$
As in \eqref{Eq:ta}, put
\begin{equation*}
t_{\vec{a}}=
\begin{pmatrix}
 0&  a_1&&&&\\
  -a_{1}&0&&&&\\
 &&\ddots&&& \\
 &&&0& a_{n}&\\
 &&& -a_{n}&0&\\
 &&&&&0\\
\end{pmatrix}.
\end{equation*}
By the isomorphism
 $\iota\colon \mathfrak{g}\xrightarrow{\sim} \mathfrak{g}^*$ in
 \eqref{Eq:identification1},
 we put \[f=f_{\vec{a}}=\iota(t_{\vec{a}}),\] which is an elliptic
 element in $\mathfrak{g}^{\ast}$.
Moreover, each elliptic coadjoint orbit in $\mathfrak{g}^{\ast}$ contains
 $f_{\vec{a}}$ for a unique vector $\vec{a}$.
We first consider regular orbits, i.e., assume that
 \[a_1>a_2>\cdots >a_{n-1}>|a_{n}|>0.\]
Write $G^{f}$ for the stabilizer of $G$ at $f$.
Then, $G^{f} = T$, where $T$ is the pre-image in $G$
 of the maximal torus
\begin{equation*}
T_1=\Biggl\{
\begin{pmatrix}
 y_{1}&z_1&&&&\\
 -z_{1}&y_{1}&&&&\\
 &&\ddots&&&\\
 &&&y_{n}&z_{n}&\\
 &&&-z_{n}&y_{n}&\\
 &&&&&1\\
\end{pmatrix}:
 y_{1}^{2}+z_{1}^{2}=\cdots=y_{n}^{2}+z_{n}^{2}=1
\Biggr\}
\end{equation*}
 of $G_{1}$.
Put $\mathcal{O}_{f}=G\cdot f$ and then $\mathcal{O}_{f}\cong G/G^{f}$.
To parametrize $P$-orbits in $\mathcal{O}_{f}$ is equivalent
 to parametrize double cosets in $P\backslash G/G^{f}$.
Since the map
\[ P\backslash G/G^{f}\ni P g G^{f}\mapsto G^{f} g^{-1} P\in G^{f}\backslash G/P\]
 is an isomorphism,
 it is also equivalent to
 parametrize $G^{f}$-orbits in $G/P$.
Write
\begin{equation*}
X_{n}=\{\vec{x}=(x_1,\dots,x_{2n},x_{0}):
 x_{0}^{2}=\sum_{i=1}^{2n} x_{i}^{2} \text{ and } x_{0}>0\}/\sim.
\end{equation*}
Here, for two vectors $\vec{x}$ and $\vec{x'}$, we defined
\[\vec{x}\sim\vec{x'}\Leftrightarrow\exists s>0
 \textrm{ such that }\vec{x'}=s\vec{x}.\]
As a manifold, $X_{n}\cong S^{2n-1}$.
The group $G$ acts on $X_{n}$ transitively as
\[g\cdot[\vec{x}]=[\vec{x}g_{1}^{t}],\quad
 G\ni g\mapsto g_{1} \in G_1.\]
Put
\begin{equation*}
v_{0}=[(0,\dots,0,1,1)].
\end{equation*}
Then $\Stab_{G}(v_{0})=P$ and hence $X_{n}\cong G/P.$
Therefore, to parametrize $G^{f}$-orbits in
 $G/P$ is equivalent to parametrize $T$-orbits in $X_{n}$.

Let
\begin{equation*}
B=\Bigl\{\vec{b}=(b_1,\dots,b_{n}):b_1,\dots,b_n\geq 0,\
 \sum_{i=1}^{n-1} b_{i}^{2}=1-2b_{n}\Bigr\}.
\end{equation*}
Then, $0\leq b_{n}\leq\frac{1}{2}$ for any $\vec{b}\in B$.
Write
\begin{equation*}
\alpha=\alpha_{\vec{b}}=(0,b_1,0,b_{2},\dots,0,b_{n-1},0)
\text{ and }
\bar{X}_{\vec{b}}=
\begin{pmatrix}
 0_{2n-1}&\alpha^{t}&\alpha^{t}\\
 -\alpha&0&0\\
 \alpha&0&0\\
\end{pmatrix}.
\end{equation*}
Put
\begin{equation*}
\bar{n}_{\vec{b}}=
\exp(\bar{X}_{\vec{b}}) =
\begin{pmatrix}
 I_{2n-1}&\alpha^{t}&\alpha^{t}\\
 -\alpha&1-\frac{1}{2}|\alpha|^2&-\frac{1}{2}|\alpha|^2\\
 \alpha&\frac{1}{2}|\alpha|^2&1+\frac{1}{2}|\alpha|^2\\
\end{pmatrix}\in\bar{N}.
\end{equation*}
Then,
\begin{equation*}
\bar{n}_{\vec{b}}^{-1}\cdot v_{0}=[(0,-b_1,0,-b_2,\dots,0,-b_{n-1},0,b_{n},1-b_{n})].
\end{equation*}


\begin{lemma}\label{L:T-orbits}
The map $B\rightarrow X_{n}/T$ defined by
\begin{equation*}
\vec{b}\mapsto(\bar{n}_{\vec{b}})^{-1} \cdot v_{0}
\end{equation*} is a bijection.
\end{lemma}

\begin{proof}
Identify the image of $T$ in $G_{1}$ with $\U(1)^{n}$.
Then, $T$ acts on $X_{n}$ by
\begin{align*}&
(y_{1}+z_{1}\mathbf{i},\dots,y_{n}+z_{n}\mathbf{i})\cdot[(x_1,\dots,x_{2n},x_{0})]&\\&=[(y_{1}x_{1}+z_{1}x_{2},
-z_{1}x_{1}+y_{1}x_{2},\dots,y_{n}x_{2n-1}+z_{n}x_{2n},-z_{n}x_{2n-1}+y_{n}x_{2n},
x_{0})].&\end{align*}
Then each $T$-orbit in $X_{n}$ has a unique representative of the form
\[ [(0,-b_1,\dots,0, -b_{n-1},0,b_{n},1-b_{n})],\]
 where $b_{i}\geq 0$ ($1\leq i\leq n$).
Moreover, the equation \[\sum_{i=1}^{2n} x_{i}^{2}=x_{0}^{2}\] leads to the equation
\[\sum_{i=1}^{n-1} b_{i}^{2}=1-2b_{n}.\]
By this, the map $\vec{b}\mapsto(\bar{n}_{\vec{b}})^{-1}\cdot v_{0}$ is a bijection.
\end{proof}

By Lemma~\ref{L:T-orbits}, we proved

\begin{lemma}\label{L:P-orbits}
Each $P$-orbit in $\mathcal{O}_{f}=G\cdot f$ contains some $\bar{n}_{\vec{b}}\cdot f$
 for a unique tuple $\vec{b}\in B$.
\end{lemma}

\subsection{The moment map $\mathcal{O}_{f}\rightarrow\mathfrak{p}^{\ast}$}\label{SS:moment}

In \S\ref{SS:P-orbit2}, we showed an explicit parametrization of coadjoint $P$-orbits.

In this subsection, we use the results in \S\ref{SS:P-orbit2} to calculate
 the image of the moment map $q\colon \mathcal{O}_f \rightarrow\mathfrak{p}^{\ast}$.
Here, the moment map $q$ is defined by the composition
 of the inclusion $\mathcal{O}_f\hookrightarrow \mathfrak{g}^*$
 and the dual map $\mathfrak{g}^*\to \mathfrak{p}^*$
 of the inclusion $\mathfrak{p}\hookrightarrow \mathfrak{g}$.

Recall the map $\pr$ defined in \eqref{Eq:identificaiton2}.
Put
\begin{equation*}
H'=
\begin{pmatrix}
0&1\\
-1&0\\
\end{pmatrix}.
\end{equation*}

\begin{lemma}\label{L:gf}
We have $q(\bar{n}_{\vec{b}}\cdot f)=\pr(X_{Y,\beta,0})$, where
\begin{align*}
&\beta=(-a_{1}b_{1},0,\dots,-a_{n-1}b_{n-1},0,a_{n}b_{n})\neq 0,\\
&Y=\diag\{a_{1}H',\dots,a_{n-1}H',0\} + (\beta')^t \alpha - \alpha^t \beta',
 \text{ and }
\beta'=(\underbrace{0,\dots,0}_{2n-2},a_{n}).
\end{align*}
\end{lemma}

\begin{proof}
Write $Y'=\diag\{a_{1}H',\dots,a_{n-1}H',0\}$.
Then following notation in \S\ref{SS:notation} we have
$t_{\vec{a}}=
 \diag\{Y',0_{2\times 2}\} - X_{\frac{\beta'}{2}}+\bar{X}_{\frac{\beta'}{2}}$.
Using \eqref{Eq:brackets} we calculate
\begin{align*}
&\Ad(\bar{n}_{\vec{b}}) (t_{\vec{a}})
= t_{\vec{a}} + [\bar{X}_{\alpha}, t_{\vec{a}}]
  + \frac{1}{2} [[\bar{X}_{\alpha}, [\bar{X}_{\alpha}, t_{\vec{a}}]] \\
&= \bigl( \diag\{Y',0_{2\times 2}\} - X_{\frac{\beta'}{2}}
 + \bar{X}_{\frac{\beta'}{2}}\bigr) \\
&\quad - \bigl(\bar{X}_{\alpha(Y')^t}
  - \diag\{(\beta')^t \alpha - \alpha^t \beta', 0_{2\times 2}\}\bigr)
  - \frac{1}{2}
 \bar{X}_{\alpha \alpha^t \beta' - \alpha(\beta')^t \alpha}.
\end{align*}
Hence the lemma follows from
\[
Y'+(\beta')^t \alpha - \alpha^t \beta' = Y \text{ and }
\frac{\beta'}{2}-\alpha(Y')^t
 - \frac{1}{2}(\alpha \alpha^t \beta' - \alpha(\beta')^t \alpha)
=\beta.
\qedhere
\]
\end{proof}



Put
\begin{align*}
&Y_{\vec{b}}=Y-\frac{1}{|\beta|^2}(Y\beta^{t}\beta-
(Y\beta^{t}\beta)^{t}), \quad
Z_{\vec{b}}=
\begin{pmatrix}
 Y_{\vec{b}}&\frac{\beta^{t}}{|\beta|}\\
 -\frac{\beta}{|\beta|}&0\\
\end{pmatrix}.
\end{align*}
By Lemmas \ref{L:p-standard4}, \ref{p:standard5} and \ref{L:class-matrix},
 the $P$-conjugacy class of $q(\bar{n}_{\vec{b}}\cdot f)$ is determined by
 the sign of the Pfaffian of $Z_{\vec{b}}$
 and singular values of $Y_{\vec{b}}$.
Put
\begin{align*}
&\gamma_{1}=(a_{1}b_{1},\dots,a_{n-1}b_{n-1},-a_{n}b_{n}),\\
&\gamma_{2}=((a_{1}^{2}-a_{n}^{2}b_{n})b_{1}, \dots,(a_{n-1}^{2}-a_{n}^{2}b_{n})b_{n-1},0).
\end{align*}
For a permutation $\sigma$ on $\{1,2,\dots,2n\}$, let
$Q_{\sigma}=(x_{ij})_{1\leq i,j\leq 2n}$ be the permutation matrix corresponding to $\sigma$,
 that is, $x_{i,j}=1$ if $j=\sigma(i)$; and $x_{i,j}=0$ if $j\neq\sigma(i)$.

\begin{lemma}\label{L:Zb}
Let $\sigma$ be the permutation
\[\sigma(i) =\begin{cases} 2i-1 &(1\leq i\leq n) \\
 2(i-n) & (n+1\leq i\leq 2n) \end{cases}.\]
Then
\[Q_{\sigma}Z_{\vec{b}}Q_{\sigma}^{-1}
 =\begin{pmatrix}
 0_{n}&Z\\
 -Z^{t}&0_{n}\\
\end{pmatrix},\]
 where
\begin{equation}\label{Eq:M2}
Z= \begin{pmatrix} a_{1}&\ldots&0&\frac{-a_{1}b_{1}}{|\beta|}\\
 \vdots&\ddots&\vdots& \vdots\\
 0&\ldots&a_{n-1}&\frac{-a_{n-1}b_{n-1}}{|\beta|}\\
 a_{n}b_{1}&\ldots&a_{n}b_{n-1} &\frac{a_{n}b_{n}}{|\beta|}\\
\end{pmatrix}-\frac{\gamma_{1}^{t}\gamma_{2}}{|\beta|^{2}}.
\end{equation}
\end{lemma}

\begin{proof}
By calculation
\[ Y\beta^{t}=(0,(a_{1}^{2}-a_{n}^{2}b_{n})b_{1},\dots,0,(a_{n-1}^{2}-a_{n}^{2}b_{n})b_{n-1},
0)^{t}.\]
By inputting the forms of $Y$, $\beta$, $Y\beta^{t}$
 in
\[Z_{\vec{b}}
 = \begin{pmatrix}
 Y-\frac{1}{|\beta|^{2}}(Y\beta^{t})\beta
  + \frac{1}{|\beta|^2}\beta^{t} (Y\beta^{t})^{t}&\frac{\beta^{t}}{|\beta|}\\
 -\frac{\beta}{|\beta|}&0\\
\end{pmatrix}.\]
 we get the form of $Z_{\vec{b}}$.
It is easy to see that
 $Q_{\sigma}Z_{\vec{b}}Q_{\sigma}^{-1}$ is of the block diagonal form as in the lemma.
\end{proof}

\begin{lemma}\label{L:Zb2}
The Pfaffian of $Z_{\vec{b}}$ is equal to
\[\frac{1-b_{n}}{|\beta|}\prod_{1\leq i\leq n}a_{i}.\]
\end{lemma}

\begin{proof}
By Lemma \ref{L:Zb}, the Pfaffian of $Z_{\vec{b}}$ is equal to $\det Z$.
Since $\gamma_{1}^{t}$ is proportional
 to the right most column of the first matrix in the right hand side of \eqref{Eq:M2}
 and the last entry of $\gamma_2$ is zero,
 the term $\frac{1}{|\beta|^2}
\gamma_{1}^{t}\gamma_{2}$ makes no contribution to $\det Z$.
Therefore,
\begin{align*}
\det Z&=
 \det \begin{pmatrix} a_{1}&\ldots&0&\frac{-a_{1}b_{1}}{|\beta|}\\
 \vdots&\ddots&\vdots&\vdots\\0&\ldots&  a_{n-1}&\frac{-a_{n-1}b_{n-1}}{|\beta|}\\
 a_{n}b_{1}&\ldots&a_{n}b_{n-1}&\frac{a_{n}b_{n}}{|\beta|}\\
 \end{pmatrix}\\
&=\frac{1-b_{n}}{|\beta|}\prod_{1\leq i\leq n}a_{i}. \qedhere
\end{align*}
\end{proof}

Write $Z'$ for the $n\times (n-1)$ matrix obtained from $Z$ by removing the last column.
Put \begin{equation*}
h_{\vec{b}}(x)=\det(xI_{n-1}-(Z')^{t}Z').
\end{equation*}
Then we claim that the singular values of $Y_{\vec{b}}$
 are square roots of zeros of $h_{\vec{b}}(x)$.
Indeed, let $W\in \SO(n)$ be a matrix such that
 the right most column of $WZ$ is $(0,0,\dots,0,1)^t$.
Then $WZ=\diag\{Z'_0,1\}$ for some $(n-1)\times (n-1)$ matrix $Z'_0$.
Hence the eigenvalues of $Z^t Z$ are
 the eigenvalues of $(Z')^{t}Z'=(Z'_0)^tZ'_0$ plus $1$.
Since the eigenvalues of $Z^t Z$ are the same as
 those of $(Z_{\vec{b}})^t Z_{\vec{b}}$
 and they are the singular values of $Y_{\vec{b}}$ plus $1$
 (see the last part of \S\ref{SS:P-orbit2}),
 the claim follows.

\begin{proposition}\label{P:hb3}
We have
\begin{equation}\label{Eq:hb2}
h_{\vec{b}}(x)
 = \sum_{1\leq i\leq n}\frac{a_{i}^{2}b_{i}^{2}}{|\beta|^{2}}
 \prod_{1\leq j\leq n,\, j\neq i}(x-a_{j}^{2}).
\end{equation}
For any $1\leq i\leq n$, $a_{i}^{2}$ is a root of $h_{\vec{b}}(x)$ if and only if $b_{i}=0$.
\end{proposition}

\begin{proof}
Put
\begin{align*}
&\gamma_{3}=(a_{n}b_{1},\dots,a_{n}b_{n-1}), \\
&\gamma_{4}=\frac{1}{|\beta|}
\bigl(
(a_{1}^{2}-a_{n}^{2}b_{n})b_{1},\dots,
 (a_{n-1}^{2}-a_{n}^{2}b_{n})b_{n-1}\bigr).
\end{align*}
By a direct calculation we see that
\[(Z')^{t}Z'=\diag\{a_{1}^{2},\dots,a_{n-1}^{2}\}
 +\gamma_{3}^{t}\gamma_{3}- \gamma_{4}^{t}\gamma_{4}.\]
From this we calculate that
\[h_{\vec{b}}(a_{i}^{2})=\frac{a_{i}^{2}b_{i}^{2}}{|\beta|^{2}}
\prod_{1\leq j\leq n,\, j\neq i}(a_{i}^{2}-a_{j}^{2})\] for $1\leq i\leq n-1$.
Since $h_{\vec{b}}(x)$ is a monic polynomial of degree $n-1$,
 we get
\[h_{\vec{b}}(x)=
 \sum_{1\leq i\leq n}\frac{a_{i}^{2}b_{i}^{2}}{|\beta|^{2}}
 \prod_{1\leq j\leq n,\, j\neq i}(x-a_{j}^{2}). \qedhere\]
\end{proof}


\begin{corollary}\label{C:hb4}
The polynomial $h_{\vec{b}}(x)$
 has $n-1$ positive roots, which lie in the intervals
 \[[a_{n}^{2},a_{n-1}^{2}],\dots,[a_{2}^{2},a_{1}^{2}],\]
 respectively.
\end{corollary}

\begin{proof}
First assume that none of $b_{i}$ is zero.
Then by Proposition~\ref{P:hb3}, $h_{\vec{b}}(a_{i}^{2})$ and $h_{\vec{b}}
(a_{i+1}^{2})$ have different signs.
Thus, $h_{\vec{b}}(x)$ has a zero in $(a_{i+1}^{2},a_{i}^{2})$
 for each $1\leq i\leq n-1$.
Hence, the $n-1$ zeros of $h_{\vec{b}}(x)$ lie in the intervals
\[(a_{n}^{2},a_{n-1}^{2}),\dots,(a_{2}^{2},a_{1}^{2}),\]
 respectively. Therefore, $h_{\vec{b}}(x)$ has no double zeros.

In general, among $\{b_1,\dots,b_{n}\}$
 let $b_{i_{1}},\dots,b_{i_{l}}$ with $1\leq i_{1}<\cdots<i_{l}\leq n$
 be all nonzero members.
Write $I=\{i_{1},\dots,i_{l}\}$ and $J=\{1,\dots,n\}-\{i_{1},\dots,i_{l}\}$.
By Proposition~\ref{P:hb3},
\[h_{\vec{b}}(x)=\Bigl(\sum_{1\leq j\leq l}\frac{a_{i_{j}}^{2}b_{i_{j}}^{2}}{|\beta|^{2}}
 \prod_{1\leq k\leq l,\, k\neq j}(x-a_{i_{k}}^{2})\Bigr)\prod_{i\in J}(x-a_{i}^{2}).\]
Thus, $a_{i}^{2}$ ($i\in J$) are zeros of
$h_{\vec{b}}(x)$. By a similar argument as above, one shows that other $l-1$ zeros of $h_{\vec{b}}(x)$ lie in
the intervals  \[(a_{i_{l}}^{2},a_{i_{l-1}}^{2}),\dots,(a_{i_{2}}^{2},a_{i_{1}}^{2})\] respectively. This shows
that: $h_{\vec{b}}(x)$ has $n-1$ positive roots, which lie in the intervals \[[a_{n}^{2},a_{n-1}^{2}],\dots,
[a_{2}^{2},a_{1}^{2}]\] respectively.
\end{proof}

By Corollary \ref{C:hb4}, $h_{\vec{b}}(x)$ has at most double zero,
 and the only possible double zeros are
 $a_{2}^{2},\dots,a_{n-1}^{2}$;
 for each $2\leq i\leq n-2$, $a_{i}^{2}$ and $a_{i+1}^{2}$ cannot be
both double zeros.
By \eqref{Eq:hb2}, in order that $a_{i}^{2}$ for $2\leq i\leq n-1$ is a double zero of
$h_{\vec{b}}(x)$ it is necessary and sufficient that
 $b_{i}=0$ and
\[\sum_{1\leq k\leq n,\, k\neq i}\frac{a_{k}^{2}b_{k}^{2}}{|\beta|^{2}}
 \prod_{1\leq j\leq n,\, j\neq i,k}(a_{i}^{2}-a_{j}^{2})=0.\]

\smallskip

Let $x_1\geq\cdots\geq x_{n-1}\geq 0$ be square roots of zeros of $h_{\vec{b}}(x)$.
By Corollary \ref{C:hb4},
$a_{i+1}\leq x_{i}\leq a_{i}$ for each $1\leq i\leq n-2$,
 and $|a_{n}|\leq x_{n-1}\leq a_{n-1}$.
Write \[\vec{x}=(x_{1},\dots,x_{n-1}).\]

\begin{corollary}\label{C:characteristic3}
The map $\vec{b}\mapsto\vec{x}$
 gives a bijection from $B$ to
\[[a_{2},a_{1}]\times\cdots\times[a_{n-1},a_{n-2}]\times[|a_{n}|,a_{n-1}].\]
\end{corollary}

\begin{proof}
For $\vec{b}=(b_1,\dots,b_{n})\in B$ and $\vec{b'}=(b'_{1},\dots,b'_{n})\in B$, suppose $h_{\vec{b}}(x)$ and $h_{\vec{b'}}(x)$ have the same zeros.
Then, $h_{\vec{b}}=h_{\vec{b'}}$.
Thus, $h_{\vec{b}}(a_{i}^{2})=h_{\vec{b'}}(a_{i}^{2})$ for each $1\leq i\leq n$.
By Proposition \ref{P:hb3}, this implies that
$b_{i}=b'_{i}$ for each $i$. Thus, $\vec{b}=\vec{b'}$.
This shows the injectivity.

The singular values $x_1,\dots,x_{n-1}$ gives a polynomial
\[p(x)=\prod_{1\leq i\leq n-1}(x-x_{i}^2).\]
Since $(-1)^{i-1}f(a_{i}^{2})\geq 0$, we can write
\[p(x)= \sum_{1\leq i\leq n} c_i \prod_{1\leq j\leq n,\, j\neq i}(x-a_{j}^2).\]
 for some $c_i\geq 0$ with $\sum_{i=1}^n c_i=1$.
Hence there exists a unique tuple $\vec{b}\in B$
 such that $h_{\vec{b}}(x)=p(x)$.
This shows the surjectivity.
\end{proof}

\begin{proposition}\label{P:pf1}
The image of the moment map $q(\mathcal{O}_{f})$
 consists of all depth one coadjoint orbits of $P$
 with the sign of the Pfaffian equal to the sign of $a_{n}$,
 and with singular values $(x_1,\dots, x_{n-1})$
 such that
\[a_1\geq x_1\geq a_2\geq x_2\geq \cdots\geq
 a_{n-1}\geq x_{n-1}\geq |a_{n}|.\]
Moreover, $q$ maps different $P$-orbits in
 $\mathcal{O}_{f}$ to different $P$-orbits in $\mathfrak{p}^{\ast}$.
\end{proposition}

\begin{proof}
By the form of $q(\bar{n}_{\vec{b}}\cdot f)$ in Lemma \ref{L:gf}, we have $\beta\neq 0$. Thus, the $P$-orbit containing
$q(\bar{n}_{\vec{b}}\cdot f)$ has depth one. By Lemma \ref{L:Zb2}, the Pfaffian of $Z_{\vec{b}}$ has the same sign as
the sign of $a_{n}$. The other statements follow from Corollary \ref{C:characteristic3}.
\end{proof}

The stabilizers of orbits are given an follows.

\begin{proposition}\label{P:pf3}
For any $\vec{b}\in B$,
\[\Stab_{P_{1}}(\bar{n}_{\vec{b}}\cdot f)\cong\SO(2)^{r}
\text{ and }
\Stab_{P_{1}}(q(\bar{n}_{\vec{b}}\cdot f)) \cong
 \U(2)^{s}\times\SO(2)^{n-1-2s},\]
 where $r$ is the number of zeros among $b_{1},\dots,b_{n}$
 and $s$ is the number of double zeros of $h_{\vec{b}}(x)$.
\end{proposition}

\begin{proof}
The first isomorphism
 follows from $\Stab_{P_{1}}(\bar{n}_{\vec{b}}\cdot f)
 \cong \Stab_{T_{1}}(\bar{n}_{\vec{b}}^{-1}\cdot v_0)$.
The second isomorphism follows from
 the description of orbit $\Stab_{P_{1}}(q(\bar{n}_{\vec{b}}\cdot f))$ above
 and Lemma~\ref{L:p-standard3}~(3).
\end{proof}

\begin{lemma}\label{L:proper1}
For any compact set $\Omega\subset\mathfrak{p}^{\ast}-\mathfrak{l}^{\ast}$, $q^{-1}(\Omega)$ is compact.
%
\end{lemma}

\begin{proof}
Write a general element in $q^{-1}(\Omega)$ as
\[f'=an'm \bar{n}_{\vec{b}}\cdot f, \]
 where $a\in A$, $n'\in N$, $m\in M$, and $\vec{b}\in B$.
Then \[q(f')=an'm\cdot q(\bar{n}_{\vec{b}}\cdot f)\in\Omega.\]
Note that $M$ and $B$ are compact.
Hence, $m$ and $\vec{b}$ are bounded.
Write \[m\cdot q(\bar{n}_{\vec{b}}\cdot f)=\eta_{1}+\phi_{n}(\xi_{1}),\]
 where $m^{-1}\cdot \xi_1$ is given by the vector $\beta$ as in Lemma \ref{L:gf}.
Then
\[\frac{1}{2}|a_{n}|\leq|a_{n}||\vec{b}|\leq |\xi_{1}|
 =|\beta|\leq|a_{1}||\vec{b}|\leq|a_{1}|,\]
 where we used $\frac{1}{2}\leq|\vec{b}|=1-b_{n}\leq 1$.
Write $an'm\cdot q(\bar{n}_{\vec{b}}\cdot f)=\eta+\phi_{n}(\xi)$,
 where $\eta\in\mathfrak{l}^{\ast}$ and $\xi\in\mathfrak{n}^{\ast}$.
By the compactness of
 $\Omega\subset\mathfrak{p}^{\ast}-\mathfrak{l}^{\ast}$,
 $|\eta|$ is bounded from above,
 and $|\xi|$ is bounded from both above and below.
We have \[an'\cdot(\eta_{1}+\phi_{n}(\xi_{1}))=\eta+\phi_{n}(\xi).\]
Write $n'=\exp(X)$ and $\xi_{1}=\pr(Y)$, where $X\in\mathfrak{n}$ and $Y\in\bar{\mathfrak{n}}$.
Then,
\[an'\cdot(\eta_{1}+\phi_{n}(\xi_{1}))
 =\eta_{1}+\pr([X,Y])+e^{-\lambda_{0}\log a}\phi_{n}(\xi_{1}).\]
Thus, $\eta=\eta_{1}+\pr([X,Y])$ and $\xi=e^{-\lambda_{0}\log a} \xi_{1}$.
Now, $|\xi_{1}|,|Y|,|\xi|$ are bounded both from above and below,
 and $|\eta|,|\eta_{1}|$ are bounded from above.
Thus, $\log a$ is bounded both from above and below, and $|X|$ is bounded from above.
This shows the compactness of $q^{-1}(\Omega)$.
\end{proof}

\begin{proposition}\label{P:pf2}
The moment map $q\colon \mathcal{O}_{f}=G\cdot f\rightarrow\mathfrak{p}^{\ast}$ has image lies in the set of depth
one elements. For any $g\in G$, the reduced space $q^{-1}(q(g\cdot f))/\Stab_{P}(q(g\cdot f))$ is a singleton.
The moment map $q$ is weakly proper, but not proper.
\end{proposition}

\begin{proof}
The first and the second statements follow from Proposition \ref{P:pf1}.
By Lemma \ref{L:proper1}, we see that $q$ is weakly proper.
Since the closure of every depth one orbit contains a depth zero orbit,
 $q(\mathcal{O}_{f})$ is not closed in $\mathfrak{p}^{\ast}$.
Hence $q\colon \mathcal{O}_{f}\to \mathfrak{p}^*$ is not proper.
\end{proof}

\subsection{Singular elliptic coadjoint orbits}\label{SS:singular-elliptic}

Now we consider singular elliptic coadjoint orbits. That is, we allow some of the singular values $a_{1},\dots,a_{n}$
to be equal.

Consider the case when $a_{n}\neq 0$.
Let $0=i_{0}<i_{1}<\cdots<i_{l}=n$ be such that $|a_{i}|=|a_{j}|$
 if $i_{k-1}<i\leq j\leq i_{k}$ for some $1\leq k\leq l$,
 and $|a_{i_{k}}|>|a_{i_{k}+1}|$ for any $1\leq k\leq l-1$.
Write $n_{k}=i_{k}-i_{k-1}$ for $1\leq k\leq l$.
Then
\begin{equation*}
\Stab_{G_{1}}(f)=\U(n_{1})\times \cdots\times\U(n_{l}).
\end{equation*}
Put
\begin{equation*}
B_{l}=\{\vec{b}=(b'_{1},\dots,b'_{l}):b'_{i}
\geq 0,\sum_{1\leq i\leq l-1}b_{i}'^{2}=1-2b'_{l}\}.
\end{equation*}
For any $\vec{b}=(b'_{1},\dots,b'_{l})\in B_{l}$,
 we have $0\leq b'_{l}\leq\frac{1}{2}$.
For $\vec{b}\in B_{l}$, write
\begin{equation*}
\alpha= \alpha_{\vec{b}}=(0,b_1,0,b_{2},\dots,0,b_{n-1},0),
\end{equation*}
 where $b_{i}=b'_{k}$ if $i=i_{k}$ for some $1\leq k\leq l-1$,
 and $b_{i}=0$ if otherwise.
Put $b_{n}=b'_{l}$.
Put
\begin{equation*}
\bar{X}_{\vec{b}}=\begin{pmatrix}
 0_{2n-1}&\alpha^{t}&\alpha^{t}\\
 -\alpha&0&0\\
 \alpha&0&0\\
\end{pmatrix} \text{ and }
\bar{n}_{\vec{b}}=\exp(\bar{X}_{\vec{b}})
\end{equation*}
Then
\begin{equation*}
\bar{n}_{\vec{b}}^{-1}\cdot
v_{0}=[(0,-b_1,0,-b_2,\dots,0,-b_{n-1},0,
b_{n},1-b_{n})].
\end{equation*}

We have the following analogue of Lemma \ref{L:P-orbits}.

\begin{lemma}\label{L:P-orbits3}
Each $P$-orbit in $\mathcal{O}_{f}=G\cdot f$ contains some $\bar{n}_{\vec{b}}\cdot f$
 for a unique tuple $\vec{b} \in B_{l}$.
\end{lemma}

With this, we employ an argument similarly to the case of regular orbits.
Then Lemma \ref{L:Zb2}, Proposition \ref{P:hb3}, Corollary \ref{C:hb4}, Corollary \ref{C:characteristic3}, Proposition \ref{P:pf1} and Lemma \ref{L:proper1}
 all extend to this singular case.
Then Proposition \ref{P:pf2} extends without change of words.

Similarly to Proposition \ref{P:pf3}, we can describe $\Stab_{P_{1}}(\bar{n}_{\vec{b}}\cdot f)$ and
$\Stab_{P_{1}}(q(\bar{n}_{\vec{b}}\cdot f))$ for $\vec{b}\in B_{l}$. Let $n_{k}+s_{k}$ be the multiplicity
of $a_{k}^{2}$ for $1\leq k\leq l$ as a zero of $h_{\vec{b}}(x)$.
Then $s_{k}\in\{-1,0,1\}$; and $s_{k}\in\{0,1\}$ if and only if $b'_{k}=0$.
Put \[r_{k}=\Bigl\lfloor\frac{s_{k}}{2}\Bigr\rfloor\text{ and }s=-\sum_{1\leq k\leq l}s_{k}.\]
Then $r_{k}\in\{-1,0\}$, and $r_{k}=0$ if and only if $b'_{k}=0$.
The stabilizers are given as follows.

\begin{proposition}\label{P:Stab2}
For $\vec{b}\in B_{l}$,
\[\Stab_{P_{1}}(\bar{n}_{\vec{b}}\cdot f)\cong\U(n_{1}+r_{1})\times\cdots\times\U(n_{l}+r_{l})\]
and \[\Stab_{P_{1}}(q(\bar{n}_{\vec{b}}\cdot f))\cong\U(n_{1}+s_{1})
 \times\cdots\times\U(n_{l}+s_{l})\times\U(1)^{s}.\]
\end{proposition}


Elliptic orbits for $a_{n}=0$
 can be regarded as degenerations of
 the non-elliptic regular orbits which are studied in the next section.
Thus, we treat them in the next section.

\section{Moment map for non-elliptic semisimple coadjoint orbits}\label{S:non-elliptic}

We continue the study of moment map of coadjoint orbits.
In this section we treat the remaining types of orbits.

Suppose first that $m$ is odd and $m=2n-1$.
The case where $m$ is even will be mentioned at the end of this section.

\subsection{$P$-orbits in $\mathcal{O}_{f}$}\label{SS:doubleCoset2}

Let $a_1\geq a_2\geq\cdots\geq a_{n-1}\geq 0$ and $a_{n}\geq 0$.
Write $\vec{a}=(a_1,\dots,a_{n})$ and
\begin{align*}
t'_{\vec{a}}
=\begin{pmatrix}
0&a_1&& &&&&\\
-a_{1}&0&&&&&&\\
&&\ddots&&&&&\\
&&&0&a_{n-1}&&&\\&&&-a_{n-1}&0&&&\\
&&&&&0&&\\&&&&&&0&a_{n}\\&&&&&&a_{n}&0\\
\end{pmatrix}.
\end{align*}
Put $f=\iota(t'_{\vec{a}})$.
Then $\mathcal{O}_{f}=G\cdot f$ is a non-elliptic semisimple
 coadjoint orbit in $\mathfrak{g}^{\ast}$ if $a_n>0$.
All non-elliptic semisimple coadjoint orbits in $\mathfrak{g}^{\ast}$
 are of this form.
We first consider regular orbits,
 that is, we assume that $a_1>a_2>\cdots>a_{n-1}>0$ and $a_{n}>0$.
Let $G^{f}$ be the stabilizer of $G$ at $f$.
Then $G^{f}=T_{s}$, where $T_{s}$ is the pre-image in $G$ of a maximal torus
\begin{align*}
&\Biggl\{
\begin{pmatrix}
y_{1}&z_{1}&&&&&&\\
-z_{1}&y_{1}&&&&&&\\
&&\ddots&&&&&\\
&&&y_{n-1}&z_{n-1}&&&\\
&&&-z_{n-1}&y_{n-1}&&&\\
&&&&&1&&\\
&&&&&&y_{n}&z_{n}\\
&&&&&&z_{n}&y_{n}\\
\end{pmatrix}\\
 &\qquad\qquad
:y_{1}^{2}+z_{1}^{2}=\cdots=y_{n-1}^{2}+z_{n-1}^{2}
 =y_{n}^{2}-z_{n}^{2}=1,\, y_{n}>0
\Biggr\}
\end{align*} of $G_{1}$.
As in the previous section,
 we identify $G/P$ with the set
\[X_{n}=\{\vec{x}=(x_1,\dots,x_{2n},x_{0}):x_{0}^{2}
 =\sum_{1\leq i\leq 2n}x_{i}^{2},\, x_{0}>0\}/\sim\]
 and consider $T_s$-orbits in $X_n$.

Set \begin{equation*}
B'=\{(b_1,\dots,b_{n-1},b_{n}):b_1,\dots,b_{n-1} \geq 0,
 \sum_{1\leq i\leq n}b_{i}^{2}=1\}.
\end{equation*}
For each $\vec{b}\in B'$, write
\begin{align}\label{Eq:Xbnb}
&\alpha=\alpha_{\vec{b}}=
 \Bigl(\frac{a_{n}b_1}{\sqrt{a_{n}^{2}+a_{1}^{2}}},\frac{-a_{1}b_1}
{\sqrt{a_{n}^{2}+a_{1}^{2}}},\dots,\frac{a_{n}b_{n-1}}{\sqrt{a_{n}^{2}+a_{n-1}^{2}}},
\frac{-a_{n-1}b_{n-1}}{\sqrt{a_{n}^{2}+a_{n-1}^{2}}},b_{n}\Bigr), \\ \nonumber
&\qquad \bar{X}_{\vec{b}}
 =\begin{pmatrix} 0_{2n-1}&\alpha^{t}&\alpha^{t}\\
  -\alpha&0&0\\
  \alpha&0&0\\
 \end{pmatrix}
\text{ and }
\bar{n}_{\vec{b}}=\exp(\bar{X}_{\vec{b}}).
\end{align}
Note that $|\alpha|=1$. Let $g_{\infty}=1$ and let $g'_{\infty}$ be a pre-image in $G$ of
 $\diag\{I_{2n-2},-1,-1,1\}\in G_1$.
Then $\bar{n}_{\vec{b}}^{-1}\cdot v_{0}$ is equal to
\begin{align}\label{Eq:nbv0}
\Bigl[\Bigl(\frac{-a_{n}b_1}{\sqrt{a_{n}^{2}+a_{1}^{2}}},
 \frac{a_{1}b_1}{\sqrt{a_{n}^{2}+a_{1}^{2}}},\dots,
 \frac{-a_{n}b_{n-1}}{\sqrt{a_{n}^{2}+a_{n-1}^{2}}},
 \frac{a_{n-1}b_{n-1}}{\sqrt{a_{n}^{2}+a_{n-1}^{2}}},
-b_{n},0,1\Bigr)\Bigr],
\end{align}
and
\[g_{\infty}^{-1}\cdot v_{0}=v_{0}, \quad (g'_{\infty})^{-1}\cdot v_{0}=v'_{0},\]
 where $v'_0:=[(0,\dots,0,-1,1)]$.
 Note that $v_0$ and $v'_{0}$ are the only $T_{s}$-fixed points in $X_{n}$.
The following lemma is easy to show.

\begin{lemma}\label{L:Ts-orbits}
The map
\[B'\rightarrow(X_{n}-\{v_{0},v'_{0}\})/T_{s},\quad
\vec{b}\mapsto (\bar{n}_{\vec{b}})^{-1} \cdot v_{0}\]
 is a bijection.
\end{lemma}

The following lemma directly follows form Lemma \ref{L:Ts-orbits}.

\begin{lemma}\label{L:P-orbits2}
Each $P$-orbit in $\mathcal{O}_{f}-(P\cdot f\sqcup Pg'_{\infty}\cdot f)$
 contains some $\bar{n}_{\vec{b}}\cdot f$ for a unique tuple $\vec{b}\in B'$.
\end{lemma}

Put
\[H'=\begin{pmatrix}0 & 1 \\ -1 & 0 \end{pmatrix}
 \text{ and } H=\begin{pmatrix}0 & 1 \\ 1 & 0 \end{pmatrix}.\]
By Lemma \ref{L:P-orbits2}, $\{\bar{n}_{\vec{b}}\cdot f:\vec{b}\in B'\}$
 together with $g_{\infty}\cdot f=f$ and
\[g'_{\infty}\cdot f=\iota(\diag\{a_{1}H',\dots,a_{n-1}H',0,-a_{n}H\})\]
 represent all $P$-orbits in $\mathcal{O}_{f}=G\cdot f$.

\subsection{The moment map $\mathcal{O}_{f}\rightarrow\mathfrak{p}^{\ast}$}\label{SS:moment2}

Let $Y=\diag\{a_{1}H',\dots,a_{n-1}H',0\}$.
We have $t'_{\vec{a}}=\diag\{Y,0_{2\times 2}\}+a_nH_0\in \mathfrak{m}+\mathfrak{a}$.
For each $\vec{b}\in B'$, by \eqref{Eq:brackets},
\begin{align*}
\Ad(\bar{n}_{\vec{b}}) (t'_{\vec{a}})
&= t_{\vec{a}} + [\bar{X}_{\alpha}, t_{\vec{a}}] \\
&= \diag\{Y,0_{2\times 2}\} + a_nH_0
 + \bar{X}_{\alpha Y} + a_n \bar{X}_{\alpha}.
\end{align*}
Hence
\[\Ad(\bar{n}_{\vec{b}})(t'_{\vec{a}})
=\begin{pmatrix}
  Y& \beta^{t}&\beta^{t}\\
 -\beta &0&a_{n}\\
 \beta&a_{n}&0\\
\end{pmatrix},
\] where
\[\beta=a_{n}\alpha+\alpha Y
=(b_{1}\sqrt{a_{n}^{2}+a_{1}^{2}},0,\dots, b_{n-1}\sqrt{a_{n}^{2}+a_{n-1}^{2}},0,a_{n}b_{n}).\]
Put
\[Y_{\vec{b}}=Y-\frac{1}{|\beta|^2}(Y\beta^{t}\beta-(Y\beta^{t}\beta)^{t}),
\quad
Z_{\vec{b}}=\begin{pmatrix}
 Y_{\vec{b}}&\frac{\beta}{|\beta|}\\-\frac{\beta}{|\beta|}&0\\
\end{pmatrix}.\]
We have
\[Y\beta^{t}
 =(0,-a_{1}b_{1}\sqrt{a_{n}^{2}+a_{1}^{2}},\dots,
 0,-a_{n-1}b_{n-1}\sqrt{a_{n}^{2}+a_{n-1}^{2}},0)^{t}.\]
Note that we always have $\beta\neq 0$.
By Lemmas \ref{L:p-standard4}, \ref{p:standard5} and \ref{L:class-matrix},
 the $P$-conjugacy class of $q(\bar{n}_{\vec{b}}\cdot f)$
 is determined by the sign of the Pfaffian of $Z_{\vec{b}}$
 and singular values of $Y_{\vec{b}}$.

Put
\begin{align*}
&\gamma_{1}=(b_{1}\sqrt{a_{n}^{2}+a_{1}^{2}},\dots,b_{n-1}\sqrt{a_{n}^{2}+a_{n-1}^{2}},a_{n}
b_{n}),\\
&\gamma_{2}=(a_{1}b_{1}\sqrt{a_{n}^{2}+a_{1}^{2}},\dots,
 a_{n-1}b_{n-1}\sqrt{a_{n}^{2}+a_{n-1}^{2}},-|\beta|).
\end{align*}
By a direct calculation we have
\begin{lemma}\label{L:Zb3}
Let $\sigma$ be the permutation
\[\sigma(i) =\begin{cases} 2i-1 &(1\leq i\leq n) \\
 2(i-n) & (n+1\leq i\leq 2n) \end{cases}.\]
Then
\[Q_{\sigma}Z_{\vec{b}} Q_{\sigma}^{-1}=
\begin{pmatrix} 0_{n}&Z\\
 -Z^{t}&0_{n}\\
\end{pmatrix},\] where
\[Z=\diag\{a_{1},\dots,a_{n-1},0\}-\frac{1}{|\beta|^2}\gamma_{1}^{t}\gamma_{2}.\]
\end{lemma}

By Lemma~\ref{L:Zb3},
 the Pfaffian of $Z_{\vec{b}}$ equals $\det Z$.
Hence
\[\Pf(Z_{\vec{b}})=\frac{b_{n}}{|\beta|}\prod_{1\leq i\leq n}a_{i}.\]

Write $Z'$ for the $n\times (n-1)$ submatrix of $Z$ by removing the last column.
Then the singular values of $Y_{\vec{b}}$ are square roots of eigenvalues of $(Z')^{t}Z'$
 as we saw in \S\ref{SS:moment}.
Write
\[h_{\vec{b}}(x)=\det(xI_{n-1}-(Z')^{t}Z').\]

\begin{lemma}\label{Ls:characteristic1}
We have \[h_{\vec{b}}(x)=\prod_{1\leq i\leq n-1}(x-a_{i}^{2})
+\sum_{1\leq i\leq n-1}\Bigl(\prod_{1\leq j\leq n-1,\, j\neq i}(x-a_{j}^{2})\Bigr)
\frac{a_{i}^{2}b_{i}^{2}(a_{n}^{2}+a_{i}^{2})}{|\beta|^2}.\]
\end{lemma}

\begin{proof}
Put
\[\gamma_{3}=\frac{1}{|\beta|}
\bigl(a_{1}b_{1}\sqrt{a_{n}^{2}+a_{1}^{2}},\dots,
a_{n-1}b_{n-1}\sqrt{a_{n}^{2}+a_{n-1}^{2}}\bigr).\]
We calculate \[(Z')^{t}Z'=\diag\{a_{1}^{2},\dots,a_{n-1}^{2}\}-\gamma_{3}^{t}\gamma_{3}.\]
Then, one shows easily that
\begin{align*}
&\det(xI_{n-1}-(Z')^{t}Z') \\
&=\prod_{1\leq i\leq n-1}(x-a_{i}^{2})
 +\sum_{1\leq i\leq n-1}\Bigl(\prod_{1\leq j\leq n-1,\, j\neq i}(x-a_{j}^{2})\Bigr)
\frac{a_{i}^{2}b_{i}^{2}(a_{n}^{2}+a_{i}^{2})}{|\beta|^2}. \qedhere
\end{align*}
\end{proof}

\begin{proposition}\label{Ps:determinant}
For $1\leq i\leq n-1$,
\[h_{\vec{b}}(a_{i}^{2})=\frac{a_{i}^{2}b_{i}^{2}(a_{n}^{2}+
a_{i}^{2})}{|\beta|^{2}}\prod_{1\leq j\leq n-1,\, j\neq i}(a_{i}^{2}-a_{j}^{2});\]
 and \[h_{\vec{b}}(0)=(-1)^{n-1}
\frac{a_{n}^{2}b_{n}^{2}}{|\beta|^{2}}\prod_{1\leq i\leq n-1}a_{i}^{2}.\]
\end{proposition}

\begin{proof}
This follows by an easy calculation from Lemma \ref{Ls:characteristic1}.
\end{proof}

With Proposition \ref{Ps:determinant} substituting Proposition \ref{P:hb3}, the following corollary can be shown in the same way as for Corollary \ref{C:hb4}.

\begin{corollary}\label{Cs:characteristic2}
The polynomial $h_{\vec{b}}(x)$ has $n-1$ non-negative roots,
 which lie in the intervals
\[ [0,a_{n-1}^{2}], [a_{n-1}^{2},a_{n-2}^{2}],\dots,[a_{2}^{2},a_{1}^{2}]\]
 respectively; $a_{i}^{2}$ ($1\leq i\leq n-1$)
 is a root if and only if $b_{i}=0$; $0$ is a root if and only if $b_{n}=0$.
\end{corollary}

By Corollary \ref{Cs:characteristic2}, $h_{\vec{b}}(x)$ has at most double zeros; the only possible double zeros of it are $a_{2}^{2},\dots,a_{n-1}^{2}$;
 $a_{i}^{2}$ and $a_{i+1}^{2}$ cannot both be double zeros.
Moreover, $a_{i}^{2}$ for $2\leq i\leq n-1$ is a double zero if and only if $b_{i}=0$ and
\[\sum_{1\leq k\leq n,\, k\neq i}\frac{a_{k}^{2}b_{k}^{2}}{|\beta|^{2}}
 \prod_{1\leq j\leq n,\, j\neq i,k}(a_{i}^{2}-a_{j}^{2})=0.\]

By Corollary \ref{Cs:characteristic2},
 write $x_1^{2}\geq\cdots\geq x_{n-1}^{2}$ for zeros of $h_{\vec{b}}(x)$.
Choose $x_1,\dots,x_{n-1}$ such that $x_i\geq 0$ for $1\leq i\leq n-2$
 and $\sgn x_{n-1} = \sgn b_n$.
Then
\[a_1\geq x_1\geq a_2\geq x_2\geq \cdots \geq a_{n-1}\geq |x_{n-1}|.
\]
Write
$\vec{x}=(x_{1},\dots,x_{n-1})$.

The following corollary and propositions are analogues of Corollary \ref{C:characteristic3}, Proposition \ref{P:pf1},
Proposition \ref{P:pf3} respectively. The proofs are similar.

\begin{corollary}\label{Cs:characteristic3}
The map $\vec{b}\mapsto \vec{x}$
 gives a bijection from $B'$ to
\[[a_{2},a_{1}]\times\cdots\times[a_{n-1},a_{n-2}]
\times[-a_{n-1}, a_{n-1}].\]
\end{corollary}

\begin{proposition}\label{Ps:pf1}
The image of the moment map $q(\mathcal{O}_f)$
 consists of two depth zero orbits $P\cdot f$, $Pg'_{\infty}\cdot f$,
 and all depth one $P$-coadjoint orbits with singular values
 $(x_1,\dots, x_{n-1})$ such that
\[a_1\geq x_1\geq a_2\geq x_2\geq \cdots\geq
 a_{n-1}\geq x_{n-1}\geq 0.\]
Moreover, $q$ maps different $P$-orbits in $\mathcal{O}_{f}$
 to different $P$-orbits in $\mathfrak{p}^{\ast}$.
\end{proposition}

\begin{proposition}\label{Ps:pf3}
For any $\vec{b}\in B'$,
\[\Stab_{P_{1}}(\bar{n}_{\vec{b}}\cdot f)\cong\SO(2)^{r}\]
 and
\[\Stab_{P_{1}}(q(\bar{n}_{\vec{b}}\cdot f)) \cong\U(2)^{s}\times
 \SO(2)^{n-1-2s},\]
 where $r$ is the number of zeros among $b_{1},\dots,b_{n-1}$ and
 $s$ is the number of double zeros of $h_{\vec{b}}(x)$.
\end{proposition}

Lemma \ref{L:proper1} has the following analogue. The proof is the same.

\begin{lemma}\label{Ls:proper1}
For any compact set $\Omega\subset\mathfrak{p}^{\ast}-\mathfrak{l}^{\ast}$, $q^{-1}(\Omega)$ is compact.
%
\end{lemma}


By Proposition \ref{Ps:pf1} and Lemma \ref{Ls:proper1}, the following proposition follows.

\begin{proposition}\label{Ps:pf2}
For any $g\in G$, the reduced space
\[q^{-1}(q(g\cdot f))/\Stab_{P}(q(g\cdot f))\] is a singleton.
The moment map $q$ is weakly proper, but not proper.
\end{proposition}

\begin{proof}
The first and the second claims are direct consequences of
 Proposition \ref{Ps:pf1} and Lemma \ref{Ls:proper1}, respectively.
For the last claim it is enough to see that $q^{-1}(q(f))$ is non-compact.
\end{proof}

\subsection{Singular semisimple coadjoint orbits}\label{SS:singular-nonelliptic}

Now we consider singular semisimple coadjoint orbits, that is, we allow some of the singular values
$a_{1},\dots,a_{n-1}$ to be equal, $a_{n-1}=0$ or $a_{n}=0$.
First consider the case when $a_{n-1}\neq 0$ and $a_{n}\neq 0$.
Let $0=i_{0}<i_{1}<\cdots<i_{l-1}=n-1$ be such that
 $a_{i}=a_{j}$ if $i_{k-1}<i\leq j\leq i_{k}$ for some $1\leq k\leq l-1$,
 and $a_{i_{k}}>a_{i_{k}+1}$ for any $1\leq k\leq l-2$.
Write $n_{k}=i_{k}-i_{k-1}$ for $1\leq k\leq l-1$.
Then,
\[\Stab_{G}(f)\cong\U(n_{1})\times\cdots\times\U(n_{l-1})
\times\mathbb{R}_{>0}.\]
Put
\[B'_{l}=\{\vec{b}=(b'_{1},\dots,b'_{l}):
b'_{1},\dots,b'_{l-1}\geq 0,\sum_{1\leq i\leq l}(b_{i}')^{2}=1\}.
\]
For each $\vec{b}=(b'_{1},\dots,b'_{l})\in B'_{l}$, write
\begin{equation*}
\alpha=\alpha_{\vec{b}}=
\Bigl(\frac{a_{n}b_1}
{\sqrt{a_{n}^{2}+a_{1}^{2}}},\frac{-a_{1}b_1}{\sqrt{a_{n}^{2}+a_{1}^{2}}},\dots,\frac{a_{n}b_{n-1}}
{\sqrt{a_{n}^{2}+a_{n-1}^{2}}},\frac{-a_{n-1}b_{n-1}}{\sqrt{a_{n}^{2}+a_{n-1}^{2}}},b_{n}\Bigr),
\end{equation*}
where $b_{i}=b'_{k}$ if $i=i_{k}$ for some $1\leq k\leq l-1$;
 $b_{i}=0$ if $n-1\geq i\neq i_{k}$ for any $1\leq k\leq l-1$; and $b_{n}=b'_{l}$.
Put $\bar{X}_{\vec{b}}$ and $\bar{n}_{\vec{b}}$ as in \eqref{Eq:Xbnb}.
Then $\bar{n}_{\vec{b}}^{-1}\cdot v_{0}$ equals \eqref{Eq:nbv0}.

We have the following analogue of Lemma \ref{L:P-orbits2}.

\begin{lemma}\label{Ls:P-orbits2}
Each $P$-orbit in $\mathcal{O}_{f}-(P\cdot f\sqcup Pg'_{\infty}\cdot f)$
 contains some $\bar{n}_{\vec{b}}\cdot f$ for a unique tuple $\vec{b}\in B'_{l}$.
\end{lemma}

With this, we take a similar study as for regular non-elliptic semisimple orbits.
All the results in the previous subsection can be extended.
In particular, Propositions \ref{Ps:pf1} and \ref{Ps:pf2} extend without change of words.

\medskip

Next, consider the case when $a_{n-1}=0$ and $a_{n}\neq 0$.
Let $0=i_{0}<i_{1}<\cdots<i_{l}=n-1$ be such that
$a_{i}=a_{j}$ if $i_{k-1}<i\leq j\leq i_{k}$ for some $1\leq k\leq l$,
 and $a_{i_{k}}>a_{i_{k}+1}$ for any $1\leq k\leq l-1$.
Write $n_{k}=i_{k}-i_{k-1}$ for $1\leq k\leq l$.
Then,
\[\Stab_{G_{1}}(f)\cong\U(n_{1})\times\cdots\times\U(n_{l-1})\times\SO(2n_{l}+1)
 \times\mathbb{R}_{>0}.\]
Put
\[B'_{l}=\{\vec{b}=(b'_{1},\dots,b'_{l}):b'_1,\dots,b'_{l}\geq 0,
 \sum_{1\leq i\leq l}(b_{i}')^{2}=1\}.\]
For each $\vec{b}=(b'_{1},\dots,b'_{l})\in B'_{l}$,
 write
\[\alpha=\alpha_{\vec{b}}=
\Bigl(\frac{a_{n}b_1}{\sqrt{a_{n}^{2}+a_{1}^{2}}},
 \frac{-a_{1}b_1}{\sqrt{a_{n}^{2}+a_{1}^{2}}},
 \dots,\frac{a_{n}b_{n-1}}{\sqrt{a_{n}^{2}+a_{n-1}^{2}}},
 \frac{-a_{n-1}b_{n-1}}{\sqrt{a_{n}^{2}+a_{n-1}^{2}}},b_{n}\Bigr),
\] where $b_{i}=b'_{k}$ if $i=i_{k}$ for some $1\leq k\leq l-1$;
 $b_{i}=0$ if $n-1\geq i\neq i_{k}$ for any $1\leq k\leq l-1$;
 and $b_{n}=b'_{l}$.
Put $\bar{X}_{\vec{b}}$ and $\bar{n}_{\vec{b}}$ as in \eqref{Eq:Xbnb}.
Then $\bar{n}_{\vec{b}}^{-1}\cdot v_{0}$ equals \eqref{Eq:nbv0}.

We have the following analogue of Lemma \ref{L:P-orbits2}.

\begin{lemma}\label{Ls:P-orbits3}
Each $P$-orbit
 in $\mathcal{O}_{f}-(P\cdot f\sqcup Pg'_{\infty}\cdot f)$
 contains some $\bar{n}_{\vec{b}}\cdot f$ for a unique tuple $\vec{b}\in B'_{l}$.
\end{lemma}

With this, we take a similar study as for regular non-elliptic semisimple orbits.
All results can be extended.
In particular, Propositions \ref{Ps:pf1} and \ref{Ps:pf2} extend without change of words.

\medskip

Next, consider the case when $a_{n}=0$ and $a_{1}\neq 0$.
Then $\mathcal{O}_f$ is elliptic.
Let $0=i_{0}<i_{1}<\cdots<i_{l}=n$ be such that $a_{i}=a_{j}$ if
$i_{k-1}<i\leq j\leq i_{k}$ for some $1\leq k\leq l$,
 and $a_{i_{k}}>a_{i_{k}+1}$ for any $1\leq k\leq l-1$.
Write $n_{k}=i_{k}-i_{k-1}$ for $1\leq k\leq l$.
Then,
\[\Stab_{G_{1}}(f)\cong\U(n_{1})\times
 \cdots\times\U(n_{l-1})\times\SO_{e}(2n_{l},1).
\]
Put
\[B'_{l}=\{\vec{b}=(b'_{1},\dots,b'_{l-1}):
b'_{1},\dots,b'_{l-1}\geq 0,\sum_{1\leq i\leq l-1}(b_{i}')^{2}=1\}.
\]
For each $\vec{b}=(b'_{1},\dots,b'_{l})\in B'_{l}$,
 write
\[\alpha=\alpha_{\vec{b}}=(0,-b_{1},\dots,0,-b_{n-1},0),\]
where $b_{i}=b'_{k}$ if $i=i_{k}$ for some $1\leq k\leq l-1$;
 $b_{i}=0$ if $i\neq i_{k}$ for any $1\leq k\leq l-1$.
Put $\bar{X}_{\vec{b}}$ and $\bar{n}_{\vec{b}}$ as in \eqref{Eq:Xbnb}.
Then,
\[\bar{n}_{\vec{b}}^{-1}\cdot v_{0}=[(0,b_{1},\dots,0,b_{n-1},0,0,1)].\]
Because of $a_{n}=0$, one has $g'_{\infty}\cdot f=f$. The following lemma can be proved along the same way as that
for Lemma \ref{L:P-orbits2}.

\begin{lemma}\label{Ls:P-orbits4}
Each $P$-orbit in $\mathcal{O}_{f}-P\cdot f$ contains some
 $\bar{n}_{\vec{b}}\cdot f$ for a unique tuple $\vec{b}\in B'_{l}$.
\end{lemma}

With this, we take a similar study as for regular non-elliptic semisimple orbits.
All results can be extended.
In particular, Propositions \ref{Ps:pf1} and \ref{Ps:pf2} extend without change of words.
However, the range of zeros of $h_{\vec{b}}(x)$ becomes a bit different:
 due to $a_{n}=0$,
 we see that $x^{n_l}$ divides $h_{\vec{b}}(x)$.
The analogue of Corollary \ref{Cs:characteristic3} is as follows.

\begin{corollary}\label{Cs:characteristic3-an0}
The map $\vec{b}\mapsto\vec{x}$ gives a bijection from $B'_{l}$ to
\[[a_{2},a_{1}]\times\cdots\times [a_{i_{l-1}},a_{i_{l-1}-1}]\times\{0\}^{n_l}.\]
\end{corollary}

\medskip

For the most degenerate case where $a_{1}=\dots=a_{n-1}=a_{n}=0$, one has $f=0$. Then, the image of
the moment map is equal to $\{0\}$.


\subsection{Non-semisimple coadjoint orbits}\label{SS:non-semisimple}

Non-semisimple coadjoint orbits of $\Spin(2n,1)$ can be thought of as limits of elliptic orbits.
For $a_1\geq \cdots\geq a_{n-1}\geq 0$, define $s_{\vec{a}}\in \mathfrak{g}$ by
\[s_{\vec{a}}=\diag\{a_1 H',\dots,a_{n-1}H' , U\},\ \text{ where }
U=\begin{pmatrix} 0 & 1 & 1 \\ -1 & 0 & 0 \\ 1 & 0 & 0\end{pmatrix}\]
and put $f=\iota(s_{\vec{a}})$.

We first assume that $a_1>a_2>\cdots >a_{n-1}>0$.
Then the coadjoint orbit $\mathcal{O}_f = G\cdot f$ is regular
 and is the limit of regular
 elliptic coadjoint orbits defined in \S\ref{SS:doubleCoset} when $a_n\to +0$.
The image of the moment map $q\colon \mathcal{O}_f \to \mathfrak{p}^*$
 is obtained along the same line as the arguments in Section~\ref{S:elliptic} for elliptic orbits.
Instead of the set $B$ in \S\ref{SS:doubleCoset}, we let
\begin{equation*}
B=\Bigl\{\vec{b}=(b_1,\dots,b_{n}):b_1,\dots,b_{n-1}\geq 0,\
 \sum_{i=1}^{n-1} b_{i}^{2}=1-2b_{n}\Bigr\}.
\end{equation*}
Here, the condition $b_n\geq 0$ is not imposed.
Put $v'_0=[(0,\dots,0,-1,1)]\in X_{n}$.
Then
\begin{lemma}\label{L:T-orbits2}
The map
\[B\rightarrow(X_{n}-\{v'_{0}\})/G^f ,\quad
\vec{b}\mapsto (\bar{n}_{\vec{b}})^{-1} \cdot v_{0}\]
 is a bijection.
\end{lemma}

To obtain the image $q(\mathcal{O}_f)$, we follow the argument in \S\ref{SS:moment}.
Similarly to Lemma~\ref{L:gf}, we have
$q(\bar{n}_{\vec{b}}\cdot f)=\pr(X_{Y,\beta,0})$,
where
\begin{align*}
Y=\diag\{a_1H',\dots,a_{n-1}H',0\} \text{ and }
\beta=(-a_1b_1,0,\dots,-a_{n-1}b_{n-1},0,1).
\end{align*}
Then Lemma~\ref{L:Zb} holds if the matrix $Z$ there is replaced by
\begin{equation*}
\begin{pmatrix} a_{1}&\ldots&0&\frac{-a_{1}b_{1}}{|\beta|}\\
 \vdots&\ddots&\vdots& \vdots\\
 0&\ldots&a_{n-1}&\frac{-a_{n-1}b_{n-1}}{|\beta|}\\
 0&\ldots&0 &\frac{1}{|\beta|}\\
\end{pmatrix}-\frac{\gamma_{1}^{t}\gamma_{2}}{|\beta|^{2}},
\end{equation*}
where $\gamma_{1}=(a_{1}b_{1},\dots,a_{n-1}b_{n-1},-1)$,
$\gamma_{2}=(a_{1}^{2}b_{1}, \dots,a_{n-1}^{2}b_{n-1},0)$. Note that this matrix is the
limit of the matrix $Z$ in Lemma~\ref{L:Zb} when $b_{n}=a_{n}^{-1}$ and $a_{n}\to 0$.
As in Lemma~\ref{L:Zb2},
 the Pfaffian is
 $\Pf(Z_{\vec{b}})=\frac{1}{|\beta|}a_1\cdots a_{n-1}>0$.
As in Proposition~\ref{P:hb3},
\begin{equation*}
h_{\vec{b}}(x)
 = \sum_{1\leq i\leq n-1}\frac{a_{i}^{2}b_{i}^{2}}{|\beta|^{2}}\,\,
 x\!\!\!\! \prod_{1\leq j\leq n-1,\, j\neq i}(x-a_{j}^{2})
 + \frac{1}{|\beta|^{2}}\prod_{1\leq j\leq n-1}(x-a_{j}^{2}).
\end{equation*}

Analogously to Proposition~\ref{Ps:pf1}, we have
\begin{proposition}\label{Ps:pf4}
The image of the moment map $q(\mathcal{O}_f)$
 consists of one depth zero orbit $Pg'_{\infty}\cdot f$,
 and all depth one $P$-coadjoint orbits with singular values
 $(x_1,\dots, x_{n-1})$ such that
\[a_1\geq x_1\geq a_2\geq x_2\geq \cdots\geq
 a_{n-1}\geq x_{n-1}>0\]
 and with the positive Pfaffian.
Moreover, $q$ maps different $P$-orbits in $\mathcal{O}_{f}$
 to different $P$-orbits in $\mathfrak{p}^{\ast}$.
\end{proposition}

Notice that $x_{n-1}$ is strictly positive.
As an analogue of Lemma~\ref{L:proper1},
 we have
\begin{lemma}
For any compact set
 $\Omega\subset (\mathfrak{p}^{\ast}-\mathfrak{l}^{\ast})\cap q(\mathcal{O}_f)$,
 the inverse image $q^{-1}(\Omega)$ is compact.
\end{lemma}
Then Proposition~\ref{P:pf2} holds without change of words.

\medskip

Proposition~\ref{Ps:pf4} extends
 to the case where some of $a_1,\dots,a_{n-1}$ coincide and $a_{n-1}\neq 0$.
This case can be thought of as a limit of singular elliptic orbits treated
 in \S\ref{SS:singular-elliptic}.

\medskip

When $a_{n-1}=0$, the non-semisimple orbit $\mathcal{O}_f$ is
  a limit of singular elliptic orbits treated
 at the end of \S\ref{SS:singular-nonelliptic}.
Similarly to Corollary~\ref{Cs:characteristic3-an0},
 the range of zeros of $h_{\vec{b}}(x)$ becomes
\[
[a_{2},a_{1}]\times \cdots [a_j,a_{j-1}]\times (0,a_{j}] \times \{0\}^{n-j-1},
\]
where $a_j>a_{j+1}=0$.

\medskip

By replacing $U$ with $-U$, namely, taking $\diag\{a_1 H',\dots,a_{n-1}H' , -U\}$, we have other non-semisimple
orbits. This case is similar to the above. The only difference is that the Pfaffian in Proposition~\ref{Ps:pf4}
becomes negative.

\subsection{Moment map for $\Spin(2n-1,1)$}

Suppose that $m$ is even and $G=\Spin(2n-1,1)$.
Let $a_1\geq\cdots\geq a_{n-2}\geq|a_{n-1}|\geq 0$ and $a_{n}\geq 0$.
Write
\[t'_{\vec{a}}=
\begin{pmatrix}
 0&a_1&&&&&\\
 -a_{1}&0&&&&&\\
 &&\ddots&&&&\\
 &&&0&a_{n-1}&&\\
 &&&-a_{n-1}&0&&\\
 &&&&&0&a_{n}\\
 &&&&&a_{n}&0\\
\end{pmatrix}.\]
Put $f = \iota(t'_{\vec{a}})$.
Then $\mathcal{O}_{f}=G\cdot f$ is a semisimple coadjoint orbit in $\mathfrak{g}^{\ast}$.
Moreover, all semisimple coadjoint orbits in $\mathfrak{g}^{\ast}$ are of this form.
Let $q\colon \mathcal{O}_{f}\rightarrow\mathfrak{p}^{\ast}$ be the moment map.
Let $g_{\infty}=1$ and
 let $g'_{\infty}$ be a pre-image in $G$ of $\diag\{I_{2n-3},-1,-1,1\}\in G_{1}$.
The following proposition summarize the result concerning the moment map $q$.
It is analogous to
Propositions~\ref{Ps:pf1} and \ref{Ps:pf2},
 and can be proved along a similar line as them.

\begin{proposition}\label{P:Dn-qf}
The image of the moment map $q(\mathcal{O}_f)$ consists of depth zero orbit(s)
 $P\cdot q(f)$ and $P\cdot q(g'_{\infty}f)$
 (these are the same orbit if and only if $a_{n-1}=a_{n}=0$),
 and depth one coadjoint $P$-orbits with singular values
 $(x_1,\dots, x_{n-2})$ such that
\[a_1\geq x_1\geq a_2\geq x_2\geq \cdots\geq
 a_{n-2}\geq x_{n-2}\geq |a_{n-1}|.\]
If $g\in G$ and $g\cdot f \not\in Pf \cup Pg'_{\infty}f$,
 then the reduced space \[q^{-1}(q(g\cdot f))/\Stab_{P}(q(g\cdot f))\] is a singleton.

The moment map $q$ is weakly proper.
It is not proper unless $f=0$.
\end{proposition}

A similar result holds for non-semisimple orbits.

\section{Verification of Duflo's conjecture in the case of $\Spin(m+1,1)$}\label{S:Duflo}

The orbit method associates some unitary representations
 of Lie groups to coadjoint orbits.
With that, algebraic properties of representations are
 reflected by geometric properties of coadjoint orbits.
The celebrated Duflo's conjecture (Conjecture~\ref{C:Duflo}) gives a connection between
 the branching law of unitary representations and the moment map of coadjoint orbits.
Here we verify Conjecture~\ref{C:Duflo} in our setting.

\subsection{Tempered representations}\label{SS:tempered}
We follow Duflo's way of associating coadjoint orbits to tempered representations.

Suppose first that $m$ is odd and $G=\Spin(2n,1)$.
Put
\[
H'=\begin{pmatrix} 0 & 1 \\ -1 & 0\end{pmatrix},\
H=\begin{pmatrix} 0 & 1 \\ 1 & 0\end{pmatrix},
\text{ and }
U= \begin{pmatrix} 0 & 1 & 1 \\ -1 & 0 & 0 \\ 1 & 0 & 0\end{pmatrix}.
\]

Let
\[\gamma= \Bigl(a_{1}+n-\frac{1}{2}, a_{2}+n-\frac{3}{2},
 \dots,a_{n}+\frac{1}{2}\Bigr)\in\Lambda_{0}\]
be a regular integral weight so that
$a_1,\dots,a_n$ are all integers or all half-integers and
 $a_1\geq \dots \geq a_n \geq 0$.
Let $\pi^{+}(\gamma)$ (resp.\ $\pi^{-}(\gamma)$)
 be a discrete series representation of $G$
 with infinitesimal character $\gamma$ and
 the lowest $K$-type $V_{K,\lambda^+}$ (resp.\ $V_{K,\lambda^-}$),
 where
\[\lambda^+=(a_1+1,\cdots,a_{n}+1) \text{ and }
 \lambda^-=(a_1+1,\cdots,a_{n-1}+1,-(a_{n}+1)).\]
In light of Remark~\ref{R:weight-orbit},
 the orbit $\mathcal{O}$ associated to $\pi^{+}(\gamma)$ is
 $G\cdot \iota(t_{-\gamma})$, where
\[
t_{-\gamma}
= - \diag\Bigl\{\Bigl(a_1+n-\frac{1}{2}\Bigr)H',
 \Bigl(a_2+n-\frac{3}{2}\Bigr)H',\dots,
 \Bigl(a_n+\frac{1}{2}\Bigr)H',0 \Bigr\}.
\]
Putting
\[
\vec{a}'=(a'_1,\dots,a'_n)
=\Bigl(a_{1}+n-\frac{1}{2}, a_{2}+n-\frac{3}{2},
 \dots, (-1)^n\Bigl(a_{n}+\frac{1}{2}\Bigr)\Bigr),
\]
$t_{-\gamma}$ is $G$-conjugate to $t_{\vec{a}'}$.
Hence $\mathcal{O}=G\cdot \iota(t_{\vec{a}'})$.
Similarly, the orbit associated to $\pi^-(\gamma)$ is
 the coadjoint $G$-orbit through
\[
\iota\Bigl(
\diag\Bigl\{\Bigl(a_1+n-\frac{1}{2}\Bigr)H',
 \Bigl(a_2+n-\frac{3}{2}\Bigr)H',\dots,
 (-1)^{n-1} \Bigl(a_n+\frac{1}{2}\Bigr)H',0\Bigr\}\Bigr).
\]

When \[ \gamma = \Bigl(a_{1}+n-\frac{1}{2}, a_{2}+n-\frac{3}{2},
 \dots,a_{n-1}+\frac{3}{2},0\Bigr)\in\Lambda_{n}\]
and $\pi^{+}(\gamma)$ is a limit of discrete series,
The corresponding orbits are not semisimple.
It is the coadjoint $G$-orbit through
\[
\iota\Bigl(
\diag\Bigl\{\Bigl(a_1+n-\frac{1}{2}\Bigr)H',\dots,
 \Bigl(a_{n-1}+\frac{3}{2}\Bigr)H',
 (-1)^{n} U \Bigr\}\Bigr).
\]
For $\pi^{-}(\gamma)$, replace $(-1)^{n} U$ by $(-1)^{n-1}U$.

Next, let
\[I(\mu,\nu)
=\Ind_{MA\bar{N}}^{G}
(V_{M,\mu}\otimes e^{\nu-\rho'}\otimes\mathbf{1}_{\bar{N}})\]
 be a unitary principal series representation of $G$.
Write $\mu=(a_1,\dots,a_{n-1})$ and $\nu=\mathbf{i}a_{n}\lambda_0$.
Then the corresponding orbit $\mathcal{O}$ is
 the $G$-coadjoint orbit through
\[
\iota\Bigl(
\diag\Bigl\{\Bigl(a_1+n-\frac{3}{2}\Bigr)H',
 \Bigl(a_2+n-\frac{5}{2}\Bigr)H',\dots,
 (-1)^{n-1} \Bigl(a_{n-1}+\frac{1}{2}\Bigr)H',0,
 a_n H\Bigr\}\Bigr).
\]

For $P$-representations,
 let $V_{M',\mu}$ be an irreducible representation of $M'$
 with highest weight $\mu=(b_1,\dots,b_{n-1})$.
Let
 $I_{P,V_{M',\mu}}=\Ind_{M'N}^{P}(V_{M',\mu}\otimes e^{\mathbf{i}\xi_0})$
 be the unitarily induced representation of $P$.
Then the corresponding orbit is $P\cdot \pr (Z_{Y,\beta})$
 in the notation of \S\ref{SS:P-orbit2} such that the singular values
 of $Y$ are
\[
(x_1,\dots,x_{n-1})=(b_1+n-2,\,b_2+n-3,\dots,b_{n-2}+1,|b_{n-1}|)
\]
 and the sign of the Pfaffian of $Z_{Y,\beta}$ equals
 the sign of $(-1)^{n-1}b_{n-1}$.

\begin{theorem}\label{T:Duflo1}
Let $P$ be a minimal parabolic subgroup $P$ of $G=\Spin(2n,1)$.
Let $\pi$ be a tempered representation of $G$,
 which is associated to a regular coadjoint orbit $\mathcal{O}\subset\mathfrak{g}^{\ast}$.
Write $q\colon \mathcal{O} \rightarrow\mathfrak{p}^{\ast}$ for the moment map.
\begin{enumerate}
\item[(1)] The restriction of $\bar{\pi}$ to $P$ decomposes
 into a finite direct sum of
 irreducible unitarily induced representations of $P$ from $M'N$,
 and this decomposition is multiplicity-free.
\item[(2)] Let $\tau$ be an irreducible unitarily induced representation of $P$ which is associated to a coadjoint orbit $\mathcal{O}'\subset\mathfrak{p}^{\ast}$.
Assume that $Z(G)$, the center of $G$, acts by the same scalar
 on $\pi$ and on $\tau$.
Then for $\tau$ to appear in $\pi|_{P}$
 it is necessary and sufficient that
 $\mathcal{O'}\subset q(\mathcal{O})$.
\item[(3)] The moment map $q\colon \mathcal{O}\rightarrow\mathfrak{p}^{\ast}$ is weakly proper,
 but not proper.
\item[(4)] The reduced space $q^{-1}(\mathcal{O'})/P$ is a singleton.
\end{enumerate}
\end{theorem}

\begin{proof}
Statement (1) follows from Theorem \ref{T:branching-ds} for (limit of) discrete series
 and Theorem \ref{T:branching-ps} or \ref{T:branching-principal}
 for unitary principal series.
Statements (3) and (4) follow from
 Propositions \ref{P:pf2}, \ref{Ps:pf2}, \ref{Ps:pf4}.
It remains to show Statement (2), that is,
 to compare the restriction of tempered representations and
 the image of moment map of corresponding coadjoint orbits.

For $a_1\geq a_2\geq \cdots\geq a_{n}\geq -\frac{1}{2}$,
 \[\gamma=
\Bigl(a_{1}+n-\frac{1}{2},a_{2}+n-\frac{3}{2},
 \dots,a_{n}+\frac{1}{2}\Bigr)\in\Lambda_{0}\cup\Lambda_{n},\]
 the restriction of the (limit of) discrete series $\pi^+(\gamma)$ is given
 by Theorem \ref{T:branching-ds}:
\[\bar{\pi}^{+}(\gamma)|_{P}=\bigoplus_{\mu}I_{P,V_{M',\mu}}\]
 where $\mu=(b_1,\dots,b_{n-1})$ runs over tuples such that
\begin{equation}\label{Eq:interlacing}
a_{1}+1\geq b_{1}\geq a_{2}\geq
 \cdots\geq b_{n-2}\geq a_{n-1}+1\geq -b_{n-1}\geq a_{n}+1
\end{equation}
 and $b_i-a_1\in \bbZ$.
On the other hand, the moment map image of the corresponding orbit $\mathcal{O}$
 was studied in \S\ref{SS:moment} and \S\ref{SS:non-semisimple}.
Let $\mathcal{O}'$ be a coadjoint $P$-orbit which corresponds
 to a unitary representation $\tau=I_{P,V_{M',\mu}}$
 with $\mu=(b_1,\dots,b_{n-1})$.
Then the singular values for the $P$-orbit $\mathcal{O}'$ are
\[
(x_1,\dots,x_{n-1})=(b_1+n-1,\dots,b_{n-2}+1,|b_{n-1}|)
\]
 and the sign of the Pfaffian equals $\sgn (-1)^{n-1}b_{n-1}$.
Assume that
 the center $Z(G)$ acts on $\pi$ and
 $\tau$ by the same scalar, which is equivalent to $b_i-a_i\in \bbZ$.
Then by Proposition~\ref{P:pf1},
 $\mathcal{O}'\subset q(\mathcal{O})$ if and only if
\begin{align*}
 a_1+n-\frac{1}{2}\geq x_1 \geq a_2+n-\frac{3}{2} \geq x_2
\geq \cdots \geq a_{n-1}+\frac{3}{2}
 \geq x_{n-1}\geq a_{n}+\frac{1}{2}
\end{align*}
and the sign of Pfaffian equals $(-1)^n$.
Under our assumption $b_1-a_i\in \bbZ$,
 this is equivalent to
\eqref{Eq:interlacing}.
Therefore, Statement (2) for $\pi=\pi^+(\gamma)$ is proved.
The case of $\pi=\pi^-(\gamma)$ is similar.

For unitary principal series representations, use Theorem~\ref{T:branching-ps} and Proposition~\ref{Ps:pf1}.
\end{proof}


Suppose next that $m$ is even and $G=\Spin(2n-1,1)$.
Then the tempered representations of $\Spin(2n-1,1)$ are all unitary principal series.
A result similar to Theorem \ref{T:Duflo1} is implied by
 Theorem \ref{T:branching-ps2} and Proposition \ref{P:Dn-qf}.
The proof is along the same line as above.

\subsection{Non-tempered representations}\label{SS:nontempered}

We may associate non-tempered representations to
 some non-regular coadjoint orbits.
For example, the derived functor module $A_{\mathfrak{q}}(\lambda)$
 is associated to elliptic orbits.

Fix $1\leq j\leq n-1$.
Let $\mathfrak{q}_j$ be a $\theta$-stable parabolic subalgebra of $\mathfrak{g}_{\bbC}$
 whose Levi component has the real form isomorphic to
 $\mathfrak{u}(1)^j\oplus \mathfrak{so}(m-2j+1,1)$.
Then for $A_{\mathfrak{q}_j}(\lambda)$ with $\lambda=(a_1,\dots,a_j,0,\dots,0)$
 in the good range, we associate the singular elliptic coadjoint orbit through
\[
\iota\Bigl(\diag\Bigl\{\Bigl(a_1+\frac{m}{2}\Bigr)H',\dots,
 \Bigl(a_j+\frac{m}{2}-j+1\Bigr)H',0,\dots,0\Bigr\}\Bigr).
\]

The branching law of $\overline{A_{\mathfrak{q}_j}(\lambda)}|_{P}$ was obtained
 in \S\ref{SS:branchinglaw} and \S\ref{SS:branchinglaw2}.
According to the formulas there, we observe that
 $I_{P,V_{M',\mu}}$ occurs in the restriction
 only if $\mu$ is of the form $\mu=(b_1,\cdots,b_{j-1},0,\dots,0)$.
Then for $I_{P,V_{M',\mu}}$ with $\mu$ in this form, we associate the coadjoint $P$-orbit
 $P\cdot \pr(Z_{Y,\beta})$ such that the singular values of $Y$ are
\[\Bigl(b_1+\frac{m-1}{2},\cdots,b_{j-1}+\frac{m+3}{2}-j,0,\dots,0\Bigr).\]
We remark that this correspondence is different from one in \S\ref{SS:tempered}
 even for the same representation of $P$.
This is related to the fact that
 the same representation of a compact group
 can be cohomologically induced from different parabolic subalgebras.

By restricting our consideration to the coadjoint $P$-orbit of
 the above singular type for fixed $j$,
 an analogue of Theorem~\ref{T:Duflo1} for $A_{\mathfrak{q}_j}(\lambda)$
 follows from
 branching laws in \S\ref{SS:branchinglaw}, \S\ref{SS:branchinglaw2},
 and results about orbits, Corollary~\ref{Cs:characteristic3-an0}.

\appendix

\section{Unitary principal series representations}\label{S:principalSeries}

In Section~\ref{S:resP} we obtain branching laws for
 all irreducible unitary representations of $\Spin(m+1,1)$
 when restricted to $P$.
In Appendices~\ref{S:principalSeries} and \ref{S:trivial}
 we give more concrete description of
 the decomposition into $P$-representations
 for particular types of representations
 without using the arguments in \S\ref{SS:CW-Cloux}
 or the result of du Cloux~\cite{duCloux}.
We treat unitary principal series representations in Appendix~\ref{S:principalSeries}
 and representations with trivial infinitesimal character in Appendix~\ref{S:trivial}.

For a finite-dimensional irreducible unitary representation $\sigma$ of $M$ and a unitary character $\nu$ of $A$,
 we have a unitary principal series representation $\bar{I}(\sigma,\nu)$.
For $f\in\bar{I}(\sigma,\nu)$, let $f_{N}=f|_{N}$.
Through the map
\[\bar{I}(\sigma,\nu)\xrightarrow{\sim} L^{2}(N,V_{\sigma},\d n),\quad f\mapsto f_{N},\]
 one identifies $\bar{I}(\sigma,\nu)$ with $L^{2}(N,V_{\sigma},\d n)$.
The action of $P=MAN$ on $\bar{I}(\sigma,\nu)$
 induces its action on $L^{2}(N,V_{\sigma},\d n)$
 given by \eqref{Eq:P-action}.

As in \eqref{Eq:Ff},
 define the inverse Fourier transform of $f_{N}\in L^{2}(N,V_{\sigma},\d n)$
 as a function on $\mathfrak{n}^{\ast}$ by
\begin{equation*}
\widehat{f_{N}}(\xi)=\mathcal{F}(f_{N})(\xi)=
(2\pi)^{-\frac{m}{2}}\int_{\bbR^{m}}e^{\mathbf{i}(\xi,x)}f(n_{x})\d x.
\end{equation*}
By the classical Fourier theory, map
\[f_{N}\mapsto \mathcal{F}(f_{N})=\widehat{f_{N}}\]
 gives an isomorphism of Hilbert spaces
\[\mathcal{F}\colon L^{2}(N,V_{\sigma},\d n)
 \rightarrow L^{2}(\mathfrak{n}^{\ast}-\{0\},V_{\sigma},\d \xi)
 \bigl(\simeq L^{2}(\mathfrak{n}^{\ast},V_{\sigma},\d \xi)\bigr).\]
The $P$-action on $\widehat{f_{N}}$ is given by \eqref{Eq:P-action2}.

Recall $\xi_0=(0,\dots,0,1)\in \mathfrak{n}^*$.
For $h\in L^{2}(\mathfrak{n}^{\ast}-\{0\},V_{\sigma},\d x)$,
 define a function $h_{at,\nu}$ on $P$ by $h_{at,\nu}(p)=(p^{-1}\cdot h)(\xi_0)$
 and then
\begin{align*}
h_{at,\nu}(p)
&=e^{-\mathbf{i}(\xi_0,x)}
 |\Ad^*(m_0a)(\xi_{0})|^{\frac{2\nu(H_0)+m}{2}}
 (\sigma(m_0)^{-1}h(\Ad^*(m_0a)\xi_{0}))
\end{align*}
for $p=m_0an_{x}\in P$ as in \eqref{Eq:anti-trivialization}.

\begin{proposition}\label{P:anti-trivialization}
The image of the map
 $L^{2}(\mathfrak{n}^{\ast}-\{0\},V_{\sigma},\d x)\ni h\mapsto h_{at,\nu}$ is equal to the representation space of the unitarily induced representation
 $\Ind_{M'N}^{MAN}(\sigma|_{M'}\otimes e^{\mathbf{i} \xi_{0}})$.
The map $h\mapsto h_{at,\nu}$ preserves the actions of $P$ and
 it preserves inner products up to scalar.
\end{proposition}

\begin{proof}
For $m_0a\in MA$, write $\xi=\Ad^*(m_0a)\xi_{0}$.
Since both $\d {}_{l}m_0a$ and $|\xi|^{-m}\d\xi$
 are $MA$ invariant measures on $MA/M'=\mathfrak{n}^{\ast}-\{0\}$,
 we have $\d {}_{l}m_0a=c|\xi|^{-m}\d\xi$ for some constant $c>0$.
Then $|h_{at,\nu}(m_0a)|^{2}=|\xi|^{m}|h(\xi)|^{2}$ and
 hence $\| h_{at,\nu}\|^{2}=c\|h\|^{2}$.
Therefore, the map $h\mapsto h_{at,\nu}$
 sends $L^{2}(\mathfrak{n}^{\ast}-\{0\},V_{\sigma},\d x)$
 to $\Ind_{M'N}^{MAN}(\sigma|_{M'}\otimes e^{\mathbf{i} \xi_{0}})$
 and it preserves inner products up to scalar.

The remaining assertions can be seen as in Lemma~\ref{L:anti-trivialization}.
\end{proof}

Proposition~\ref{P:anti-trivialization}
 gives the branching law for unitary principal series
 more directly than that in Section~\ref{S:resP}.

\begin{theorem}\label{T:branching-principal}
For a finite-dimensional unitary representation $\sigma$ of $M$, let
\[\sigma|_{M'}=\bigoplus_{j=1}^s \tau_{j}\]
be the decomposition of $\sigma$ into a direct sum of irreducible unitary representations of $M'$.
Then,
\[\bar{I}(\sigma,\nu)|_{P}=\bigoplus_{j=1}^s I_{P,\tau_{j}}.\]
\end{theorem}

\begin{proof}
By realizing $\bar{I}(\sigma,\nu)$ with the non-compact picture of induced representation, we identify $\bar{I}(\sigma,\nu)$ with $L^{2}(N,V_{\sigma},\d n)$.
From \eqref{Eq:P-action},
 we know the action of $P$ on $L^{2}(N,V_{\sigma},\d n)$.
Taking the inverse Fourier transform
  the action of $P$ on the Fourier transformed picture is given as \eqref{Eq:P-action2}.
Applying the anti-trivialization, by Proposition \ref{P:anti-trivialization}
 we identify $\bar{I}(\sigma,\nu)$
 with $\Ind_{M'N}^{MAN}(\sigma|_{M'}\otimes e^{\mathbf{i}\xi_{0}})$.
Hence the conclusion follows.
\end{proof}

\section{Restriction to $P$ of irreducible representations of $G$ with trivial infinitesimal character
and some complementary series}\label{S:trivial}

In this section we study
 representations of $G$ (or $G_2$) with trivial infinitesimal character
 and some complementary series representations in detail.
In particular we prove Proposition~\ref{P:P-restriction2}, a branching law
 for a discrete series, which was used in \S\ref{SS:branchinglaw}.

\subsection{Principal series with infinitesimal character $\rho$}\label{SS:rep-O(m+1,1)}

Let \[m>1 \text{ and } n:= \Bigl\lfloor \frac{m}{2} \Bigr\rfloor +1\] and use the notation in Section~\ref{S:repP}.


For each $0\leq j\leq n-1$,
 define an representation of $M_{2}=\OO(m)\times\Delta_{2}(\OO(1))$,
 denoted by $(\sigma_j, V_j)$ as
\[
V_j:=\bigwedge^{j}\mathbb{C}^{m}
\]
on which $\OO(m)$ acts naturally and $\Delta_{2}(\OO(1))$
 acts trivially.
The restriction $V_j|_{M_1}$ is irreducible and has the highest weight
\begin{equation*}
\mu_{j}=(\underbrace{1,\dots,1}_{j},\underbrace{0,\dots,0}_{n-j-1}).
\end{equation*}
Put
\begin{align*}
 &I_{j}(\nu):=\Ind_{M_{2}A\bar{N}}^{G_{2}}(V_j\otimes e^{\nu-\rho'}\otimes \mathbf{1}_{\bar{N}}).
\end{align*}

\subsection{Normalized Knapp-Stein intertwining operators}\label{SS:intertwining}


The formal Knapp-Stein intertwining operator \[J'_{j}(\nu):I_{j}(-\nu)\rightarrow I_{j}(\nu)\] is defined by
\begin{equation}\label{Eq:intertwining3}(J'_{j}(\nu)f)(g)=\int_{N}f(gsn)\d n
\end{equation}
 for $f\in I_{j}(-\nu)$ and $g\in G_{2}$, where $s=\diag\{I_{m},-1,1\}$ as in (\ref{Eq:s}).

\begin{proposition}\label{L:intertwining1}
When $\Re\nu(H_0)>0$, one has
\begin{equation}\label{Eq:Jf1}(J'_{j}(\nu)f)(n_{x})=\int_{\mathbb{R}^{m}}
|y|^{-2(\rho'-\nu)(H_{0})}\sigma_j(r_{y})f(n_{x-y})\d y
\end{equation}
 for $x\in\mathbb{R}^{m}$, and the integral on the right hand side converges absolutely.

When $f|_{K_{2}}$ is fixed, the value of $(J'_{j}(\nu)f)(n_{x})$ varies holomorphically with respect to $\nu$ when
$\Re\nu(H_0)>0$. 
\end{proposition}

\begin{proof}
Using Lemma \ref{L:barn-iwasawa}, one has
\begin{align*}
(J'_{j}(\nu)f)(n_{x})
&=\int_{\mathbb{R}^{m}}
 f(n_{x}sn_{y})\d y\\
&=\int_{\mathbb{R}^{m}}\sigma_j(r_{y})^{-1}e^{(-2\log|y|)(\rho'+\nu)(H_{0})}
 f(n_{x+\frac{y}{|y|^{2}}})\d y\\
&=\int_{\mathbb{R}^{m}}\sigma_j(r_{y})^{-1}e^{(2\log|y|)(\rho'+\nu)(H_{0})}
|y|^{-4\rho'(H_0)}f(n_{x-y})\d y\\
&=\int_{\mathbb{R}^{m}}|y|^{-2(\rho'-\nu)(H_{0})}\sigma_j(r_{y})^{-1}f(n_{x-y})\d y.
\end{align*}
Note that $r_y^2=I$ and hence $\sigma_j(r_{y})^{-1}=\sigma_j(r_{y})$.

By Lemma \ref{L:Iwasawa}, for $f\in I_{j}(-\nu)$,
\begin{equation*}
f(n_{x})=(1+|x|^{2})^{-(\rho'+\nu)(H_{0})}f(s_{x}).
\end{equation*}
Let $C=\max_{x\in K_{2}}|f(x)|$. Then,
\[|f(n_{x})|\leq C(1+|x|^{2})^{-(\rho'+\nu)(H_{0})}\] for all $x\in\mathbb{R}^{m}$.
Since \[(-2(\rho'-\nu)H_{0})+(-2(\rho'+\nu)(H_{0}))=-4\rho'(H_0)=-2m<-m,\]
the integral in (\ref{Eq:Jf1}) converges absolutely in the range $|y|\geq |x|+1$.
For a fixed $x$,  $\sigma_j(r_{y})f(n_{x-y})$ is bounded
 in the range $|y|\leq |x|+1$.
By \[-\Re2(\rho'-\nu)(H_{0})=-m+2\Re\nu(H_0)>-m,\]
 the integral in (\ref{Eq:Jf1}) also converges absolutely on $|y|\leq |x|+1$.
By the absolute convergence, the integral on the right hand side of (\ref{Eq:Jf1}) varies holomorphically with respect to $\nu$.
\end{proof}


The Knapp-Stein intertwining operator in the current
 setting was studied in \cite[\S8.3]{Kobayashi-Speh}.
Following \cite[(8.12)]{Kobayashi-Speh}, define
 an $\End(V_j)$-valued distribution on $\mathbb{R}^{m}$ by
\begin{equation}
\label{Eq:Tdistribution}
(T_{j}(\nu))(x)=\frac{1}{\Gamma(\nu(H_0))}|x|^{-2(\rho'-\nu)(H_{0})}\sigma_j(r_{x}).
\end{equation}
Define
\begin{equation*}
(J_{j}(\nu)f)(n_{x})=\int_{\mathbb{R}^{m}}(T_{j}(\nu))(y)f(x-y)\d y.
\end{equation*}
We call $J_{j}(\nu)$ the {\it normalized Knapp-Stein intertwining operator.} The following is
 \cite[Lemma 8.7]{Kobayashi-Speh}, which is implied by Lemma \ref{L:Riesz}.


\begin{lemma}\label{L:continuation}
Both the distribution $T_{j}(\nu)$ and the intertwining operator $J_{j}(\nu)$ admit holomorphic continuations to the
whole $\mathfrak{a}^{\ast}_{\mathbb{C}}$.
\end{lemma}


Assume $\nu(H_{0})\in \bbR$.
For two functions $f,h$ in the image of $J_{j}(\nu)$,
 choose $\tilde{f},\tilde{h}\in I_{j}(-\nu)$
such that $J_{j}(\nu)\tilde{f}=f$ and $J_{j}(\nu)\tilde{h}=h$.
Define
\begin{equation}\label{Eq:Hermitian1}(f,h):=(\tilde{f}|h)=
\int_{\mathbb{R}^{m}}(\tilde{f}(n_{x}),h(n_{x}))\d x,\end{equation}
 where $(\tilde{f}(n_{x}),h(n_{x}))$ is the inner product on
 $V_{M,\mu_{j}}$.
By the following lemma, $(f,h)$ is a well-defined $G$-invariant
 Hermitian form on the image of $J_j(\nu)$.


\begin{lemma}\label{L:Hermitian2}
For any $\nu\in\mathfrak{a}_{\bbC}^*$,
 $\tilde{f}\in I_{j}(-\nu)$ and $\tilde{h}\in I_{j}(-\bar{\nu})$,
\begin{equation*}
\int_{\mathbb{R}^{m}}((J_{j}(\nu)\tilde{f})(n_{x}), \tilde{h}(n_{x}))\d x
=\int_{\mathbb{R}^{m}}(\tilde{f}(n_{x}), (J_j(\bar{\nu})\tilde{h})(n_{x}))\d x.
\end{equation*}
\end{lemma}

\begin{proof}
Let $\phi(g)=(\tilde{f}(g), (J_j(\bar{\nu})\tilde{h})(g))$ ($g\in G_{2}$).
Then it satisfies the relation $\phi(gma\bar{n})=e^{2\rho'(\log a)}\phi(g)$.
By this, $\phi$ gives a left $G_{2}$-invariant density form on $G_{2}/\bar{P}_{2}$.
Hence, \[\int_{\mathbb{R}^{m}}\phi(n_{x})\d x=\int_{K_{2}}\phi(k)\d k\]
 (\cite[Chapter V, \S 6]{Knapp}).
Therefore, the integral in \eqref{Eq:Hermitian1} converges.

Put
 \[(\tilde{f}, \tilde{h})_{\nu}
 = \int_{\bbR^m} \bigl(\tilde{f}(n_{x}), (J_{j}(\bar{\nu}) \tilde{h})(n_x)\bigr) \d x. \]
Fix $\tilde{f}|_{K_{2}}, \tilde{h}|_{K_{2}}$ and vary $\nu$.
By Lemma \ref{L:continuation} the value of $(\tilde{f},\tilde{h})_{\nu}$
 varies holomorphically with respect to $\nu$.
When $\Re\nu(H_0)>0$, by Proposition~\ref{L:intertwining1}
 the intertwining integral converges absolutely. Hence
\begin{align*}
&\quad(\tilde{f},\tilde{h})_{\nu}\\
&=\int_{\mathbb{R}^{m}}\int_{\mathbb{R}^{m}}(\tilde{f}(n_{y}),(T_{j}({\bar{\nu}}))(y-x)\tilde{h}(n_{x}))
 \d x\d y\\
&=\int_{\mathbb{R}^{m}}\int_{\mathbb{R}^{m}}((T_{j}(\nu))(x-y)\tilde{f}(n_{y}),
 \tilde{h}(n_{x}))\d y\d x.
\end{align*}
By holomorphicity, this holds for all $\nu\in\mathfrak{a}^{\ast}_{\mathbb{C}}$.
\end{proof}

\subsection{Fourier transformed picture}\label{SS:Fourier2}

For $f\in I_{j}(\nu)$, let $f_{N}=f|_{N}$. As in \eqref{Eq:Ff}, define the inverse Fourier transform of
$f_{N}$ as a function (or a distribution) on $\mathfrak{n}^{\ast}$ by
\begin{equation*}
\widehat{f_{N}}(\xi)=\mathcal{F}(f_{N})(\xi)
=(2\pi)^{-\frac{m}{2}}\int_{\mathbb{R}^{m}}e^{\mathbf{i}(\xi,x)}f(n_{x})\d x.
\end{equation*}
Since $f_{N}$ is a tempered distribution, $\widehat{f_{N}}$ is a tempered distribution as well. The action of
$P_2$ on the Fourier transformed picture is defined as $p\widehat{f_{N}}=\widehat{(p f)_{N}}$ for $p\in P_2$.
The $P_2$-actions on $I_j(\nu)$ and $\mathcal{F}I_j(\nu)$ are given as \eqref{Eq:P-action} and \eqref{Eq:P-action2},
respectively.

When $\Re\alpha<m$, the function $|x|^{-\alpha}$ on $\mathbb{R}^{m}$ is locally $L^1$
 and is a tempered distribution.
The lemma below follows from \cite[Chapter I, \S3.9]{Gelfand-Shilov}
 and it implies Lemma \ref{L:continuation}.

\begin{lemma}\label{L:Riesz}
The distribution $|x|^{-\alpha}$ on $\mathbb{R}^{m}$, originally defined when $\Re\alpha<m$, admits a unique meromorphic extension to all $\alpha\in\mathbb{C}$ as tempered distributions.
Moreover, the extension has only simple poles,
 which are at $\alpha=m+2k$ with $k\in\mathbb{Z}_{\geq 0}$.

The distribution $\Gamma(\frac{m-\alpha}{2})^{-1} |x|^{-\alpha}$ admits a holomorphic continuation to the whole complex plane such that it is a tempered distribution
 for every $\alpha\in\mathbb{C}$.
\end{lemma}



It follows from \cite[Chapter II, \S 3.3]{Gelfand-Shilov} that
\begin{align}\label{Eq:Riesz1}
\mathcal{F}|x|^{-\alpha}=d_{\alpha}|\xi|^{-(m-\alpha)},
\text{ where }
d_{\alpha}=2^{\frac{m}{2}-\alpha}\frac{\Gamma(\frac{m-\alpha}{2})}{\Gamma(\frac{\alpha}{2})}.
\end{align}
From \eqref{Eq:Mult-Differ} and \eqref{Eq:Riesz1}, one verifies the following two equalities.
For $1\leq k,l \leq m$ with $k\neq l$,
\begin{align}\label{Eq:Riesz2}
&\mathcal{F}(|x|^{-\alpha-2}x_{k}x_{l})=
-d_{\alpha+2}(m-\alpha-2)
(m-\alpha)\xi_{k}\xi_{l}|\xi|^{-(m-\alpha)-2}, \\ \nonumber
&\mathcal{F}(|x|^{-\alpha-2}x_{k}^2)
=d_{\alpha+2}(m-\alpha-2)
(|\xi|^{2}-(m-\alpha)\xi_{k}^{2})|\xi|^{-(m-\alpha)-2}.
\end{align}


In the following lemma \ref{L:Riesz3}, we give a formula for the Fourier transformed counter-part of the
intertwining kernel $T_{j}(\nu)$. The formula \eqref{Eq:fourier-T} is crucial for us. It enables us to
show algebraic and analytic properties of Fourier transformed picture of the intertwining image, and to
construct $L^2$-models for irreducible unitary representations with infinitesimal character $\rho$ and
for some complementary series $I(\mu_{j},\nu)$ of $G_{2}=\OO(m+1,1)$ when $0<\nu(H_0)<\frac{m}{2}-j$.

\begin{lemma}\label{L:Riesz3}
Let $0\leq j\leq n-1$. Then
\begin{equation}
\label{Eq:fourier-T}
\mathcal{F}T_{j}(\nu)
=\frac{2^{(2\nu-\rho')(H_0)}|\xi|^{-2\nu(H_0)}}{\Gamma(1+(\rho'-\nu)(H_0))}
\big(\rho'(H_0)-j-\nu(H_0)\sigma_j(r_{\xi})\big).
\end{equation}
\end{lemma}

\begin{proof}
Choose an orthonormal basis $\{v_1,\dots,v_{m}\}$ of $\mathbb{C}^{m}$.
The exterior product $\bigwedge^{j}\mathbb{C}^{m}$ has a basis
$v_{I}=v_{i_{1}}\wedge\cdots\wedge v_{i_{j}}$,
 where $I=\{i_{1},\dots,i_{j}\}$ with $1\leq i_{1}<\cdots<i_{j}\leq m$.
By this,
\begin{align*}
&\quad\sigma_j(r_{x})v_{I} \\
 &=r_{x}(v_{i_1})\wedge \cdots \wedge r_{x}(v_{i_{j}})\\
 &= (v_{i_{1}}-2|x|^{-2}(v_{i_{1}},x)x) \wedge\cdots
  \wedge (v_{i_{j}}-2|x|^{-2}(v_{i_{j}},x)x)\\
 &=v_{I}+\sum_{k=1}^{j}2(-1)^{k}|x|^{-2} (v_{i_{k}},x)
 x \wedge v_{i_{1}}\wedge\cdots\wedge v_{i_{k-1}}\wedge v_{i_{k+1}}\cdots\wedge v_{i_{j}}\\
&=v_{I}+
 \sum_{k=1}^{j}\sum_{l=1}^{m}2(-1)^{k}|x|^{-2} (v_{i_{k}},x)(v_{l},x)
 v_{l} \wedge v_{i_{1}} \wedge \cdots \wedge v_{i_{k-1}} \wedge v_{i_{k+1}} \cdots \wedge
 v_{i_{j}}\\
&=\sum_{J}c_{I,J}(x)v_{J},
\end{align*}
where $J$ runs over $J\subset \{1,\dots,m\}$ such that $|J|=j$; and
\[c_{I,J}(x)\!=\!
\begin{cases}
2(-1)^{k+a(I,k,l)}|x|^{-2} (v_{i_{k}},x)(v_{l},x)
 \ \text{if $J=(I-\{i_{k}\})\sqcup\{l\}\neq I$},\\
|x|^{-2}\Bigl(\sum_{\substack{l'\neq i_{k} (1\leq k\leq j)}}(v_{l'},x)^{2}-\sum_{k=1}^{j}(v_{i_{k}},x)^{2}\Bigr)
 \ \text{if $I=J$},\\
0\quad\text{if $|I\cap J|<j-1$}.
\end{cases}\]
Here, $a(I,k,l)$ is the number of indices $k'\in\{1,\dots,k-1,k+1,\dots,j\}$ such that $i_{k'}<l$.

Put $\alpha=2(\rho'-\nu)(H_0)$. Using two formulas \eqref{Eq:Riesz2} and the above expression of $\sigma_j(r_{x})v_{I}$,
one finds that the inverse  Fourier transform of
\[(T_{j}(\nu))(x)=\frac{|x|^{-2(\rho'-\nu)(H_{0})}}{\Gamma(\nu(H_0))}\sigma_j(r_{x})\]
is equal to
\[d_{\alpha+2}\frac{(m-\alpha-2)|\xi|^{-(m-\alpha)}}{\Gamma(\nu(H_0))}
 (m-2j-(m-\alpha)\sigma_j(r_{\xi})).\]
Note that
\begin{align*}
m-\alpha=2\nu(H_0),\quad
\frac{d_{\alpha+2}}{d_{\alpha}}
=\frac{1}{\alpha(m-\alpha-2)}, \quad
\frac{d_{\alpha}}{\Gamma(\nu(H_0))}=\frac{2^{\frac{m}{2}-\alpha}}{\Gamma(\frac{\alpha}{2})}.
\end{align*}
Multiplying these together, we find that $\mathcal{F}T_{j}(\nu)$ is equal to
\[\frac{2^{(2\nu-\rho')(H_0)}|\xi|^{-2\nu(H_0)}}{\Gamma(1+(\rho'-\nu)(H_0))}
\bigl(\rho'(H_0)-j-\nu(H_0)\sigma_j(r_{\xi})\bigr).
\qedhere\]
\end{proof}

After writing a draft of this paper,
 the authors noticed that Lemma~\ref{L:Riesz3} was proved in a more general form
 in \cite[Theorem 3.2 and Remark 4.9]{Fischmann-Orsted}.
The Knapp-Stein intertwining operators and $L^2$ inner products for complementary series
 in the Fourier transformed picture
 were previously studied in \cite{Speh-Venkataramana}.
The authors thank Professor {\O}rsted for kindly showing references.

We need some general facts about the convolution and the Fourier transform.

\begin{fact}\label{L:Fourier-convolution1}
\begin{enumerate}
\item Let $1\leq p,q,r\leq\infty$ be such that $\frac{1}{r}=\frac{1}{p}+\frac{1}{q}$.
For any complex valued functions
 $u_{1}\in L^{p}(\mathbb{R}^{m})$ and $u_{2}\in L^{q}(\mathbb{R}^{m})$, one has
$u_{1}u_{2}\in L^{r}(\mathbb{R}^{m})$ and
\begin{equation*}
\| u_{1}u_{2}\|_{r}
\leq \|u_{1}\|_{p}\|u_{2}\|_{q}.
\end{equation*}
\item Let $1\leq p,q,r\leq\infty$ be such that $\frac{1}{r}=\frac{1}{p}+\frac{1}{q}-1$.
For any complex valued functions $u_{1}\in L^{p}(\mathbb{R}^{m})$
 and $u_{2}\in L^{q}(\mathbb{R}^{m})$, one has
 $u_{1}\ast u_{2}\in L^{r}(\mathbb{R}^{m})$ and
\begin{equation*}
\|u_{1}\ast u_{2}\|_{r}\leq\|u_{1}\|_{p}
\|u_{2}\|_{q}.
\end{equation*}
\item Let $1\leq p\leq 2$ and $\frac{1}{q}=1-\frac{1}{p}$. For any complex valued function
$u\in L^{p}(\mathbb{R}^{m})$, one has $\widehat{u}\in L^{q}(\mathbb{R}^{m})$ and
\begin{equation*}
\|\widehat{u}\|_{q}\leq(2\pi)^{\frac{m}{2}-\frac{m}{p}}\|u\|_{p}.
\end{equation*}
\end{enumerate}
\end{fact}

\begin{proof}
(1) is the classical H\"older's inequality.
(2) is \cite[Corollary 4.5.2]{Hormander} and (3) is \cite[Theorem 7.1.13]{Hormander}.
\end{proof}

The following is a version of the convolution theorem.

\begin{lemma}\label{L:Fourier-convolution2}
Let $p_{1},p_{2}\geq 1$ be such that
 $\frac{1}{p_{1}}+\frac{1}{p_{2}}\geq\frac{3}{2}$.
For any complex valued functions $u_{1}\in L^{p_{1}}(\mathbb{R}^{m})$
 and $u_{2}\in L^{p_{2}}(\mathbb{R}^{m})$, one has
\begin{equation}\label{Eq:Tconvolution}
\widehat{u_{1}\ast u_{2}}=(2\pi)^{\frac{m}{2}}\widehat{u_{1}}\widehat{u_{2}}.
\end{equation}
\end{lemma}

\begin{proof}
As $\frac{1}{p_{1}}+\frac{1}{p_{2}}\geq\frac{3}{2}$ and $p_{1},p_{2}\geq 1$, we have $1\leq p_{1},p_{2}\leq 2$.
Let $p,q_1,q_2$ be given by \[\frac{1}{p}=\frac{1}{p_{1}}+\frac{1}{p_2}-1,\quad\frac{1}{q_{1}}=1-\frac{1}{p_{1}},
\quad\frac{1}{q_2}=1-\frac{1}{p_{2}}.\] Then $1\leq p\leq 2$ and $q_{1},q_{2}\geq 2$. Let $q\geq 2$ be given by
\[\frac{1}{q}=1-\frac{1}{p}=\frac{1}{q_{1}}+\frac{1}{q_2}.\] By Fact~\ref{L:Fourier-convolution1},
\[u_{1}\ast u_{2}\in L^{p},\quad
\widehat{u_{1}}\in L^{q_{1}},\quad \widehat{u_{2}}\in L^{q_{2}}.\]
Again by Fact~\ref{L:Fourier-convolution1},
\[\widehat{u_{1}\ast u_{2}}\in L^{q}
\text{ and }
\widehat{u_{1}}\widehat{u_{2}}\in L^{q}.\]
Put
\[\phi(u_1,u_2)=\widehat{u_{1}\ast u_{2}}-(2\pi)^{\frac{m}{2}}\widehat{u_{1}}\widehat{u_{2}}.\]
Then $\phi$ gives a map
$L^{p_{1}}\times L^{p_{2}}\rightarrow L^{q}$, which is bilinear and bounded on both variables. By the classical
convolution theorem, $\phi(u_1,u_{2})=0$ whenever $u_1,u_{2}$ are both Schwartz functions. Since the space of
Schwartz functions are dense in both $L^{p_{1}}$ and $L^{p_{2}}$, we get $\phi(u_1,u_2)=0$ for any
$u_{1}\in L^{p_{1}}$ and $u_{2}\in L^{p_{2}}$. Thus, (\ref{Eq:Tconvolution}) follows.
\end{proof}

Recall some facts about the (single variable) $K$-Bessel function. See e.g.\ \cite{Watson} for more details.
The $K$-Bessel function $K_{\alpha}(x)$ is a solution to the second order linear ordinary differential
equation \begin{equation*}
x^{2}y''+xy'-(x^{2}+\alpha^{2})y=0
\end{equation*}
that has the asymptotic behavior
\begin{equation}\label{Eq:Bessel2}K_{\alpha}(x)
=\sqrt{\frac{\pi}{2x}}e^{-x}\Bigl(1+O\Bigl(\frac{1}{x}\Bigr)\Bigr)
\quad (x\to +\infty).
\end{equation}
This uniquely determines the holomorphic function $K_{\alpha}(x)$ on $\bbC - \bbR_{\leq 0}$.
Note that $K_{\alpha}(x)=K_{-\alpha}(x)$.
When $x\to +0$, $K_{\alpha}(x)$ behaves as
\begin{equation*}
K_{\alpha}(x)=\begin{cases}
\frac{\Gamma(\alpha)}{2}(\frac{x}{2})^{-\alpha}(1+o(1))
&\textrm{ if $\Re\alpha>0$},\\
(\frac{\Gamma(\alpha)}{2}(\frac{x}{2})^{-\alpha}
+\frac{\Gamma(-\alpha)}{2}(\frac{x}{2})^{\alpha})(1+o(1))
&\textrm{ if $\alpha\in \mathbf{i}\mathbb{R} -\{0\}$},\\
-\log(\frac{x}{2})(1+o(1))&\textrm{ if $\alpha=0$}.
\end{cases}
\end{equation*}
Define
\begin{equation*}
\tilde{K}_{\alpha}(x):=\Bigl(\frac{|x|}{2}\Bigr)^{\alpha}K_{\alpha}(|x|)
\quad \text{ for $x\in \bbR-\{0\}$}.
\end{equation*}
Then as $x\to +0$,
\begin{equation}\label{Eq:Bessel4}
\tilde{K}_{\alpha}(x)
=\begin{cases}
 \frac{\Gamma(\alpha)}{2}(1+o(1))
  & \textrm{ if $\Re\alpha>0$},\\
 (\frac{\Gamma(\alpha)}{2}
 + \frac{\Gamma(-\alpha)}{2}(\frac{x}{2})^{2\alpha})(1+o(1))
  & \textrm{ if $\alpha\in \mathbf{i}\mathbb{R} -\{0\}$},\\
 -\log(\frac{x}{2})(1+o(1))
  & \textrm{ if $\alpha=0$},\\
 \frac{\Gamma(-\alpha)}{2}(\frac{x}{2})^{2\alpha}(1+o(1))
  & \textrm{ if $\Re\alpha<0$}.
\end{cases}
\end{equation}
The function $y=\tilde{K}_{\alpha}(x)$ solves the second order linear ODE
\begin{equation*}
\label{Eq:Bessel5}
xy''+(1-2\alpha)y'-xy=0.
\end{equation*}
It follows that (\cite[III.71(5)]{Watson})
\begin{equation}\label{Eq:KBessel-recursive}
\tilde{K}_{\alpha+1}'(x)=-\frac{x}{2}\tilde{K}_{\alpha}(x).
\end{equation}
as both sides satisfy the same second order linear ODE and have the same asymptotic behavior when $x\to +\infty$.
Then \[\tilde{K}''_{\alpha+1}
=\tilde{K}_{\alpha+1}+\frac{2\alpha+1}{x}\tilde{K}'_{\alpha+1}
=\tilde{K}_{\alpha+1}-\frac{2\alpha+1}{2}\tilde{K}_{\alpha}\]
and hence
\begin{equation*}
\tilde{K}_{\alpha}=\frac{1}{\alpha+\frac{1}{2}}
(\tilde{K}_{\alpha+1}- \tilde{K}''_{\alpha+1}).
\end{equation*}
By this one shows that
\[\frac{\tilde{K}_{\alpha}(x)}{\Gamma(\alpha+\frac{1}{2})}\]
admits a holomorphic continuation to the whole complex plane and gives tempered distributions on $\bbR$
for any $\alpha\in \bbC$.

By \cite[Chapter II, \S2.5]{Gelfand-Shilov}, for each $\lambda\in \bbC$,
\begin{equation*}
\mathcal{F}(1+x^2)^{\lambda}
=\frac{\sqrt{2}}{\Gamma(-\lambda)}\tilde{K}_{-\lambda-\frac{1}{2}}(\xi)
\end{equation*}
as tempered distributions on $\bbR$.
%

On $\mathbb{R}^{m}$, define the modified $K$-Bessel function by \[\tilde{K}_{\alpha}(x)=
\Bigl(\frac{|x|}{2}\Bigr)^{\alpha} K_{\alpha}(|x|).\] Write
\[E=\sum_{1\leq i\leq m}x_{i}\partial_{x_{i}}\quad\textrm{ and }\quad\Delta=\sum_{1\leq i\leq m}(\partial_{x_{i}})^{2}.\]
By \eqref{Eq:KBessel-recursive} we have \begin{equation}\label{Eq:KBessel-recursive2}
(\partial_{x_{j}}\tilde{K}_{\alpha+1})(x)=-\frac{x_{j}}{2}\tilde{K}_{\alpha}(x).
\end{equation} Then, one shows that \[|x|^{2}\Delta \tilde{K}_{\alpha}+(2-m-2\alpha)E\tilde{K}_{\alpha}-
|x|^{2}\tilde{K}_{\alpha}=0\] and \[E\tilde{K}_{\alpha+1}=-\frac{|x|^{2}}{2}\tilde{K}_{\alpha}.\] By these,
we get \[\tilde{K}_{\alpha}=\frac{1}{\alpha+\frac{m}{2}}(-\Delta\tilde{K}_{\alpha+1}+\tilde{K}_{\alpha+1}).\]
Then, one shows that \[\frac{\tilde{K}_{\alpha}(x)}{\Gamma(\alpha+\frac{m}{2})}\] admits a holomorphic continuation
to the whole complex plane for $\alpha$ such that it is a tempered distribution on $\bbR^m$ for each $\alpha$.

\begin{lemma}\label{L:Fourier-KBessel3}
For all $\lambda\in \bbC$, \[\mathcal{F}(1+|x|^2)^{\lambda}
=\frac{2^{1-\frac{m}{2}}}{\Gamma(-\lambda)}\tilde{K}_{-\lambda-\frac{m}{2}}(|\xi|)\]
as tempered distributions on $\bbR^m$.
\end{lemma}

\begin{proof}
By the holomorphicity of both sides with respect to $\lambda$,
 we may assume that $\Re\lambda<-\frac{m}{2}$.
We calculate the Fourier transform of $(1+|x|^{2})^{\lambda}$ on $\mathbb{R}^{m}$
by reducing it to the case of $m=1$.
Let $m\geq 2$.
Write $\Omega_{m}$ for the volume of the $(m-1)$-dimensional sphere. It is well-known that
\[\Omega_{m}=\frac{2\pi^{\frac{m}{2}}}{\Gamma(\frac{m}{2})}.\]
Letting $x=(x_{1},\sqrt{1+x_{1}^{2}} z)$ with $z\in\mathbb{R}^{m-1}$, one gets
\begin{align*}
&\int_{\mathbb{R}^{m}}(1+|x|^2)^{\lambda}e^{\mathbf{i}(\xi,x)}\d x\\
&=\int_{\mathbb{R}^{m}}(1+|x|^2)^{\lambda}e^{\mathbf{i}x_{1}|\xi|}\d x\\
&=\int_{\mathbb{R}^{m}}(1+x_{1}^{2})^{\lambda+\frac{m-1}{2}}
 e^{\mathbf{i}x_{1}\xi}(1+|z|^2)^{\lambda}\d x_1\d z\\
&=\sqrt{2\pi}\mathcal{F}(1+x_{1}^{2})^{\lambda+\frac{m-1}{2}}
\times \Omega_{m-1}
 \int_{0}^{\infty}(1+r^{2})^{\lambda}r^{m-2}\d r\\
&=\frac{2\sqrt{\pi}}{\Gamma(-\lambda-\frac{m-1}{2})}
 \tilde{K}_{-\lambda-\frac{m}{2}}(|\xi|)
 \times \frac{2\pi^{\frac{m-1}{2}}}
{\Gamma(\frac{m-1}{2})}\frac{\Gamma(\frac{m-1}{2})\Gamma(-\lambda-\frac{m-1}{2})}{2\Gamma(-\lambda)}\\
&=(2\pi)^{\frac{m}{2}}
\frac{2^{1-\frac{m}{2}}}{\Gamma(-\lambda)}\tilde{K}_{-\lambda-\frac{m}{2}}(|\xi|)
\end{align*}
For the first equation, we used a rotation on coordinates
 to assume $\xi=(|\xi|,0,\dots,0)$.
\end{proof}

Now we show properties of $K_{2}$-finite functions in $I(\sigma,\nu)$ (particularly $I_{j}(\nu)$) and their Fourier
transforms.

\begin{lemma}\label{L:f-polynomial}
For any tuple $\alpha=(\alpha_1,\dots,\alpha_{m})$ with $\alpha_{i}\in\mathbb{Z}_{\geq 0}$ and any $K_{2}$-finite
function $f\in I(\sigma,\nu)_{K_{2}}$, each entry of \[(1+|x|^{2})^{(\rho'-\nu)(H_{0})}\partial^{\alpha}f(n_{x})\]
is a polynomial of the functions $(1+|x|^{2})^{-1}$, $x_{i}(1+|x|^{2})^{-1}$, $x_{i}x_{j}(1+|x|^{2})^{-1}$
($1\leq i,j\leq m$). In particular, \[|\partial^{\alpha}f(n_{x})|\leq C(1+|x|^2)^{(\Re\nu-\rho')(H_0)}\]
for a constant $C$ depending on $f$ and $\alpha$.
\end{lemma}

\begin{proof}
By Lemma~\ref{L:Iwasawa}, we have \[f(n_{x})=(1+|x|^{2})^{-(\rho'-\nu)(H_{0})}f(s_{x}),\] where
\[s_{x}=\begin{pmatrix}I_{m}-\frac{2x^{t}x}{1+|x|^2}&-\frac{2x^{t}}{1+|x|^{2}}&0\\
 \frac{2x}{1+|x|^2}&\frac{1-|x|^{2}}{1+|x|^{2}}&0\\0&0&1\\
\end{pmatrix}.\]
Since $f$ is a $K_{2}$-finite function, $f(s_{x})$ is a polynomial of its entries. Hence the claim holds when
$\alpha=0$ by Lemma \ref{L:Fourier-KBessel3}. By taking derivatives, the lemma is proved.
\end{proof}

\begin{lemma}\label{L:f-Fourier}
\begin{enumerate}
\item[(1)]For any $\nu\in\mathfrak{a}^{\ast}_{\mathbb{C}}$ and any $K_{2}$-finite function $f\in I(\sigma,\nu)_{K_{2}}$,
$\widehat{f_{N}}(\xi)$ is a smooth function on $\mathfrak{n}^{\ast}-\{0\}$ and it decays fast near infinity in
the sense that $\widehat{f_{N}}(\xi)e^{|\xi|}$ has a polynomial bound as $\xi\rightarrow\infty$.
\item[(2)]If $\Re\nu(H_0)<0$, then $\widehat{f_{N}}(\xi)$ is continuous near $\xi=0$; if $\Re\nu(H_0)=0$, then
$\widehat{f_{N}}(\xi)$ is bounded by a multiple of $\log|\xi|$ near $\xi=0$; if $0<\Re\nu(H_0)<\rho'(H_0)$,
then $\widehat{f_{N}}(\xi)$ is bounded by a multiple of $|\xi|^{-2\Re\nu(H_0)}$ near $\xi=0$.
\item[(3)]If $\Re\nu(H_0)<\rho'(H_0)$, then $\widehat{f_{N}}$ belongs to $L^1(\mathfrak{n}^*)$.
\end{enumerate}
\end{lemma}

\begin{proof}
(1) By Lemma \ref{L:f-polynomial}, $f(n_{x})$ is of the form \[f(n_{x})=(1+|x|^{2})^{-(\rho'-\nu)(H_{0})}h(x)\]
where $h(x)$ is a polynomial of $(1+|x|^{2})^{-1}$, $x_{i}(1+|x|^{2})^{-1}$, $x_{i}x_{j}(1+|x|^{2})^{-1}$
($1\leq i,j\leq m$). By Lemma \ref{L:Fourier-KBessel3}, the inverse Fourier transform of $(1+|x|^{2})^{\lambda}$
is \[\frac{2^{1-\frac{m}{2}}}{\Gamma(-\lambda)}\tilde{K}_{-\lambda-\frac{m}{2}}(|\xi|).\] Then
$\widehat{f_{N}}(\xi)$ is a linear combination of terms of the form
\[\frac{\partial^{\alpha}\tilde{K}_{-\nu(H_0)+k}(|\xi|)}{\Gamma((\rho'-\nu)(H_0)+k)}\] with $k\geq 0$ and
$|\alpha|\leq 2k$. By the recursive relation \eqref{Eq:KBessel-recursive2} and the smoothness and asymptotic
behavior \eqref{Eq:Bessel2} of the $K$-Bessel function, one sees that $\widehat{f_{N}}(\xi)$ is a smooth function
over $\mathfrak{n}^{\ast}-\{0\}$ and it decays fast near infinity in the sense that $\widehat{f_{N}}(\xi)e^{|\xi|}$
has a polynomial bound as $\xi\rightarrow\infty$.

\noindent (2) By the recursive relation \eqref{Eq:KBessel-recursive2}, each $\partial^{\alpha}\tilde{K}_{-\nu(H_0)+k}(|\xi|)$
($|\alpha|\leq 2k$) is a finite linear combination of terms of the form $\xi^{\beta}\tilde{K}_{-\nu(H_0)+k-t}(|\xi|)$
where $t\geq|\beta|\geq 0$ and $2t-|\beta|=|\alpha|\leq 2k$. When $\Re\nu(H_0)<\rho'(H_0)$, we have
$\Re(-\nu(H_0)+k)>-\frac{m}{2}$.
Moreover, by \eqref{Eq:Bessel4} one shows that the singularity at 0 of
$\xi^{\beta}\tilde{K}_{-\nu(H_0)+k-t}(|\xi|)$ is not worse than that described in (2).

\noindent (3) By (1) and (2), $\widehat{f_{N}}|_{\mathfrak{n}^*-\{0\}}$ is in $L^1$ when $\Re\nu(H_0)<\rho'(H_0)$. Let $h:=
\widehat{f_{N}}|_{\mathfrak{n}^*-\{0\}}$. To prove (3), we only need to show that $\widehat{f_{N}}=h$ as a
distribution. Since the support of $\widehat{f_{N}}-h$ is contained in $\{0\}$, its Fourier transform
$\mathcal{F}^{-1}(\widehat{f_{N}}-h)$ is a polynomial. Both $\mathcal{F}^{-1}(\widehat{f_{N}})(x)=f_{N}(x)$ and
$\mathcal{F}^{-1}(h)(x)$ tend to 0 as $x\to \infty$. Hence $\mathcal{F}^{-1}(\widehat{f_{N}}-h)=0$ and
$\widehat{f_{N}}=h$.
\end{proof}



\begin{lemma}\label{L:f-L1}
Let $0\leq j\leq n-1$ and $f\in I_{j}(\nu)_{K_{2}}$. \begin{enumerate}
\item[(1)]If $\Re\nu(H_0)<0$, then $f_{N}\in L^{1}$, and $|\xi|^{k}\widehat{f_{N}}$ is a bounded function over
$\mathfrak{n}^{\ast}$ for any $k\geq 0$.
\item[(2)]If $0\leq\Re\nu(H_0)<\rho'(H_0)$, put \[p_{0}= \frac{\rho'(H_0)}{\rho'(H_0)-\Re\nu(H_0)}\geq 1
\textrm{ and }q_0=\frac{\rho'(H_0)}{\Re\nu(H_0)}\geq 1.\] Then, $f_{N}\in L^{p}$ for any $p$ with $p>p_0$ and
$\widehat{f_{N}}\in L^{q}$ for any $q$ with $1\leq q<q_0$.
\end{enumerate}
\end{lemma}

\begin{proof}
If $\Re\nu(H_0)<0$, then $\partial^{\alpha}f(n_{x})\in L^{1}$ for any tuple $\alpha$ by Lemma~\ref{L:f-polynomial}.
Hence, $\xi^{\alpha}\widehat{f_{N}}(\xi)$ is bounded over $\mathfrak{n}^{\ast}$. Therefore, for any $k\geq 0$,
$|\xi|^{k}\widehat{f_{N}}$ is a bounded function over $\mathfrak{n}^{\ast}$. If $0\leq\Re\nu(H_0)<\rho'(H_0)$,
then by Lemma~\ref{L:f-polynomial} one gets $f_{N}\in L^{p}$ for any $p$ with $p>p_0$. By Lemma~\ref{L:f-Fourier},
one sees that $\widehat{f_{N}}\in L^{q}$ for any $q$ with $1\leq q<q_0$.
\end{proof}

\begin{lemma}\label{L:Riesz5}
Let $0\leq j\leq n-1$ and $f\in I_{j}(-\nu)_{K_{2}}$. If $-\rho'(H_0)<\Re\nu(H_0)<\rho'(H_0)$, then
\begin{equation}\label{Eq:intertwining-Fourier}
\mathcal{F}((J_{j}(\nu)f)_{N})=(2\pi)^{\frac{m}{2}}\mathcal{F}(T_{j}(\nu))\cdot \mathcal{F}(f_{N}).
\end{equation}
\end{lemma}

\begin{proof}
Recall that the operator $J_{j}(\nu)$ is defined as the convolution with $T_j(\nu)$ when $\Re\nu(H_0)>0$ and
it extends holomorphically to all $\nu\in \bbC$. Then $J_{j}(\nu)f\in I_{j}(\nu)_K$ and when
$-\rho'(H_0)<\Re\nu(H_0)<\rho'(H_0)$, $\mathcal{F}((J_{j}(\nu)f)_{N})$ is in $L^1$ by Lemma~\ref{L:f-Fourier}.
By Lemma~\ref{L:Riesz3} and Lemma~\ref{L:f-Fourier}, the right hand side of \eqref{Eq:intertwining-Fourier} is
$L^1$ near $0$, and decays rapidly when $|\xi|\to \infty$. Hence it is in $L^1(\mathfrak{n}^*)$.

Let $f_{0}: K_{2}\rightarrow V_{M,\mu_{j}}$ be a fixed $K_{2}$-finite function satisfying
\[f_{0}(km_0)=\sigma_j(m_0)^{-1}f_{0}(k),\quad  (k,m_0)\in K_{2}\times M_{2}.\]
Let $f=f_{\nu}$ be defined by
\[f_{\nu}(ka\bar{n})=e^{-(\nu-\rho')(H_{0})} f_{0}(k),\quad
 (k,a,\bar{n})\in K_{2}\times A\times\bar{N}.\]
Write
\[T_{\nu}:=T_{j}(\nu),\quad g_{\nu}:=T_{\nu}\ast (f_\nu)_{N} = J_{j}(\nu)(f_\nu)_{N}.\]
It is enough to show
\begin{equation}\label{Eq:convolution2}
\mathcal{F}(g_{\nu})
=(2\pi)^{\frac{m}{2}}\mathcal{F}(T_{\nu})\cdot \mathcal{F}((f_{\nu})_N).
\end{equation}

We first show (\ref{Eq:convolution2}) when $0<\nu(H_0)<\frac{1}{2}\rho'(H_0)$. Let
\[T_{\nu}^{+}=T_{\nu}\chi_{|x|\leq 1},\  T_{\nu}^{-}=T_{\nu}\chi_{|x|>1}, \text{ and }
g_{\nu}^{+}=T_{\nu}^{+}\ast f_{\nu},\  g_{\nu}^{-}=
T_{\nu}^{-}\ast f_{\nu},\]
where $\chi_{|x|\leq 1}$ and $\chi_{|x|>1}$
 are characteristic functions of ${|x|\leq 1}$ and ${|x|> 1}$, respectively.
Then $T_{\nu}^{+}\in L^{1}$ and $T_{\nu}^{-}\in L^2$.
By Lemma~\ref{L:f-Fourier}, $f_{\nu}\in L^{1}$.
Hence Lemma~\ref{L:Fourier-convolution2} implies that
\begin{align*}
&\mathcal{F}(g_{\nu}^{+})
=(2\pi)^{\frac{m}{2}}\mathcal{F}(T_{\nu}^{+})\cdot \mathcal{F}((f_{\nu})_N)
 \text{ and }\\
&\mathcal{F}(g_{\nu}^{-})
=(2\pi)^{\frac{m}{2}}\mathcal{F}(T_{\nu}^{-})\cdot \mathcal{F}((f_{\nu})_N).
\end{align*}
Taking the sum, one gets \eqref{Eq:convolution2}.

When $-\rho'(H_0)<\Re\nu(H_0)<\rho'(H_0)$,
 both sides of \eqref{Eq:convolution2} are smooth on $\mathfrak{n}^{\ast}-\{0\}$
 and they are in $L^1(\mathfrak{n}^*)$.
Moreover, they are holomorphic in $\nu$.
Therefore, \eqref{Eq:convolution2} holds
 for all $\nu$ with $-\rho'(H_0)<\Re\nu(H_0)<\rho'(H_0)$.
\end{proof}

The following proposition follows from Lemma \ref{L:Riesz3} and Lemma \ref{L:Riesz5}.

\begin{proposition}\label{P:Riesz5}
Let $0\leq j\leq n-1$ and $-\rho'(H_0)<\Re\nu(H_0)<\rho'(H_0)$. Then for any function $f\in I_{j}(-\nu)_{K_2}$,
\begin{equation*}
\mathcal{F}((J_{j}(\nu)f)_{N})(\xi)=\frac{\pi^{\frac{m}{2}}|\frac{1}{2}\xi|^{-2\nu(H_0)}}{\Gamma(\frac{m}{2}-
\nu(H_0)+1)}\Bigl(\frac{m}{2}-j-\nu(H_0)\sigma_j(r_{\xi})\Bigr)\mathcal{F}(f_{N})(\xi).
\end{equation*}
\end{proposition}

Put \[\nu_{j}=\Bigl(\frac{m}{2}-j\Bigr)\lambda_{0}.\] Then Proposition~\ref{P:Riesz5} gives

\begin{proposition}\label{P:Riesz4}
Let $1\leq j\leq n-1$. Then for any function $f\in I_{j}(-\nu_{j})_{K_2}$, \begin{equation*}
\mathcal{F}((J_{j}(\nu_{j})f)_N)(\xi)=\frac{\pi^{\frac{m}{2}}|\frac{1}{2}\xi|^{-m+2j}}{\Gamma(j+1)}
\Bigl(\frac{m}{2}-j\Bigr)(1-\sigma_j(r_{\xi}))\mathcal{F}(f_N)(\xi);\end{equation*}
and for any function $f\in I_{j}(\nu_{j})_{K_2}$, \begin{equation*}
\mathcal{F}((J_{j}(-\nu_{j})f)_N)(\xi)=\frac{\pi^{\frac{m}{2}}|\frac{1}{2}\xi|^{m-2j}}{\Gamma(m-j+1)}
\Bigl(\frac{m}{2}-j\Bigr)(1+\sigma_j(r_{\xi}))\mathcal{F}(f_N)(\xi).
\end{equation*}
\end{proposition}

\begin{lemma}\label{L:intertwining2}
Let $1\leq j\leq n-1$. \begin{enumerate}
\item[(1)]For any function $f\in I_{j}(-\nu_{j})_{K_2}$, \[(1+\sigma_j(r_{\xi}))\mathcal{F}((J_{j}(\nu_{j})f)_N)
=0;\] if $J_{j}(\nu_{j})f=0$, then $(1-\sigma_j(r_{\xi}))(\mathcal{F}(f_N))=0$.
\item[(2)]For any function $f\in I_{j}(\nu_{j})_{K_2}$, \[(1-\sigma_j(r_{\xi}))\mathcal{F}((J_{j}(-\nu_{j})f)_N)
=0;\] if $J_{j}(-\nu_{j})f=0$, then $(1+\sigma_j(r_{\xi}))\mathcal{F}(f_N)=0$.
\end{enumerate}
\end{lemma}

\begin{proof}
By Proposition \ref{P:Riesz4}, one gets the four equalities on $\mathfrak{n}^{\ast}-\{0\}$. By Lemma \ref{L:f-Fourier},
Fourier transformed functions appearing in the four equalities are always in $L^1$. Hence, one has such equalities
as tempered distributions.
\end{proof}

For two functions $f_1,f_2\in I_{j}(\nu)_{K_2}$ which lie in the image of $J_{j}(\nu)$, take a function
$\tilde{f}_{1}\in I_{j}(-\nu)_{K_2}$ such that $J_{j}(\nu)\tilde{f}_{1}=f_{1}$. Define
\begin{equation*}
(\mathcal{F}((\tilde{f}_{1})_{N})|\mathcal{F}((f_{2})_{N}))=
\int_{\mathfrak{n}^{\ast}}(\mathcal{F}(\tilde{f}_{1})(\xi),\mathcal{F}(f_{2})(\xi))\d\xi.
\end{equation*}

\begin{lemma}\label{L:innerProduct2}
Let $1\leq j\leq n-1$ and $-\rho'(H_0)<\nu(H_0)<\rho'(H_0)$. For any two functions $f_1,f_2\in I_{j}(\nu)_{K_{2}}$
which lie in the image of $J_{j}(\nu)$, we have
\begin{equation}\label{Eq:innerProduct3}
(\mathcal{F}((\tilde{f}_{1})_{N})|\mathcal{F}((f_{2})_{N}))=(f_{1},f_{2}).
\end{equation}
Here, the right hand side is defined in \eqref{Eq:Hermitian1}.
\end{lemma}

\begin{proof}
When $-\frac{1}{2}\rho'(H_0)<\nu(H_0)<\frac{1}{2}\rho'(H_0)$,
 both $\tilde{f}_1$ and $f_{2}$ are in $L^2$. By Parseval's theorem, one gets
\[(\mathcal{F}((\tilde{f}_{1})_{N})|\mathcal{F}((f_{2})_{N}))
 =((\tilde{f}_{1})_{N}|(f_{2})_{N})=(f_{1},f_{2}).\]

When $0<\nu(H_0)<\rho'(H_0)$, write $h_2=\mathcal{F}((f_2)_{N})$
 so that $(f_{2})_{N}$ is the Fourier transform of $h_2$.
By Lemma \ref{L:f-L1}, $(\tilde{f}_{1})_{N}\in L^{1}$, $h_{2}\in L^{1}$.
Then
\begin{align*}
(\mathcal{F}((\tilde{f}_{1})_{N})|h_{2})
&=(2\pi)^{\frac{m}{2}} \int_{\bbR^m}\int_{\bbR^m}
   (\tilde{f}_1)_N(x) e^{\mathbf{i}x\xi} \overline{h_2(\xi)} \d x \d\xi \\
&=((\tilde{f}_{1})_{N}|(f_{2})_{N})=(f_{1},f_{2}).
\end{align*}

When $-\rho'(H_0)<\nu(H_0)< 0$, we have
 $(f_{2})_{N}\in L^{1}$ and $\mathcal{F}((\tilde{f}_1)_{N})\in L^1$.
Then \eqref{Eq:innerProduct3} is proved in the same way.
\end{proof}

Similarly to \S\ref{SS:irreducible}, we write $\pi_j$ for the image of $J_{j}(\nu_j)$
 and write $\pi'_j$ for the image of $J_{j}(-\nu_j)$.
The following fact is a special case of
 Facts~\ref{F:gamma-classification}, \ref{F:unitarizable},
 \ref{F:gamma-classification2}, \ref{F:unitarizable2}
 (cf.\ \cite{Collingwood}, \cite{Kobayashi-Speh}).

\begin{fact}\label{L:rho-classification}
For each $j$, both $\pi_{j}$ and $\pi'_{j}$ are irreducible representations of $G_{2}$. When $0\leq j\leq n-2$,
$\pi_{j+1}\cong \pi'_{j}$ and they are infinite-dimensional unitarizable representations, and $\pi_{0}$ is the
trivial representation. If $m$ is odd, then $\pi'_{n-1}$ is a discrete series; if $m$ is even, then
$I_{n-1}(\pm{\nu_{n-1}})$ is irreducible and hence $\pi'_{n-1}\cong \pi_{n-1}\cong \pi'_{n-2}$.
\end{fact}

We write $\bar{\pi}_j$ and $\bar{\pi}'_j$ for the Hilbert completion of $\pi_j$ and $\pi'_j$, respectively.
They are irreducible unitary representations of $G_2$.

The following proposition follows from Proposition~\ref{P:Riesz4} and Lemma~\ref{L:innerProduct2}.

\begin{proposition}\label{P:image-Fourier2}
For each $j$ with $1\leq j\leq n-1$, the inverse Fourier transform $\mathcal{F}$ gives maps \[\mathcal{F}\colon
J_{j}(\nu_{j})(I_j(-\nu_j))\to L^{2}(\mathfrak{n}^{\ast}-\{0\},(V_j)^{-\sigma_j(r_{\xi})},|\xi|^{m-2j}\d \xi),\]
and \[\mathcal{F}\colon J_{j}(-\nu_{j})(I_j(\nu_j))\to L^{2}(\mathfrak{n}^{\ast}-\{0\},
(V_j)^{\sigma_j(r_{\xi})},|\xi|^{-m+2j}\d \xi),\] which are isometries up to scalars.
\end{proposition}

\begin{proof}
We prove the first statement. The second statement can be proved similarly.

We first prove the statement for $K_{2}$-finite functions. Let $f\in I_{j}(\nu_{j})_{K_2}$ be equal to
$J_{j}(\nu_{j})(\tilde{f})$ for some function $\tilde{f}\in I_{j}(-\nu_{j})_{K_2}$. By Lemma~\ref{L:intertwining2}, $(1+\sigma_j(r_{\xi}))\mathcal{F}(f_{N})(\xi)=0$. Hence, $\mathcal{F}(f_{N})(\xi)\in (V_j)^{-\sigma_j(r_{\xi})}$
for all $\xi\in\mathfrak{n}^{\ast}-\{0\}$. By Lemma \ref{L:innerProduct2},
\[(\mathcal{F}(f_{N})|\mathcal{F}(\tilde{f}_{N}))=(f_{N},f_{N})<\infty.\] By Proposition~\ref{P:Riesz4},
\[\mathcal{F}(f_N)(\xi) =\frac{\pi^{\frac{m}{2}}|\frac{1}{2}\xi|^{-m+2j}}{\Gamma(j+1)}
\Bigl(\frac{m}{2}-j\Bigr)(1-\sigma_j(r_{\xi}))\mathcal{F}(\tilde{f}_N)(\xi).\]
Therefore, the statement follows for $K_{2}$-finite functions.

For a general smooth function $f\in J_{j}(\nu_{j})(I_j(-\nu_j))$, Write $f=\sum_{\sigma\in\widehat{K}}f_{\sigma}$
for the $K_{2}$-type decomposition of $f$. Put \[f_{n}=\sum_{|\sigma|<n}f_{\sigma},\] where $|\sigma|$ denotes
the norm of the highest weight. Then, $\lim_{n\rightarrow\infty}f_{n}=f$ both as sections of a vector bundle on
$K_{2}/M_{2}\cong S^{m}$ with respect to the $C^{0}$-norm and as vectors in $\bar{\pi}_{j}$ with respect to the
inner product. Since $\lim_{n\rightarrow\infty}f_{n}=f$ as sections of the vector bundle, we have
$\lim_{n\rightarrow\infty}(f_{n})|_{N}=f|_{N}$ as tempered distributions.
Then, \[\lim_{n\rightarrow\infty}\mathcal{F}((f_{n})|_{N})=\mathcal{F}(f|_{N})\] as tempered distributions.
Since $\lim_{n\rightarrow\infty}f_{n}=f$ with respect to the inner product of $\bar{\pi}_{j}$, there exists a
function \[g\in L^{2}(\mathfrak{n}^{\ast}-\{0\},(V_j)^{-\sigma_j(r_{\xi})}, |\xi|^{m-2j}\d \xi)\] such that \[\lim_{n\rightarrow\infty}\mathcal{F}((f_{n})|_{N})=g\] in this $L^2$ space. Then one must have
\[\mathcal{F}(f|_{N})=g\in L^{2}(\mathfrak{n}^{\ast}-\{0\},(V_j)^{-\sigma_j(r_{\xi})},|\xi|^{m-2j}\d \xi).
\qedhere\]
\end{proof}

\begin{remark}
In light of Proposition~\ref{P:Riesz5} and Lemma~\ref{L:innerProduct2}, the Hermitian form
\eqref{Eq:Hermitian1} on $I_j(\nu)$ is positive definite if $-\frac{m}{2}+j < \nu(H_0)< \frac{m}{2}-j$.
In this case, the Hilbert completion of $I_j(\nu)$ is called a complementary series representation
 (cf.\ \cite[Remark 4.10]{Fischmann-Orsted}, \cite[Section 4]{Speh-Venkataramana}). One can determine
their restriction to $P_{2}$ similarly as that in Proposition~\ref{P:P'-restriction1}.
\end{remark}

\subsection{Anti-trivialization and restriction to $P_2$}\label{SS:restriction2}

As in \eqref{Eq:anti-trivialization},
 for a $V_j$-valued function $h$ on $\mathfrak{n}^{\ast}-\{0\}$,
 define its anti-trivialization by
$h_{at,\nu}(p)=(p^{-1}\cdot h)(\xi_0)$ for $p\in P$.

\begin{proposition}\label{P:anti-trivialization2}
The map $h\mapsto h_{at,\nu}$ gives an isomorphism \[L^{2}(\mathfrak{n}^{\ast}-\{0\},V_j,|\xi|^{2\nu(H_0)}
\d\xi)\xrightarrow{\sim}\Ind_{M'_{2}N}^{M_{2}AN}(V_j|_{M'_{2}} \otimes e^{\mathbf{i}\xi_{0}}),\] which
preserves inner products and actions of $P_{2}$. Moreover, the map $h\mapsto h_{at,\nu_{j}}$ gives
\[L^{2}(\mathfrak{n}^{\ast}-\{0\},(V_j)^{-\sigma_j(r_{\xi})},|\xi|^{m-2j}\d \xi)\xrightarrow{\sim}
\Ind_{M'_{2}N}^{M_{2}AN}((V_j)^{-\sigma_j(r_{\xi_0})}\otimes e^{\mathbf{i}\xi_{0}})\] and the map
$h\mapsto h_{at,-\nu_{j}}$ gives \[L^{2}(\mathfrak{n}^{\ast}-\{0\},(V_j)^{\sigma_j(r_{\xi})},|\xi|^{-m+2j}
\d\xi)\xrightarrow{\sim}\Ind_{M'_{2}N}^{M_{2}AN}((V_j)^{\sigma_j(r_{\xi_0})}\otimes e^{\mathbf{i}\xi_{0}}).\]
\end{proposition}

\begin{proof}
The first statement can be shown along the same way as the proof for Proposition~\ref{P:anti-trivialization}.
The second and the third statements follow from the first one easily.
\end{proof}

\begin{proposition}\label{P:P'-restriction1}
For each $j$ with $1\leq j\leq n-1$, the restriction of unitary representations $\bar{\pi}_j,\bar{\pi}'_j$ to
$P_2$ are as follows: \[\bar{\pi}'_j|_{P_2}\simeq\Ind_{M'_{2}N}^{M_{2}AN}\bigl(\bigwedge^{j}\mathbb{C}^{m-1}
\otimes e^{\mathbf{i}\xi_{0}}\bigr)\textrm{ and }\bar{\pi}_j|_{P_2}\simeq\Ind_{M'_{2}N}^{M_{2}AN}
\bigl(\bigwedge^{j-1}\mathbb{C}^{m-1}\otimes e^{\mathbf{i}\xi_{0}}\bigr).\]
\end{proposition}

\begin{proof}
Recall that the actions of $P_2$ on $\pi_{j}$, $\pi'_{j}$
 and on their Fourier transforms are
 given by \eqref{Eq:P-action} and \eqref{Eq:P-action2}.

Then by Proposition~\ref{P:image-Fourier2},
 $\mathcal{F}$ composed with anti-trivialization gives
 a $P_{2}$-equivariant isometric (up to scalar) embedding from $\pi_{j}$ or $\pi'_{j}$
 into unitarily induced representation of $P$ in the proposition.
Extending to the unitary completion, it gives
 a $P_{2}$-equivariant embedding
 from the unitary completion of $\pi_{j}$ or $\pi'_{j}$
 into unitarily induced representation of $P$.
Since $V_j|_{M_{2}}\cong\bigwedge^{j}\mathbb{C}^{m}$, and
\[(\bigwedge^{j}\mathbb{C}^{m})^{\sigma_j(r_{\xi_0})}\cong\bigwedge^{j}\mathbb{C}^{m-1}
\text{ and }
(\bigwedge^{j}\mathbb{C}^{m})^{-\sigma_j(r_{\xi_0})}\cong\bigwedge^{j-1}\mathbb{C}^{m-1}\]
 are irreducible representations of $M'_{2}$.
It follows that unitarily induced representations in the proposition are irreducible.
Therefore, the $P_{2}$-equivariant embeddings are isomorphisms.
\end{proof}

\subsection{Restriction to $P$}\label{SS:restriction3}

For each $j$ with $0\leq j\leq n-1$, write \[\pi_{j}(\rho)=\pi_{j}|_{G_{1}}\textrm{ and }\pi'_{j}(\rho)=
\pi'_{j}|_{G_{1}}.\] These representations of $G_1$ are also regarded as representations of $G$ via the
covering map $G\to G_1$. We use the same notation for these representations of $G$, which agrees with the
notation in \S\ref{SS:irreducible}.

In this subsection we suppose $m$ is odd so that $m=2n-1$ and $G=\Spin(2n,1)$. Write $\pi^{+}(\rho)$ for the
discrete series with lowest $K$-type $V_{K,(\underbrace{\scriptstyle 1,\dots,1}_{n})}$; and write $\pi^{-}(\rho)$
for the discrete series with lowest $K$-type $V_{K,(\underbrace{\scriptstyle 1,\dots,1}_{n-1},-1)}$. Then $\pi'_{n-1}(\rho)\cong\pi^{+}(\rho)\oplus\pi^{-}(\rho)$. Other representations $\pi_{j}$
($0\leq j\leq n-1$) are irreducible. Write $\bar{\pi}_{j}(\rho)$, $\bar{\pi}^{+}(\rho)$, $\bar{\pi}^{-}(\rho)$
for the Hilbert completion of $\pi_{j}(\rho)$, $\pi^{+}(\rho)$, $\pi^{-}(\rho)$, respectively.

The following proposition is a direct consequence of Proposition~\ref{P:P'-restriction1}.
It is a special case of Theorem~\ref{T:branching-regular}.
\begin{proposition}\label{P:P-restriction1}
For each $1\leq j\leq n-1$, \begin{align*}
&\bar{\pi}_{j}(\rho)|_{P}\cong\Ind_{M'N}^{MAN}(\bigwedge^{j-1}\mathbb{C}^{2n-2}\otimes e^{\mathbf{i}\xi_{0}})
\text { and }\\
&\bar{\pi}^{+}(\rho)|_{P}\oplus\bar{\pi}^{-}(\rho)|_{P}
 \cong\Ind_{M'N}^{MAN}(\bigwedge^{n-1}\mathbb{C}^{2n-2}\otimes e^{\mathbf{i}\xi_{0}}).
\end{align*}
The restrictions $\bar{\pi}_{j}(\rho)|_{P}$ ($1\leq j\leq n-1$), $\bar{\pi}^{+}(\rho)|_{P}$,
$\bar{\pi}^{-}(\rho)|_{P}$ are all irreducible.
\end{proposition}

The restriction $\bigwedge^{n-1}\mathbb{C}^{2n-2}|_{M'}$
 is the direct sum of two finite-dimensional irreducible representations of $M'=\Spin(2n-2)$
 with highest weights
\[\mu^+=(\underbrace{1,\dots,1}_{n-1})\textrm{ and }\mu^-=(\underbrace{1,\dots,1}_{n-2},-1),\]
 respectively.
After Proposition~\ref{P:P-restriction1}, we need to determine
 whether $\bar{\pi}^{-}(\rho)|_{P}$ is isomorphic to
 $\Ind_{M'N}^{MAN}(V_{M',\mu^+}\otimes e^{\mathbf{i}\xi_{0}})$
 or $\Ind_{M'N}^{MAN}(V_{M',\mu^-}\otimes e^{\mathbf{i}\xi_{0}})$.
In order to do this, we calculate the Fourier transform
 of a specific $K$-type function in $\pi^{-}(\rho)$.
For $f$ in a small $K$-type, the explicit form of $f|_N$
 was given in Kobayashi-Speh \cite[\S8.2]{Kobayashi-Speh}.
We follow their description and then we calculate its Fourier transform.

One has \[\bigwedge^{n}\mathbb{C}^{2n}\cong V_{K,\lambda^+}\oplus V_{K,\lambda^-},
 \text{ where }
\lambda^+=(\underbrace{1,\dots,1}_{n})
\textrm{ and }\lambda^-=(\underbrace{1,\dots,1}_{n-1},-1).\]
Note that $V_{K,\lambda^+}$ is the lowest $K$-type of $\pi^{+}(\rho)$,
 and $V_{K,\lambda^-}$ is the lowest $K$-type of $\pi^{-}(\rho)$.

Put $V=\mathbb{C}^{2n}$ and $V'=\mathbb{C}^{2n-1}$. Let $\{e_{j}:1\leq j\leq 2n\}$ be the standard orthonormal basis
of $V$. Then, $V=V'\oplus\mathbb{C}e_{2n}$ and this decomposition induces
\[\bigwedge^{n}V=\bigwedge^{n}V'\oplus \Bigl(\bigwedge^{n-1}V'\wedge\mathbb{C}e_{2n}\Bigr).\]
Let $p\colon V\rightarrow V'$ be the projection along $\mathbb{C}e_{2n}$.
Then it induces the projection $\bigwedge^{n}V\rightarrow\bigwedge^{n}V'$
 along $\bigwedge^{n-1}V'\wedge\mathbb{C}e_{2n}$, still denoted by $p$.
Note that $p$ is $M$-equivariant.
For each $u\in\bigwedge^{n}V$, define
\[f_{u}(ka\bar{n})=e^{(\rho'+\nu_{n-1})\log a}p(k^{-1}u),
 \quad ka\bar{n}\in KA\bar{N}, \]
which belongs to
 $\Ind_{MA\bar{N}}^{G}(\bigwedge^{n}\mathbb{C}^{2n-1}\otimes
 e^{-\nu_{n-1}-\rho'}\otimes\mathbf{1}_{\bar{N}})
 \cong I_{n-1}(-\nu_{n-1})$.
This isomorphism is induced by $\bigwedge^{n}V'|_{M'}\cong\bigwedge^{n-1}V'|_{M'}$.
By (\ref{Eq:Iwasawa3}), we have
\begin{equation*}
 f_{u}(n_{x})=(1+|x|^{2})^{-n}p(s_{x}^{-1}u),
\end{equation*}
where $s_x$ is as in \eqref{Eq:sx}.

Write $v_{j}=e_{2j-1}+\mathbf{i}e_{2j}$ for each $1\leq j\leq n$.
Put
\[u^{+}=v_{1}\wedge\cdots\wedge v_{n} \text{ and }
u^{-}=v_{1}\wedge\cdots\wedge v_{n-1}\wedge(e_{2n-1}-\mathbf{i}e_{2n}).\]
Then $u^{+}$ (resp.\ $u^{-}$)
 is a nonzero vector in $\bigwedge^{n}\mathbb{C}^{2n}$
 with highest weight $\lambda^+$ (resp.\ $\lambda^-$).
Then $f_{u^{-}}$ gives a function in
 $\pi^{-}(\rho)\subset
 \Ind_{MA\bar{N}}^{G}(\bigwedge^{n}\mathbb{C}^{2n-1}\otimes
 e^{-\nu_{n-1}-\rho'}\otimes\mathbf{1}_{\bar{N}})$,
 corresponding to a highest weight vector of the lowest $K$-type of $\pi^{-}(\rho)$.
Below we calculate the inverse Fourier transform of $f_{u^{-}}$.

Put $y=(x,1)\in\mathbb{R}^{2n}$. Let $r'_{x}$ denote both
\[I_{2n}-\frac{2y^{t}y}{|y|^{2}}\in\OO(2n) \text{ and }
\diag\Bigl\{I_{2n}-
\frac{2y^{t}y}{|y|^{2}},1\Bigr\}\in\OO(2n,1).\]
Note that $s_{x}=sr'_{x}$. Then
\[p(s_{x}^{-1}u^{-})=p(r'_{x}su^{-})=p(r'_{x}u^{+}).\]
Set $x_{2n}=1$ for notational convenience,
 but we keep $|x|^2=x_1^2+\cdots+x_{2n-1}^2$.

\begin{lemma}\label{L:rv+}
One has
\begin{align}\label{Eq:rv+}
&\ \ r'_{x}u^{+}\\\nonumber &=u^{+}+\sum_{1\leq k\leq n}(-1)^{k}\frac{2(x_{2k-1}+\mathbf{i}x_{2k})}
 {1+|x|^{2}}(x_{2k-1}e_{2k-1}+x_{2k}e_{2k})\\ \nonumber
&\qquad \qquad \qquad \qquad \wedge v_{1}\wedge \cdots \wedge\hat{v_{k}}\wedge\cdots\wedge v_{n}
 \\  \nonumber
&\quad +\sum_{1\leq k<j\leq n}(-1)^{j-k}\frac{2\mathbf{i}(x_{2k-1}+\mathbf{i}x_{2k})
 (x_{2j-1}+\mathbf{i}x_{2j})}
 {1+|x|^{2}}e_{2j-1}\wedge e_{2j}\\  \nonumber
&\qquad \qquad \qquad \qquad
  \wedge v_{1}\wedge \cdots\wedge\hat{v_{k}}\wedge\cdots\wedge\hat{v_{j}}
 \wedge\cdots\wedge v_{n}\\ \nonumber
&\quad +\sum_{1\leq j<k\leq n}(-1)^{j-k-1}\frac{2\mathbf{i}(x_{2k-1}+\mathbf{i}x_{2k})
 (x_{2j-1}+\mathbf{i}x_{2j})}{1+|x|^{2}}e_{2j-1}\wedge e_{2j} \\ \nonumber
&\qquad\qquad\qquad \qquad
 \wedge v_{1}\wedge \cdots\wedge\hat{v_{j}}
  \wedge\cdots\wedge\hat{v_{k}}\wedge\cdots\wedge v_{n}.
\end{align}
\end{lemma}

\begin{proof}
We have
\[r'_{x}u^{+}=r'_{x}v_1\wedge \cdots \wedge r'_{x}v_n.\]
Since $r'_{x}v_k-v_k$ is proportional to $y$,
\[r'_{x}u^{+}
=u^{+}
 + \sum_{k=1}^n (-1)^{k-1}
 (r'_{x}v_k -v_k) \wedge
 v_1\wedge \cdots \wedge \hat{v_k} \wedge
 \cdots  \wedge v_n.\]
Then we calculate
\begin{align*}
 r'_{x}v_k-v_k
&=-\frac{2(v_k,y)}{|y|^2}y \\
&= - \sum_{i=1}^n \frac{2(x_{2k-1}+\mathbf{i}x_{2k})}{1+|x|^2}(x_{2i-1}e_{2i-1}+x_{2i}e_{2i}).
\end{align*}
The term for $i=k$ corresponds to the
 second term of the right hand side of \eqref{Eq:rv+}.
Then the lemma follows from
\[(x_{2j-1}e_{2j-1}+x_{2j}e_{2j})\wedge v_j
=\mathbf{i}(x_{2j-1}+\mathbf{i}x_{2j}) e_{2j-1} \wedge e_{2j}.
\qedhere\]
\end{proof}



To calculate the inverse Fourier transform
 $\mathcal{F}((f_{u^{-}})_N)$ we need some formulas.
\begin{align} \label{Eq:F-formula1}
&\mathcal{F}(1+|x|^2)^{-n}
 =\frac{2^{\frac{1}{2}-n}\pi^{\frac{1}{2}}}{(n-1)!}e^{-|\xi|},\\
\label{Eq:F-formula2}
&\mathcal{F}(1+|x|^2)^{-(n+1)}=\frac{2^{-\frac{1}{2}-n}\pi^{\frac{1}{2}}}{n!}(1+|\xi|)e^{-|\xi|}.
\end{align}

First, the Fourier transform of the function of one variable
 $\sqrt{\frac{\pi}{2}}e^{-|\xi|}$ is equal to $(1+x^{2})^{-1}$.
By the Fourier inversion formula, we get $\mathcal{F}(1+x^{2})^{-1}=\sqrt{\frac{\pi}{2}}e^{-|\xi|}$.
Using $(1+x^{2})^{-2}=(1+\frac{x}{2}\frac{d}{dx})(1+x^{2})^{-1}$,
 we see that $\mathcal{F}(1+x^{2})^{-2}=\frac{\sqrt{\pi}}{2\sqrt{2}}(1+|\xi|)e^{-|\xi|}$.
Then as in the proof of Lemma~\ref{L:Fourier-KBessel3},
 we obtain \eqref{Eq:F-formula1} and \eqref{Eq:F-formula2}.

Then by \eqref{Eq:Mult-Differ}, we obtain the following.
\begin{align} \label{Eq:F-formula3}
&\mathcal{F}(x_j(1+|x|^2)^{-(n+1)})
 = \mathbf{i} \frac{2^{-\frac{1}{2}-n}\pi^{\frac{1}{2}}}{n!}\xi_je^{-|\xi|},\\
\label{Eq:F-formula4}
&\mathcal{F}(x_j^2(1+|x|^2)^{-(n+1)})
 = \frac{2^{-\frac{1}{2}-n}\pi^{\frac{1}{2}}}{n!}
  \Bigl(1-\frac{\xi_j^2}{|\xi|}\Bigr)e^{-|\xi|},\\
\label{Eq:F-formula5}
&\mathcal{F}(x_jx_k(1+|x|^2)^{-(n+1)})
 = -\frac{2^{-\frac{1}{2}-n}\pi^{\frac{1}{2}}}{n!}
  \frac{\xi_j\xi_k}{|\xi|} e^{-|\xi|}\quad (j\neq k).
\end{align}

\begin{lemma}\label{P:Ffnx}
We have
\[
\mathcal{F}(f_{u^{-}})_{N}
=\frac{2^{-\frac{1}{2}-n}\pi^{\frac{1}{2}}}{n!}e^{-|\xi|}
 (|\xi|(1-r_{\xi})u+2u'\wedge\xi)\]
 at $0\neq\xi\in\mathbb{R}^{2n-1}$, where
\begin{align*}
&u=v_{1}\wedge\cdots\wedge v_{n-1}\wedge e_{2n-1}\in\bigwedge^{n}V' \text{ and }\\
& u'=v_{1}\wedge\cdots\wedge v_{n-1}\in\bigwedge^{n-1}V'.
\end{align*}
\end{lemma}

\begin{proof}
By Lemma \ref{L:rv+}, we have
\begin{align*}
&f_{u^{-}}(n_{x})
=(1+|x|)^{-n}v_{1}\wedge\cdots \wedge v_{n-1}\wedge e_{2n-1}\\
&+\sum_{1\leq k\leq n-1}(-1)^{k}2(1+|x|)^{-n-1} (x_{2k-1}+\mathbf{i}x_{2k})
 (x_{2k-1}e_{2k-1}+x_{2k}e_{2k})\\
&\qquad\qquad\qquad\qquad
 \wedge v_{1}\wedge\cdots\wedge\hat{v_{k}}\wedge\cdots\wedge v_{n-1}\wedge e_{2n-1}\\
&+(-1)^{n}2(1+|x|)^{-n-1}(x_{2n-1}+\mathbf{i})x_{2n-1}e_{2n-1}\wedge
  v_{1}\wedge\cdots\wedge v_{n-1}\\
&+\sum_{1\leq k<j\leq n-1}(-1)^{j-k}2(1+|x|)^{-n-1}
 \mathbf{i}(x_{2k-1}+\mathbf{i}x_{2k})(x_{2j-1}+\mathbf{i}x_{2j})\\
&\qquad\qquad\qquad
 e_{2j-1}\wedge e_{2j}\wedge v_{1} \wedge\cdots\wedge\hat{v_{k}}\wedge\cdots\wedge
 \hat{v_{j}}\wedge\cdots\wedge v_{n-1}\wedge e_{2n-1}\\
&+\sum_{1\leq j<k\leq n-1}(-1)^{j-k-1}
2(1+|x|)^{-n-1}\mathbf{i}(x_{2k-1}+\mathbf{i}x_{2k})(x_{2j-1}+\mathbf{i}x_{2j})\\
&\qquad\qquad\qquad  e_{2j-1}\wedge e_{2j}\wedge v_{1}
\wedge\cdots\wedge\hat{v_{j}}\wedge\cdots\wedge\hat{v_{k}}\wedge\cdots\wedge
 v_{n-1}\wedge e_{2n-1}\\
&+\sum_{1\leq j\leq n-1}(-1)^{j-n-1}
2(1+|x|)^{-n-1}\mathbf{i}(x_{2n-1}+\mathbf{i})(x_{2j-1}+\mathbf{i}x_{2j})\\
&\qquad\qquad\qquad\qquad e_{2j-1}\wedge e_{2j}\wedge v_{1}\wedge
\cdots\wedge\hat{v_{j}}\wedge\cdots\wedge v_{n-1}.
\end{align*}
Then by \eqref{Eq:F-formula1} -- \eqref{Eq:F-formula5},
\begin{align*}
&\Bigl(\frac{2^{\frac{1}{2}-n}\pi^{\frac{1}{2}}}{n!}e^{-|\xi|}\Bigr)^{-1}
 \mathcal{F}((f_{u^{-}})_{N})
= nv_{1}\wedge\cdots \wedge v_{n-1}\wedge e_{2n-1}\\
&+\sum_{1\leq k\leq n-1}(-1)^{k}
 \Bigl(\Bigl(1-\frac{\xi_{2k-1}^{2}}{|\xi|}
 -\mathbf{i}\frac{\xi_{2k-1}\xi_{2k}}{|\xi|}\Bigr)e_{2k-1}
 +\Bigl(\mathbf{i}-\mathbf{i}\frac{\xi_{2k}^{2}}{|\xi|}
 -\frac{\xi_{2k-1}\xi_{2k}}{|\xi|}\Bigr)e_{2k}\Bigr)\\
&\qquad\qquad\qquad\qquad
 \wedge v_{1}\wedge\cdots\wedge\hat{v_{k}}\wedge\cdots\wedge v_{n-1}\wedge e_{2n-1}\\
&+(-1)^{n}\Bigl(1-\frac{\xi_{2n-1}^{2}}{|\xi|}-\xi_{2n-1}\Bigr)
 e_{2n-1}\wedge v_{1}\wedge\cdots\wedge v_{n-1}\\
&+\sum_{1\leq k<j\leq n-1}(-1)^{j-k-1}\mathbf{i}|\xi|^{-1}(\xi_{2k-1}+\mathbf{i}\xi_{2k})
 (\xi_{2j-1}+\mathbf{i}\xi_{2j})\\
&\qquad\qquad\qquad  e_{2j-1}\wedge e_{2j}\wedge
 v_{1}\wedge\cdots\wedge\hat{v_{k}}\wedge\cdots\wedge\hat{v_{j}}\wedge\cdots
 \wedge v_{n-1}\wedge e_{2n-1}\\
&+\sum_{1\leq j<k\leq n-1}(-1)^{j-k}
 \mathbf{i}|\xi|^{-1}(\xi_{2k-1}+\mathbf{i}\xi_{2k})(\xi_{2j-1}+\mathbf{i}\xi_{2j})\\
&\qquad\qquad\qquad
 e_{2j-1}\wedge e_{2j}\wedge v_{1}\wedge\cdots\wedge\hat{v_{j}}\wedge\cdots\wedge\hat{v_{k}}
 \wedge \cdots\wedge v_{n-1}\wedge e_{2n-1}\\
&+\sum_{1\leq j\leq n-1}(-1)^{j-n}\mathbf{i}(1+|\xi|^{-1}\xi_{2n-1})
 (\xi_{2j-1}+\mathbf{i}\xi_{2j})\\
&\qquad\qquad\qquad\qquad
 e_{2j-1}\wedge e_{2j}\wedge v_{1}\wedge\cdots\wedge\hat{v_{j}}\wedge\cdots\wedge v_{n-1}.
\end{align*}
Similarly to Lemma \ref{L:rv+}, we have
\begin{align*}
r_{\xi}(u)&=
u+\sum_{1\leq k\leq n-1}(-1)^{k}\frac{2(\xi_{2k-1}+\mathbf{i}\xi_{2k})}{|\xi|^{2}}
(\xi_{2k-1}e_{2k-1}+\xi_{2k}e_{2k})\\
&\qquad\qquad\qquad\qquad
  \wedge v_{1}\wedge\cdots\wedge\hat{v_{k}}\wedge\cdots\wedge v_{n-1}\wedge e_{2n-1}\\
&+(-1)^{n}\frac{2\xi_{2n-1}}{|\xi|^{2}}\xi_{2n-1}e_{2n-1}\wedge v_{1}\wedge\cdots\wedge v_{n-1}\\
&+\sum_{1\leq k<j\leq n-1}(-1)^{j-k}\frac{2\mathbf{i}(\xi_{2k-1}+\mathbf{i}\xi_{2k})(\xi_{2j-1}+\mathbf{i}\xi_{2j})}
{|\xi|^{2}}e_{2j-1}\wedge e_{2j}\\
&\qquad\qquad\qquad\qquad
 \wedge v_{1}\wedge\cdots\wedge\hat{v_{k}}\wedge\cdots\wedge\hat{v_{j}}\wedge
\cdots\wedge v_{n-1}\wedge e_{2n-1}\\
&+\sum_{1\leq j<k\leq n-1}(-1)^{j-k-1}\frac{2\mathbf{i}(\xi_{2k-1}+\mathbf{i}\xi_{2k})
(\xi_{2j-1}+\mathbf{i}\xi_{2j})}{|\xi|^{2}}e_{2j-1}\wedge e_{2j}\\
&\qquad\qquad\qquad\qquad \wedge v_{1}\wedge\cdots\wedge\hat{v_{j}}
\wedge\cdots\wedge\hat{v_{k}}\wedge\cdots\wedge v_{n-1}\wedge e_{2n-1}\\
&+\sum_{1\leq j\leq n-1}(-1)^{j-n-1}
\frac{2\mathbf{i}\xi_{2n-1}(\xi_{2j-1}+\mathbf{i}\xi_{2j})}{|\xi|^{2}}e_{2j-1}\wedge e_{2j}\\
&\qquad\qquad\qquad\qquad
  \wedge v_{1}\wedge \cdots\wedge\hat{v_{j}}\wedge\cdots\wedge v_{n-1}.
\end{align*}
The lemma follows from these equations.
\end{proof}

\begin{proposition}\label{P:P-restriction2}
One has
\[\bar{\pi}^{+}(\rho)|_{P}\cong\Ind_{M'N}^{MAN}(V_{M',\mu^-}\otimes e^{\mathbf{i}\xi_{0}})
 \text{ and }
\bar{\pi}^{-}(\rho)|_{P} \cong\Ind_{M'N}^{MAN}(V_{M',\mu^+}\otimes e^{\mathbf{i}\xi_{0}}).\]
\end{proposition}

\begin{proof}
Let $h:=\mathcal{F}((f_{u^{-}})_N)$.
By Lemma~\ref{P:Ffnx},
\[h(\xi)=\frac{2^{-\frac{1}{2}-n}\pi^{\frac{1}{2}}}{n!}e^{-|\xi|}(|\xi|(1-r_{\xi})u+2u'\wedge\xi)\]
for $\xi\neq 0$.
Evaluating at $\xi=\xi_0=e_{2n-1}$, we have
\[h(e_{2n-1})=\frac{2^{-\frac{1}{2}-n}\pi^{\frac{1}{2}}}{n! e}\cdot (4u).\]
Hence
\[h_{at,\nu}(e)=h(\xi_0) = cu\]
 for a constant $c\neq 0$.
This is a highest weight vector for $M'$ with weight $\mu^+$
 in the representation $\bigwedge^{n}V'|_{M'}\cong\bigwedge^{n-1}V'|_{M'}$.
Hence the inverse Fourier transform of $f_{u^{-}}$
 must lie in $\Ind_{M'N}^{MAN}(V_{M',\mu^+}\otimes e^{\mathbf{i}\xi_{0}})$.
Therefore,
\[\bar{\pi}^{-}(\rho)|_{P}\cong\Ind_{M'N}^{MAN}(V_{M',\mu^+}\otimes e^{\mathbf{i}\xi_{0}})\]
 and then
\[\bar{\pi}^{+}(\rho)|_{P}\cong\Ind_{M'N}^{MAN}(V_{M',\mu^-}\otimes e^{\mathbf{i}\xi_{0}}).
\qedhere\]
\end{proof}

\section{Obtaining $\bar{\pi}|_{MN}$ from $\bar{\pi}|_{K}$}\label{S:MN.K}

By the Cartan decomposition $\mathfrak{g}=\mathfrak{k}\oplus\mathfrak{p}$ and the decomposition $\mathfrak{g}=
\mathfrak{m}\oplus\mathfrak{a}\oplus\mathfrak{n}\oplus\bar{\mathfrak{n}}$, we have $\mathfrak{k}=\mathfrak{m}
\oplus\{X+\theta(X):X\in\mathfrak{n}\}$. By this, we can view $\mathfrak{m}+\mathfrak{n}$ as the limit
$\lim_{t\rightarrow+\infty}\Ad(\exp(tH_{0}))(\mathfrak{k})$. On the group level, we can view $MN$ as the limit
$\lim_{t\rightarrow+\infty}\exp(tH_{0})K\exp(-tH_{0})$.
From this viewpoint, we may expect that $\bar{\pi}|_{MN}$ is determined
by $\bar{\pi}|_{K}$ for any unitarizable irreducible representation $\pi$ of $G$. Here we observe the relationship between two restrictions. The writing of this section is motivated by a question of Professor David Vogan.

As in Section \ref{S:repP}, write $I_{P,\tau}=\Ind_{M'N}^{P}(\tau\otimes e^{\mathbf{i}\xi_{0}})$ for a
unitarily induced representation of $P$ where $\tau$ is a finite-dimensional unitary representation of
$M'$. Then, $I_{P,\tau}$ is irreducible when $\tau$ is so. For any finite-dimensional unitary representation
$\tau$ of $M'$ and any $0\neq t\in\mathbb{R}$, write $I_{t,\tau}=\Ind_{M'N}^{MN}(\tau\otimes
e^{\mathbf{i}t\xi_{0}})$ for a unitarily induced representation of $MN$. Then, $I_{t,\tau}$ is irreducible when $\tau$ is so.

Using Mackey's theory for unitarily induced representations and considering the action of $MN$ on $P/M'N$,
the following lemma follows easily.

\begin{lemma}\label{L:P-MN}
We have \[I_{P,\tau}|_{MN}\cong\int_{t>0}I_{t,\tau}\d t.\]
\end{lemma}

\begin{corollary}\label{C:P-MN}
Let $\pi$ be an infinite-dimensional irreducible unitarizable representation of $G=\Spin(m+1,1)$. Then
$\bar{\pi}|_{P}$ and $\bar{\pi}|_{MN}$ determine each other.
\end{corollary}

\begin{proof}
We have shown that $\bar{\pi}|_{P}$ is a finite direct sum of $I_{P,\tau}$. Then, the conclusion follows as the
spectra of $I_{P,\tau}|_{MN},I_{P,\tau'}|_{MN}$ are disjoint whenever $\tau\not\cong\tau'$.
\end{proof}

In \S\ref{SS:CW-Cloux}, we constructed a homomorphism \[\Psi\colon K(G)\rightarrow K(M').\] Write $\widehat{K}$
for the set of isomorphism classes of finite-dimensional irreducible representations of $K$ and write
$\mathbb{Z}^{\widehat{K}}$ for the abelian group of functions $\widehat{K}\rightarrow\mathbb{Z}$ with addition
given by point-wise addition. Taking the multiplicities of irreducible representations of $K$ appearing in
$\pi|_{K}$ ($\pi\in\mathcal{C}(G)$), we obtain a homomorphism \[m\colon K(G)\rightarrow\mathbb{Z}^{\widehat{K}}.\]
Write $\mathbb{Z}(K)$ for the quotient group of $\mathbb{Z}^{\widehat{K}}$ by the subgroup of functions
$f\colon \widehat{K}\rightarrow\mathbb{Z}$ such that $\sharp\{[\sigma]\in\widehat{K} : f([\sigma])\neq 0\}$ is finite.
Let \[p\colon\mathbb{Z}^{\widehat{K}}\rightarrow\mathbb{Z}(K)\] be the quotient map.

As in Section \ref{S:repP}, put $n=\lfloor\frac{m+2}{2}\rfloor$ and $n'=\lfloor\frac{m+1}{2}\rfloor$. Then,
\[n=\begin{cases} n'\quad\ &\textrm{ if }m\textrm{ is odd};\\
n'+1\quad &\textrm{ if }m\textrm{ is even}.
\end{cases}\]
The ranks of $K=\Spin(m+1)$, $M=\Spin(m)$, $M'=\Spin(m-1)$ are equal to $n'$, $n-1$, $n'-1$,
respectively. For a highest weight $\vec{b}=(b_1,\dots,b_{n'-1})$ of $M'$, write $V_{M',\vec{b}}$ for an irreducible
representation of $M'$ with highest weight $\vec{b}$. Then, $[V_{M',\vec{b}}]$ is a basis of $K(M')$.
Let \[\phi\colon K(M')\rightarrow\mathbb{Z}(K)\] be defined by \[\phi([V_{M',\vec{b}}])=\sum_{k\geq 0}[V_{K,(k+b_1,
b_1,\dots,b_{n'-2},(-1)^{m}b_{n'-1})}].\]

\begin{proposition}\label{P:P-K}
We have $\phi\circ\Psi=p\circ m$.
\end{proposition}

\begin{proof}
When $m$ is even, $K(G)$ is generated by induced representations $I(\sigma,\nu)$ and finite-dimensional
representations. Then, the conclusion follows from Proposition~\ref{P:res-induced} and branching laws for the
pair $M\subset K$ (giving $K$ types of induced representations).

When $m$ is odd,
 $K(G)$ is generated by induced representations $I(\sigma,\nu)$, (limits of) discrete
series and finite-dimensional representations. Then, the conclusion follows from Proposition~\ref{P:res-induced},
branching laws for the pair $M\subset K$, Theorem \ref{T:branching-ds} and Blattner's formula (giving
$K$ types of discrete series and limits of discrete series).
\end{proof}

\begin{corollary}\label{C:P-K}
For any $\pi\in\mathcal{C}(G)$, $\Psi(\pi)$ is determined by $\pi_{K}|_{K}$.
\end{corollary}

\begin{proof}
Note that $\Psi(\pi)$ is a finite direct sum of finite-dimensional irreducible unitary representations of $M'$.
Then, the conclusion follows from Proposition~\ref{P:P-K} directly.
\end{proof}

\begin{corollary}\label{C:MN-K}
Let $\pi$ be a unitarizable irreducible representation of $G$. Then $\bar{\pi}|_{MN}$ is determined by
$\bar{\pi}|_{K}$.
\end{corollary}

\begin{proof}
When $\pi$ is a unitarizable irreducible representation, $\bar{\pi}|_{K}$ and $\pi_{K}|_{K}$ determine each other.
By Corollary \ref{C:P-K}, $\Psi(\pi)$ is determined by $\pi_{K}|_{K}$. By Lemma \ref{L:unitary-J}, $\bar{\pi}|_{P}$
and $\Psi(\pi)$ determine each other. By Corollary \ref{C:P-MN}, $\bar{\pi}|_{P}$ and $\bar{\pi}|_{MN}$ determine each
other. Then, the conclusion of the corollary follows.
\end{proof}

\section{A case of Bessel model and relation with the local GGP conjecture}\label{S:GGP}

Take $G_{3}=\SO(m+1,1)$ and let $P_{3}=M_{3}AN$ be a standard minimal parabolic subgroup. Put $H_{3}=
M'_{3}\ltimes N$. Then, determining $\Hom_{H_{3}}(\pi,\tau\otimes e^{\mathbf{i}\xi_{0}})$ for irreducibles
$\pi\in\mathcal{C}(G_{3})$ and $\tau\in\mathcal{C}(M'_{3})$ is related to a case of Bessel model in the
local Gan-Gross-Prasad conjecture. Precisely to say, this is the case that $\dim W^{\perp}=3$ for the Bessel
model of $G_{3}=\SO(m+1,1)$ as described in \S 2 of \cite{Gan-Gross-Prasad}. Note that Bessel models were
studied by Gomez-Wallach in more general setting (\cite{Gomez-Wallach}).

First, define categories $\mathcal{C}(G_{3})$, $\mathcal{C}(P_{3})$, $\mathcal{C}(M_{3})$ similarly as that for
$\mathcal{C}(G)$, $\mathcal{C}(P)$, $\mathcal{C}(M)$ in \S \ref{S:repP}. For any $\pi\in\mathcal{C}(P_{3})$, as in
\S \ref{S:repP}, define $\Psi(\pi)=\pi/\mathfrak{m}_{\xi_0}\cdot\pi$. Then $\Psi(\pi)\in\mathcal{C}(M'_{3})$ and
$\Psi$ defines a functor $\mathcal{C}(P_{3})\rightarrow\mathcal{C}(M'_{3})$. Then, for any $\pi\in\mathcal{C}(G_{3})$
we have \[\Psi(\pi)\cong\bigoplus_{\tau\in\widehat{M'_{3}}}n_{\tau}\tau\] where $n_{\tau}(\pi)=
\dim\Hom_{H_{3}}(\pi,\tau\otimes e^{\mathbf{i}\xi_{0}})$. By an analogue of Proposition~\ref{P:res-induced} for the
pair $P_{3}\subset G_{3}$ and Casselman's subrepresentation theorem, it follows that: when $\pi\in\mathcal{C}(G_{3})$
is irreducible, $n_{\tau}(\pi)=0$ or 1 (the multiplicity one theorem) and there are only finitely many $\tau\in
\widehat{M'_{3}}$ such that $n_{\tau}(\pi)=1$. Moreover, $\Psi(\pi)$ are all calculated with the Langlands parameter
of $\pi$ by results in this paper. Hence, we know $n_{\tau}(\pi)=1$ for exactly which $\tau\in\mathcal{C}(M'_{3})$. The
multiplicity one theorem in this case was shown before in a very general setting (cf.\ \cite[\S 15]{Gan-Gross-Prasad}
and references therein). However, to the authors' knowledge determining $n_{\tau}(\pi)=1$ for which pairs $(\pi,\tau)$
in this case is new. This is related to the local Gan-Gross-Prasad conjecture in the Bessel model case
(\cite[Conjecture 17.1]{Gan-Gross-Prasad}, \cite[Conjecture 6.1]{Gan-Gross-Prasad2}).

Second, for any unitarizable irreducible infinite-dimensional representation $\pi\in\mathcal{C}(G_{3})$, an analogue
of Lemma \ref{L:unitary-J} indicates that $\bar{\pi}|_{P_{3}}\cong I_{P_{3}}(\Psi(\pi)\otimes e^{\mathbf{i}\xi_{0}})$.
That is to say, ``$L^{2}$ spectrum=smooth quotient" in this case.

\end{document}